\documentclass[a4paper, 11pt]{article}
 
\usepackage[T1]{fontenc}
\usepackage[utf8]{inputenc}
\usepackage{amsmath,amssymb, amsthm}
\usepackage{ifthen}
\usepackage{graphicx, psfrag}
\usepackage{float}
\usepackage{mathtools}
\usepackage{tikz}  
\usetikzlibrary{arrows,calc}
\tikzset{
	>=stealth',
	help lines/.style={dashed, thick},
	axis/.style={<->},
	important line/.style={thick},
	connection/.style={thick, dotted},
}
\usepackage{pgfplotstable}
\pgfplotsset{compat=1.5}
\usepackage{bm} 
\usepackage{bbm}
\usepackage{hyperref}
\usepackage{enumerate}
\usepackage{soul}
\usepackage{subcaption}
\usepackage{enumitem}
\setlist[enumerate]{leftmargin=.2in}
\setlist[itemize]{leftmargin=.2in}
\usepackage[margin=1in]{geometry}
\usepackage{doi}
\newtheorem{theorem}{Theorem}
\newtheorem{example}[theorem]{Example}
\newtheorem{assumption}{Assumption}
\newtheorem{lemma}[theorem]{Lemma}

\newtheorem{definition}[theorem]{Definition}
\newtheorem{remark}[theorem]{Remark}
\newtheorem{proposition}[theorem]{Proposition}

\def\revision#1{#1}  % {\textcolor{red}{#1}}
\def\C{{\mathbb C}}
\def\P{{\mathbb P}}
\def\R{{\mathbb R}}
\def\Sb{{\mathbb S}}

\def\E{{\mathbb E}}
\def\N{{\mathbb N}}

\def\U{{\mathbb U}}
\def\W{{\mathbb W}}

\def\bB{{\boldsymbol{B}}}
\def\bH{{\boldsymbol{H}}}
\def\bM{\boldsymbol{M}}
\def\bR{\boldsymbol{R}}
\def\bx{{\boldsymbol{x}}}
\def\by{{\boldsymbol{y}}}
\def\bz{{\boldsymbol{z}}}
\def\bv{{\boldsymbol{v}}}
\def\bw{{\boldsymbol{w}}}

\def\ba{{\boldsymbol{a}}}

\def\br{{\boldsymbol{r}}}

\def\bg{\boldsymbol{g}}
\def\bx{\boldsymbol{x}}
\def\bu{\boldsymbol{u}}
\def\bphi{\boldsymbol{\phi}}

\def\brho{\boldsymbol{\rho}}
\def\bmu{\boldsymbol{\mu}}
\def\bnu{{\boldsymbol{\nu}}}

\def\AA{{\mathcal A}}

\def\CC{{\mathcal C}}
\def\EE{{\mathcal E}}
\def\FF{{\mathcal F}}
\def\GG{{\mathcal G}}
\def\HH{{\mathcal H}}
\def\II{{\mathcal I}}

\def\LL{{\mathcal L}}
\def\MM{{\mathcal M}}

\def\OO{{\mathcal O}}
\def\PP{{\mathcal P}}

\def\RR{{\mathcal R}}

\def\TT{{\mathcal T}}
\def\UU{{\mathcal U}}

\def\XX{{\mathcal X}}
\def\YY{{\mathcal Y}}

\def\fD{{\mathfrak D}}
\def\fN{{\mathfrak N}}

\def\norm#1#2{\left\lVert#1\right\rVert_{#2}}
\def\seminorm#1#2{\left\vert #1\right\vert_{#2}}
\def\set#1#2{\left\{#1\,:\,#2\right\}}

\def\b#1{{\boldsymbol{#1}}}
\def\ol#1{\overline{#1}}
\newcommand{\dual}[1]{\left\langle #1\right\rangle}

\def\L2norm#1{\norm{#1}{L^2(D)}}

\def\H1norm#1{\norm{#1}{H^1(D)}}
\def\Re#1{\mathfrak{Re}\left( #1 \right)}
\def\Im#1{\mathfrak{Im}\left( #1 \right)}

\def\tC{\tilde{\CC}}
\def\tmu{\tilde{\mu}}
\def\tRR{\widetilde{\RR}}
\def\hbnu{\hat{\bnu}}

\def\eps{\varepsilon}
\def\rmd{{\rm d}}
\DeclareMathOperator\supp{\rm supp}
\def\forall{\textrm{\ for\ all\ }}
\DeclareMathOperator\erf{erf}
\def\rhol{\tilde{\rho}_{\ell}}

\def\m{\bm{m}}

\newcommand{\hC}{\hat{\mathcal{C}}}
\def\FFd{\FF\left\{0,1\right\}}
\def\sRR{\RR_s}
\def\normT#1{\norm{#1}{L^2_{\mu}(\XX_{\R}; \U)}}
\def\normMag#1{\norm{#1}{1+\alpha/2,2+\alpha}}
\def\normRes#1{\norm{#1}{\alpha/2,\alpha}}
\def\seminormMag#1{\seminorm{#1}{1+\alpha/2,2+\alpha}}
\def\normW#1{\norm{#1}{L^q([0,T])}}
\def\normV#1{\norm{#1}{L^q([0,T], C^{\alpha}(D))}}
\def\normM0#1{\norm{#1}{C^0([0,T], C^{\alpha}(D))}}
\title{Sparse grid approximation of nonlinear SPDEs: The Landau--Lifshitz--Gilbert equation
	\thanks{
		XA was fully supported by the Australian Research Council (ARC) grant DP190101197, JD was partially funded by ARCDP190101197, ARCDP220101811, MF and AS were funded by the Deutsche Forschungsgemeinschaft (DFG, German Research Foundation) -- Project-ID 258734477 -- SFB 1173 as well as the Austrian Science Fund (FWF) under the special research program Taming complexity in PDE systems (grant SFB F65) and project I6667-N. MF and AS also received funding from the European Research Council (ERC) under the	European Union's Horizon 2020 research and innovation programme (Grant agreement No. 101125225). TT was partially supported by ARCDP190101197, ARCDP200101866, and ARCDP220101811.
	}
}

\author{Xin An
\and
Josef Dick\footnotemark[2]
\and
Michael Feischl\footnotemark[3]
\and
Andrea Scaglioni\footnotemark[4]
\and
Thanh Tran\footnotemark[2]
}

\date{\today}
\begin{document}
\raggedright

\maketitle

\renewcommand{\thefootnote}{\fnsymbol{footnote}}
\footnotetext[2]{School of Mathematics \& Statistics, UNSW, Sydney, Australia}
\footnotetext[3]{Institute of Analysis and Scientific Computing, TU Wien, Vienna, Austria}
\footnotetext[4]{Department of Mathematics, University of Vienna, Vienna, Austria. Corresponding author.\\\href{mailto:andrea.scaglioni@univie.ac.at}{andrea.scaglioni@univie.ac.at} 
}
\renewcommand{\thefootnote}{\arabic{footnote}}

\begin{abstract}
    We show convergence rates for a sparse grid approximation of the distribution of solutions of the stochastic Landau-Lifshitz-Gilbert equation. Beyond being a frequently studied equation in engineering and physics, the stochastic Landau-Lifshitz-Gilbert equation poses many interesting challenges that do not appear simultaneously in previous works on uncertainty quantification: The equation is strongly non-linear, time-dependent, and has a non-convex side constraint. Moreover, the parametrization of the stochastic noise features countably many unbounded parameters and low regularity compared to other elliptic and parabolic problems studied in uncertainty quantification. We use a novel technique to establish uniform holomorphic regularity of the parameter-to-solution map based on a Gronwall-type estimate and the implicit function theorem. 
    This method is very general and based on a set of abstract assumptions. Thus, it can be applied beyond the Landau-Lifshitz-Gilbert equation as well.
    We demonstrate numerically the feasibility of approximating with sparse grid and show a clear advantage of a multilevel sparse grid scheme. 
\end{abstract}

\noindent{ \emph{Keywords:}
	stochastic Landau-Lifshitz-Gilbert equation, 
	curse of dimensionality, 
	sparse high-dimensional approximation, 
	sparse grid interpolation, 
	stochastic collocation,
	multilevel methods
}

\vspace{1em}

\noindent{ \emph{2020 Mathematics Subject Classification:}
	35R60, % Partial differential equations with randomness
	47H40, % Nonlinear operators and their properties, Random operators
	65C30, % Probabilistic methods, simulation and stochastic differential equations, Stochastic differential and integral equations
	60H25, % Stochastic analysis, Random operators and equations
	60H35, % Stochastic analysis, Computational methods for stochastic equations
	65M15 % Partial differential equations, initial value and time-dependent initial-boundary value problems, Error bounds
}

\section{Introduction}
While the methods developed in this work are fairly general and apply to different model problems, we focus on the specific task of approximating the stochastic Landau-Lifshitz-Gilbert equation as it contains many of the difficulties one encounters in nonlinear and stochastic partial differential equations.

The Landau-Lifshitz-Gilbert (LLG) equation is a phenomenological model for the dynamic evolution of the magnetization in ferromagnetic materials. In order to capture heat fluctuations of the magnetization one considers a stochastic extension of the LLG equation driven by stochastic noise, see e.g.,~\cite{brown1963thermal, kubo1970brownian} for some of the first works devoted to the modelling of magnetic materials under thermal agitation.
Following these early works, great interest in the physics community lead to extensive research, see e.g. e.g.,~\cite{berkov2007magnetization,garcia1998langevin,kohn2005magnetic, mayergoyz2009nonlinear,Scholz2001Thermal}.

This work gives a first efficient approximation of the probability distribution of the solution of the stochastic LLG equation. 
To that end, we employ the Doss-Sussmann transform and discretize the resulting Wiener process via a Lévy-Ciesielski expansion. 
This gives a parametrized nonlinear time-dependent PDE with infinite dimensional and unbounded parameter space which we approximate with sparse grid techniques. Our main results are:
\begin{itemize}
    \item The first rigorous convergence result on approximation of a nonlinear, time-dependent parametric coefficient PDE with unbounded parameter space. Precisely, we show convergence of piecewise quadratic sparse grid interpolation for the stochastic LLG equation with  order $1/2$ independent of the number of dimension (see Theorem~\ref{th:conv_SG_C}). The result assumes that the stochastic LLG equation has uniformly H\"older (in time and space) solutions. This is the case for regular, sufficiently close to constant, initial conditions. Under some reasonable assumptions and simplifications of the stochastic input, we show order $1/2$ dimension independent convergence (see Theorem~\ref{th:conv_SG_L1}).
    \item The first result on uniform holomorphic regularity of the parameter-to-solution map for the Landau-Lifshitz-Gilbert equation. To the best of our knowledge, this is also the first uniform holomorphic regularity result for unbounded parameter spaces and strongly nonlinear and time-dependent problems.
    \item Improved convergence rate of a multilevel version of the stochastic collocation algorithm under natural assumptions on the underlying finite element method.
\end{itemize}

In order to achieve these results, we overcome challenges posed by the nonlinear nature of the problem:
\begin{itemize}
   \item Holomorphic parameter-to-solution map: 
   	This is well-understood for linear problems but technically challenging for nonlinear problems. 
	While we apply the implicit function theorem as in \cite{chkifa2015breaking}, our parameter space is not compact.
    To overcome this problem, we control the growth of the extension by means of a Gronwall-like estimate for small imaginary parts.
    The main challenge here is that there is no canonical complex version of LLG which supports holomorphy. The main reason for this is that any extension of the cross product is either not complex differentiable or loses orthogonality properties which normally ensure $L^\infty$-boundedness of solutions of the LLG equation.
    
    \item Lack of parametric regularity: 
    The mentioned works on uncertainty quantification require strong summability of the coefficients which arise in the expansion of the stochastic noise.
    Typically, $\ell^p$-summability with $p<1$ is required. 
    Despite the holomorphic regularity, the present problem is only $\ell^p$-summable for $p>2$.  
    We propose a simplification of the stochastic input which leads to $L^1$-integrable sample paths in time.
    This increases the parametric regularity and allows for dimension independent estimates.

    \item  Lack of sample path regularity: Regularity results for LLG are sparse even in the deterministic setting. We refer to~\cite{carbour3,partreg2,cimExist,Lin,partreg1,Melcher12, partreg4} for partial results in 2D and 3D. Sample path regularity directly influences holomorphic regularity via the implicit function theorem. To that end, we rely on H\"older space regularity results for the stochastic LLG equation (Theorem~\ref{th:Holder_sample_paths}).
\end{itemize}

\subsection{Related work on the numerics of the LLG equation}
The nonlinear nature of LLG combined with the stochastic noise attracted a lot of interest in numerical analysis: 
For the deterministic version of LLG, weak convergence of some time stepping schemes was known since at least 2008 (see, e.g., the midpoint scheme~\cite{BartelsProhl:2006} and the tangent-plane scheme~\cite{Alouges:2008}). It took another ten years to obtain strong a~priori convergence of uniform time stepping schemes that obey physical energy bounds, which has first been proved  in~\cite{FeischlTran:2017_fembem} and was then extended to higher-order in~\cite{Akrivis2021Highorder}.
The latter two works build on the tangent plane idea first introduced in~\cite{Alouges:2008} in order to remove the nonlinear solver required in~\cite{BartelsProhl:2006}.
This is achieved by solving for the time derivative of the magnetization instead of the magnetization itself.
 
To study the stochastic version of LLG (SLLG),~\cite{BrzezniakGoldys2013,Brzezniak2013AnisoSLLG} formulate a rigorous definition of \emph{weak martingale solution} to the SLLG equation, prove existence using the Faedo-Galerkin method and discuss regularity even with anisotropy in the effective field and for finite multidimensional noise in space.
In \cite{Brzezniak2017Large,Brzezniak2019WongZakai}, the authors study the 1D (in space) SLLG equation, which has applications in the manufacturing of nanowires. They prove existence of weak martingale solutions for the problem for a larger class of coefficients compared to previous works in 3D. They also show pathwise existence and uniqueness of strong solutions and a large deviation principle. This is used to analyze the transitions between equilibria.
The space and time approximation of the SLLG equation was considered in \cite{BanasBrzezniak2014}. The authors consider an implicit midpoint scheme that preserves the unit modulus constraint on the magnetization and satisfies relevant discrete energy estimates. With a compactness argument, they prove that the method converges almost surely and weakly to the exact solution, up to extraction of a subsequence. In the follow-up work~\cite{BanasBrzezniak2013}, the scheme reproduces physically relevant phenomena such as finite-time blow-up of the solution and thermally-activated switching.
A different approach is followed in \cite{goldysLeTran}, where the authors propose to discretize SLLG in space and time by first applying the Doss-Sussmann transform~\cite{doss,sussman} to SLLG to obtain a random coefficient PDE. They then discretize this problem using the tangent-plane scheme~\cite{Alouges:2008}
and prove convergence (again in the sense of weak convergence of a subsequence), which in particular  proves that the random coefficient LLG equation is well posed. A tangent plane scheme is also considered in~\cite{alouge2014asemidiscrete}, where the sample paths of the SPDE in It\^o form are approximated and stability and convergence results are derived.
For the approximation of multi-dimensional (finite) noise,~\cite{Goldys202Multidimensional} generalizes the approach based on the Doss-Sussmann transform.

\subsection{Related work on the approximation of PDEs with random coefficients}
Dimension independent approximation of PDEs with random coefficients has first been proposed in~\cite{Cohen2010} and the idea of using a holomorphic extension of the exact solution in order to obtain convergence rates for the parametric approximation goes back to~\cite{Cohen2011Analytic}.

The works~\cite{Xiu2005High} and~\cite{Babuska2007Stocashtic} begin the mathematical study of collocation-type schemes for random coefficient PDEs. Several extensions and improvements of certain aspects of the theory can be found in, e.g.,~\cite{Nobile2008Sparse}, which uses sparse grid interpolation to improve the dependence on the number of parametric dimensions and~\cite{nobile2008anisotropic}, which employs \emph{anisotropic} sparse grid interpolation to achieve dimension independent convergence (under tractability assumptions on the problem).
In~\cite{nobile2016convergence}, the authors select the sparse grid with a profit-maximization principle, effectively recasting the sparse grid selection problem into a Knapsack problem. They also prove an error bound with explicit dependence on the number of approximated dimensions. In this work, we use the same principle to build sparse grid interpolation operators and apply the framework to prove dimension independent convergence.

In \cite{Zhang2012Error}, the authors extend the methodology developed in \cite{Babuska2007Stocashtic} to a linear parabolic problem with random coefficients under the finite-dimensional noise assumption. 
They prove existence of a holomorphic extension and, based on the ideas in~\cite{Babuska2007Stocashtic}, show that this leads to convergence of stochastic collocation schemes for both a space semi-discrete and fully discrete approximations.
In \cite{Nobile2009Parabolic}, the authors study a linear parabolic problem with uncertain diffusion coefficient under the finite dimensional noise assumption. They prove existence of a holomorphic extension by extending the problem to complex parameters and verifying the Cauchy-Riemann equations. They study convergence of stochastic Galerkin and stochastic collocation approximation. 
In \cite{Ganapathysubramanian2007}, the authors consider a coupled Navier-Stokes and heat equation problem with uncertainty and use a heuristic adaptive sparse grid scheme based on~\cite{GerstnerGriebel2003}.

In order to discretize the Wiener process in the stochastic LLG equation, one needs to deal with unbounded parameter spaces.
This has been done in, e.g.,~\cite{Bachmayr2017Sparse}, where the authors study the Poisson problem with lognormal diffusion and establish summability results for Hermite coefficients based on \emph{local-in-space} summability of the basis used to expand the logarithm of the diffusion.
In \cite{Ernst2018Convergence}, the authors approximate functions with this property by means of sparse grid interpolation built using global polynomials with Gauss-Hermite interpolation nodes. They prove algebraic and dimension independent convergence rates.

In the monograph~\cite{dung2023analyticity}, the authors study the regularity of  a large class of problems depending on Gaussian random field inputs as well as the convergence of several numerical schemes. Several examples of PDEs with Gaussian random coefficients are given e.g. elliptic and parabolic PDEs with lognormal diffusion.
The regularity result implies estimates on the Hermite coefficients of the parameter-to-solution map. These, in turn, can be used to study the convergence of Smolyak-Hermite interpolation and quadrature among other numerical methods.

Beyond linear problems, in~\cite{chkifa2015breaking} the authors deal with infinite-dimensional parametric problems with compact coefficient spaces, but possibly non-affine parametric dependence. They prove the existence of a holomorphic extension of the coefficient-to-solution map without extending the problem to the complex domain (as is usually done for the random Poisson problem). Rather, they employ the implicit function theorem.
In \cite{Cohen2018Stokes}, the authors use similar techniques in the setting of the stationary Navier-Stokes equation with random domain.

\revision{
There is a large body of literature on approximate \emph{quadrature} for SPDEs. Classical problems in this setting include approximating moments and probabilities of events. 
Semilinear parabolic SPDEs and their numerical approximation are treated e.g. in~\cite[Chapter 10]{Lord2014Introduction}.
As for the SLLG equation, Monte Carlo quadrature is the standard method used to approximate integrals, see e.g.~\cite{Banas2014Stochastic}.
The Multilevel Monte Carlo method is a popular and effective way to accelerate Monte Carlo quadrature of SPDEs (as well as a number of other high-dimensional approximation methods).
See~\cite{Giles2015Multilevel} for a comprehensive review and~\cite{Giles2012Stochastic, Belomestny2013multilevel} for applications in finance.
Approximate quadrature of random coefficient PDEs is also a widely studied topic~\cite{Cliffe2011, Teckentrup2013, Luo2019Multilevel} even with quasi-Monte Carlo methods~\cite{Herrmann2019Multilevel, Kuo2015Multi} and stochastic collocation~\cite{teckentrup2015multilevel}.
The convergence of Monte Carlo and multilevel Monte Carlo for weak error simulation of SPDEs is discussed in~\cite{Lang2018Monte}.
}

\subsection{Structure of the work} 
In Section \ref{sec:proofstrategy} we introduce a general framework for the study of the parametric regularity of solutions of SPDEs. We first explain in Section~\ref{sec:from_SPDE_to_paramPDE_general} how to reduce a SPDE to a parametric coefficients PDE. Then, in Sections~\ref{sec:param_regularity_general} and \ref{sec:uniform_extension_general} we prove that the parameter-to-solution map admits a sparse holomorphic extension. The result is based on four main assumptions that have to be proved for each concrete problem.  Finally, we estimate the derivatives of the parameter-to-solution map with Cauchy's integral theorem.\\
We introduce the stochastic version of the LLG equation in Section~\ref{sec:sllg}, and, following the general strategy from Section \ref{sec:from_SPDE_to_paramPDE_general}, transform it into a parametric nonlinear and time-dependent PDE in Section~\ref{sec:rnd_LLG_by_Doss_Sussmann}.
In the same section, we prove that the solution's sample paths are Hölder-continuous under regularity assumption on the problem data and uniformly bounded  with respect to the Wiener process sample paths.\\
In Section~\ref{sec:holomoprhic_regularity_Holder}, we apply the regularity analysis from Sections \ref{sec:param_regularity_general} and \ref{sec:uniform_extension_general} to the parametric LLG equation and prove that the parameter-to-solution map is holomorphic under the assumptions that sample paths of random coefficients and solutions are Hölder continuous.\\
In Section~\ref{sec:holomorphic_regularity_L1_small}, we do the same for a simplified version of the parametric LLG equation obtained with additional modelling assumptions. This time, sample paths are assumed to be Lebesgue integrable in time.\\
The sparsity properties of the parameter-to-solution map in the Hölder setting are weaker than in the Lebesgue setting. This is reflected by the convergence of sparse grid interpolation discussed in Section~\ref{sec:SG}.  The results are confirmed by numerical experiments.\\
The final Section~\ref{sec:ML} derives the multilevel version of the stochastic collocation method and provides numerical tests.

\section{General approach to deriving parametric regularity of a SPDE}\label{sec:proofstrategy}
In this section, we outline a fairly general strategy to prove a regularity property of solutions of stochastic partial differential equations (SPDE) driven by the Wiener process.
The resulting regularity properties can be used to tailor sparse grid approximation methods to the problem.
\revision{The problem formulation and arguments presented in this section are formal and need respectively to be rigorously defined and verified for each concrete problem}. The most important assumptions are listed explicitly below.

\subsection{Reduction to a parametric problem}\label{sec:from_SPDE_to_paramPDE_general}
Consider a spatial domain $D\subset \R^d$ of dimension $d\in\N$ and a final time $T>0$. Denote by $\partial D$ the boundary and by $\partial_n$ the unit exterior normal derivative. The space-time cylinder is denoted by $D_T\coloneqq [0,T]\times D$. 
Consider the initial condition $U^0:D\rightarrow \R^m$ for $m\in\N$, a drift coefficient $\fD : \R^m\times [0,T]\times D\rightarrow \R^m$ and a noise coefficient $\fN :\R^m\times D\rightarrow \R^m$. While a more general noise coefficient can be treated with analogous techniques, we consider this simple case as it is sufficient for the examples below. Given the probability space $(\Omega, \EE, \mathbb{P})$, we consider the SPDE problem:
Find a random field $U:\Omega\times D_T\rightarrow \R^m$ such that, $\P$-a.s.
\begin{align*}
\begin{cases}
\rmd U = \fD(U, t, \bx)\rmd t + \fN(U,\bx) \circ \rmd W(t)\qquad &\textrm{on } D_T\\
\partial_n U =0 \qquad &\textrm{on } [0,T] \times \partial D\\
U(\cdot, 0, \cdot) = U^0 \qquad &\textrm{on } D,
\end{cases}
\end{align*}
where by $\circ \rmd W(t)$ we denote the Stratonovich differential applied to a Wiener process $W$.

The \emph{Doss-Sussmann transform}~\cite{doss,sussman} of $U$ is, by definition,
\begin{align}\label{eq:def_Doss_Sussmann}
u = e^{-W \fN} U,
\end{align}
i.e. the exponential of the operator $-W \fN$ applied to $U$. The resulting random field $u:\Omega\times[0,T]\times D\rightarrow \R^m$ solves a \emph{random coefficient partial differential equation} (PDE):
\begin{align}\label{eq:rndPDE_general}
\RR(W(\omega),u(\omega)) = 0\qquad \textrm{in } R,\ \P\textrm{-a.e. } \omega\in \Omega.
\end{align}
The \emph{residual operator} $\RR: \W_{\R}\times \U_{\R}\rightarrow R$ is defined for Banach spaces $\W_{\R}, \U_{\R}$ and $R$.
In general, it is a differential operator in time and space with respect to $u\in\U_{\R}$ while it does not contain It\^o or Stratonovich differentials of $W$.

In order to make the distribution of $u$ amenable to \emph{approximation}, we need to parametrize the Brownian motion. It turns out that a local wavelet-type expansion of $W$ is very beneficial as it reduces the number of active basis function at any given moment in time \revision{(used e.g. to prove the $L^q(0,T)$-summability of appropriate complex extensions of the Wiener process sample paths; see the proof of Lemma~\ref{lemma:Lp_paths_W} below).}
The Lévy-Ciesielski expansion (LCE) (see e.g. \cite[Section 4.2]{griebel2017anova}) of the Brownian motion $W:\Omega\times [0,1] \rightarrow \R$ reads
\begin{align}\label{eq:def_LCE}
	W(\omega, t) 
	= \sum_{\ell=0}^{\infty} \sum_{j=1}^{\lceil 2^{\ell-1}\rceil} Y_{\ell,j}(\omega)\eta_{\ell,j}(t),
\end{align}
where $Y_{\ell,j}$ are independent standard normal random variables and \\ 
$\set{\eta_{\ell,j}}{ \ell\in\N_0, j=1,\dots, \lceil2^{\ell-1}\rceil}$ is the \emph{Faber-Schauder} hat-function basis on $[0,1]$, i.e., 
\begin{align}\label{eq:faber-shauder_basis}
\begin{split}
& \eta_{0,1} (t) = t,\\
& \eta_{\ell,j} (t) = 2^{-\frac{\ell-1}{2}} \eta\left(2^{\ell-1}t - j + 1\right) \qquad \forall \ell\in\N,\ j = 1,\dots, 2^{\ell-1},
\end{split}
\end{align}
where $\eta(t) \coloneqq 
\begin{cases}
	t\qquad t\in[0,\frac{1}{2}]\\
	1-t\qquad t\in[\frac{1}{2},1]\\
	0\qquad {\rm otherwise},
\end{cases}.$
Observe that $\norm{\eta_{0,1}}{L^{\infty}(0,1)} = 1$, $\supp\eta_{0,1}=(0,1]$ and  
$\norm{\eta_{\ell,j}}{L^{\infty}(0,1)} = 2^{-(\ell+1)/2}$,
$\supp{\eta_{\ell,j}} = \left(\frac{j-1}{2^{\ell-1}},\frac{j}{2^{\ell-1}}\right)$
for all $ \ell\in\N, j = 1,\dots, 2^{\ell-1}$.
The LCE converges uniformly in $t$, almost surely to a continuous function which coincides with the Brownian motion everywhere in $[0,1]$ (see \cite[Section 3.4]{steele2001Stochastic}).
\revision{We consider a parametric version of the random field $W$ in the form }
$W:\R^{\N}\times [0,1]\rightarrow \R$ so that
\begin{align}\label{eq:def_param_LC}
W(\bm{y}, t) = \sum_{\ell=0}^{\infty} \sum_{j=1}^{\lceil 2^{\ell-1}\rceil} y_{\ell,j} \eta_{\ell,j}(t),
\end{align}
where $y_{\ell,j}\in \R$ for all $\ell\in \N_0, j = 1,\dots,\lceil2^{\ell-1}\rceil$. 
For $L\in\N_0$, we define the \emph{level-$L$ truncation of $W$} by 
$ W_L(\by, t) = \sum_{\ell=0}^{L} \sum_{j=1}^{\lceil 2^{\ell-1}\rceil} y_{\ell,j}\eta_{\ell,j}(t). $
We will sometimes also index the same sum as $W_L(\by, t) = \sum_{n=0}^N  y_n\eta_n(t)$. The two indexing systems, hierarchical and linear, are related via
\begin{align}\label{eq:re-indexing_LC}
\eta_{\ell, j} = \eta_n \quad  \Longleftrightarrow \quad n = \lfloor 2^{\ell-1} \rfloor + j-1.
\end{align}
We note that the total number of parameters is $N = \sum_{\ell=0}^L \lceil 2^{\ell-1}\rceil = 1 + \sum_{\ell=1}^{L} 2^{\ell-1} = 2^L$.

The fact that the parameter domain is unbounded requires the use of appropriate collocation nodes, a topic we treat in Section \ref{sec:SG}, below.

We denote by $\XX_{\R}$ an appropriate \revision{separable Banach} space of real sequences such that if $\by\in\XX_{\R}$, then $W(\by, \cdot)$ belongs to a desired Banach space \revision{$\W$} of functions. 
\revision{
The Banach space $\XX_\R$ is assumed to be separable in order for $u:\XX_\R \rightarrow \U$ to be separably valued (i.e. its image $u(\XX_\R)$ be separable) under the mild regularity assumption that $u$ is continuous.
As a consequence of Pettis measurably theorem, the parameter-to-solution map is also measurable.
Measurability is a naturally important property because it is necessary for the well posedness of integral quantities such as the moments of the random field.
}

\begin{example} 
Consider the Banach space of sequences:
\begin{align*}
    \XX_{\R}\coloneqq \set{\by = (y_n)_{n\in\N} \in \R^{\N}}{\norm{\by}{\XX_{\R}} < \infty},\ \norm{\by}{\XX_{\R}} \coloneqq \seminorm{y_{0,1}}{} + \sum_{\ell \in \N} \max_{j=1,\dots, 2^{\ell-1}} \seminorm{y_{\ell, j} }{} 2^{-(\ell+1)/2}.
\end{align*}
Simple computations show that if $\by\in\XX_{\R}$, then $\norm{W(\by)}{L^\infty(0,T)}\leq \norm{\by}{\XX_{\R}}$, thus $\W \subset L^\infty(0,T)$.
\end{example}

Assume without loss of generality that $T=1$. By substituting the random field $W(\omega, t)$ in the random coefficient PDE \eqref{eq:rndPDE_general} with the parametric expansion \eqref{eq:def_param_LC}, we obtain a \emph{parametric coefficient PDE}:
Find $u : \XX_{\R}\times D_T\rightarrow \R^m$ such that
\begin{align}\label{eq:parametric_problem}
\RR(W(\by), u(\by)) = 0\qquad \textrm{in } R, \ \bmu\textrm{-a.e. } \by\in\XX_{\R},
\end{align}
\revision{where $\bmu$ denotes the standard Gaussian measure on $\R^{\N}$, i.e. the product measure $\bmu \coloneqq \bigotimes_{n\in\N} \mu_n$, where $(\mu_n)_{n\in\N}$ is a sequence of standard Gaussian probability measures on $\R$ (see, e.g.,~\cite{Kakutani1948On} and~\cite[Section 2.4]{dung2023analyticity} for details on infinite product measures).}

\subsection{Holomorphic regularity of the solution operator} \label{sec:param_regularity_general}
While holomorphic parameter regularity of random elliptic equations is well-known by now (see, e.g.,
\cite[Section 3]{Babuska2007Stocashtic}, for the case of bounded or unbounded parameter spaces under the finite dimensional noise assumption, \cite{Cohen2011Analytic}, for countably-many parameters taking values on tensor product of bounded intervals, \cite{Bachmayr2017Sparse}, for a discussion of the Poisson problem with lognormal coefficients, in which the authors study countably many unbounded parameters), the literature is much sparser for nonlinear and time-dependent problems. In this section, we follow an approach from~\cite{chkifa2015breaking} which uses the implicit function theorem to obtain analyticity. While the authors in~\cite{chkifa2015breaking} can rely on a compact parameter domain to ensure a non-trivial domain of extension, we have to use intricate bounds on the parametric gradient of the solution. A recent result on the implicit function theorem for Gevrey regularity~\cite{schwabGevrey} could also be used to achieve similar results in a less explicit fashion.

We require some assumptions to work in a  more general setting.
\begin{assumption} \label{assume:uniform_bounded_solution}
For any $\by\in\XX_{\R}$ there exists $u(\by)\in \U_{\R}$ such that $\RR(W(\by), u(\by))=0$ in $R$. Moreover, there exists $C_r>0$ such that, for any $\by\in\XX_{\R}$, $\norm{u(\by)}{\U_{\R}} \leq C_r $.
\end{assumption}
\begin{assumption}\label{assume:assumptions_IFTh}
The residual operator $\RR:\W_{\R}\times\U_{\R}\to R$ admits an extension to \emph{complex} Banach spaces $\W \supset \W_{\R}$ and $\U\supset \U_{\R}$. The extended map $\RR: \W \times \U \rightarrow R$ satisfies the following properties:
\begin{enumerate}%[i.]
\renewcommand{\labelenumi}{\theenumi}
\renewcommand{\theenumi}{{\rm (\roman{enumi})}}
\item $\RR$ is \revision{Fréchet} continuously differentiable;
\item $\partial_u \RR(W, u): \U\rightarrow R$ is a homeomorphism for all $(W,u)\in\W_{\R}\times\U_{\R}$ so that $\RR(W,u)=0$.
\end{enumerate}
\end{assumption}
With this complex extension in mind, in the following for any $W_0\in\W_{\R}$ and $u_0\in\U_{\R}$ we denote, for $\varrho>0$,
\begin{equation}
\label{eq:ball def}
\begin{aligned}
    B_{\varrho}(W_0) &:= \big\{W\in\W : \norm{W-W_0}{\W} < \varrho \big\},
    \\
    B_{\varrho}(u_0) &:= \big\{u\in\U : \norm{u-u_0}{\U} < \varrho \big\}.
\end{aligned}
\end{equation}

We recall the implicit function theorem for maps between Banach spaces (see, e.g.,~\cite[Theorem 10.2.1]{dieudonne1960treatise}).
\begin{theorem}[Implicit function]\label{th:implicit_fun_th}
Let $E,F,G$ be Banach spaces, 
$A\subset E\times F$
and $f:A\rightarrow G$ be a \revision{Fréchet} continuously differentiable function.
Let $(x_*, y_*)\in A$ be such that $f(x_*, y_*)=0$
and the partial derivative $D_2f(x_*, y_*)$ is a linear homeomorphism from $F$ onto $G$.
Then, there exists a neighborhood $U_*$ of $x_*$ in $E$ such that, for every open connected neighborhood $U$ of $x_*$ in $U_*$, there exists a unique continuous mapping $\UU:U\rightarrow F$ such that 
$\UU(x_*) = y_*$,
$(x, \UU(x))\in A$ and $f(x, \UU(x)) = 0$ for any $x$ in $U$.
Moreover, $\UU$ is continuously differentiable in $U$ and its derivative is given by
\begin{align}\label{eq:formula_dU_ifth}
\UU'(x) = - \left(D_2f(x, \UU(x))\right)^{-1}\circ \left(D_1f(x, \UU(x))\right) \quad \forall x\in U.H
\end{align}
\end{theorem}
Invoking Theorem~\ref{th:implicit_fun_th} for the operator $\RR : \W \times \U \rightarrow R$, with $\by\in\XX_{\R}$ and $u(\by)\in\U_{\R}$ satisfying $\RR(W(\by),u(\by))=0$, there exists $\eps(\by)>0$ and a holomorphic map $\UU:B_{\eps(\by)}(W(\by))\rightarrow \U$ such that
$\UU(W(\by)) = u(\by)$ and $\RR(W, \UU(W))=0$ for all $W\in B_{\eps(\by)}(W(\by))$.

For any $W\in B_{\eps(\by)}(W(\by))$, the differential $\UU'(W)$ belongs to $\LL(\W,\U)$, the set of linear bounded operator from $\W$ into $\U$ equipped with the usual norm.

Recalling definition~\eqref{eq:ball def}, we make additional assumptions on the regularity of the derivatives of the residual operator $\RR$.
\begin{assumption}\label{assume:bounds_d1R_d2Rinv}
There exist $\eps_W,\eps_u>0$ such that for any $\by\in\XX_{\R}$ and any $W\in B_{\eps_W}(W(\by))$ with $\UU(W)\in B_{\eps_u}(\UU(W(\by)))$, the operator $\partial_W \RR(W, \UU(W))$ is well-defined and $\partial_u \RR(W, \UU(W))$ is homeomorphic with
\begin{align*}
\norm{\partial_W \RR(W, \UU(W))}{\LL(\W, R)} \leq\GG_1(\norm{\UU(W)}{\U}),\\
\norm{\partial_u \RR(W, \UU(W))^{-1}}{\LL(R, \U)} \leq \GG_2(\norm{\UU(W)}{\U}),
\end{align*}
where the functions $\GG_1, \GG_2$ are continuous and may depend on problem coefficients and $\eps_u$, $\eps_W$ but depend on $W$ and $\UU(W)$ only through $\norm{\UU(W)}{\U}$ and are independent of $\by$.
\end{assumption}
Together with \eqref{eq:formula_dU_ifth} from Theorem \ref{th:implicit_fun_th}, this assumption implies the existence of a continuous increasing function ${\GG=\GG(\norm{\UU(W)}{\U})>0}$ such that
\begin{align}\label{eq:estim_norm_DU}
\norm{\UU'(W)}{\LL(\W,\U)} \leq \GG(\norm{\UU(W)}{\U})\qquad \forall W \in B_{\min(\eps(\by), \eps_W)}(W(\by)).
\end{align}

\subsection{Uniform holomorphic extension of solution operator}\label{sec:uniform_extension_general}
Since we cannot rely on a compact parameter domain, we show existence of a uniformly bounded holomorphic extension through the application of a generalized version of Gronwall's lemma.

Fix $\by\in\XX_{\R}$. 
We can assume, without loss of generality, that $\eps(\by) \leq \eps_W$.
\begin{definition}\label{def:H} 
We consider an open set $\HH(\by)\subseteq B_{\eps_W}(W(\by))$ with the following properties:
\begin{itemize}
    \item $B_{\eps(\by)}(W(\by))\subseteq \HH(\by)$, 
    \item $\UU(W)\in B_{\eps_u}(\UU(W(\by)))$ for all $W\in \HH(\by)$,
    \item the solution operator $\UU : B_{\eps(\by)}(W(\by))\to\U$ extends holomorphically to $\HH(\by)$,    
    \item for all $W\in \HH(\by)$ we have $\sigma W +(1-\sigma) W(\by)\in\HH(\by)$ for all $0\leq \sigma\leq 1$.
\end{itemize}
\end{definition}

In contrast to \cite{chkifa2015breaking}, this domain of real parameters $\W_{\R}$ may not be compact.
Therefore, $\eps(\by)$ can be arbitrarily small and hence $\HH(\by)$ might become very small for certain parameters $\by$. The goal of the arguments below is to show that there exists $\eps>0$ such that for all $\by\in\XX_\R$ $\HH(\by)=B_\eps(W(\by))$ is a valid choice.
Instead of relying on compactness, we exploit estimate~\eqref{eq:estim_norm_DU} through the following nonlinear generalization of Gronwall's lemma:

\begin{lemma}[\cite{dragomir2003applications}, Theorem 27]\label{lem:Gronwall_nonlin}
Let $0\leq c\leq d<\infty$, $\varphi:[c,d]\rightarrow \R$ and $k:[c,d]\rightarrow \R$ be positive continuous functions on $[c,d]$ and let $a,b$ be non-negative constants.
Further, let $\GG:[0,\infty)\rightarrow \R$ be a positive non-decreasing function. If
\begin{align*}
	\varphi(t) \leq a+ b \int_c^t k(s) \GG(\varphi(s)) \rmd s\quad \forall t\in [c,d],
\end{align*}
then
\begin{align*}
	\varphi(t) \leq G^{-1}\left(G(a) + b \int_c^t k(s)\rmd s\right)\quad \forall c\leq t\leq d_1\leq d
\end{align*}
where \revision{$G$ is defined, with some fixed $\xi>0$, by}
\begin{align}\label{eq:def_G}
    \revision{G}(\lambda) \coloneqq \int_{\xi}^\lambda \frac{\rmd s}{\GG(s)}\qquad \revision{\forall \lambda > \xi} 
\end{align}
and $d_1 $ is defined such that $G(a)+b\int_c^tk(s)\rmd s$ belongs to the domain of $G^{-1}$ for $t\in[c, d_1]$.
\end{lemma}

\begin{theorem}\label{th:global_holo_ext}
Assume the validity of Assumptions \ref{assume:uniform_bounded_solution}, \ref{assume:assumptions_IFTh}, and~\ref{assume:bounds_d1R_d2Rinv}. With $C_r>0$ given in Assumption \ref{assume:uniform_bounded_solution}, choose $0<\eps < \eps_W$ such that $G\left(C_r\right) + \eps$ belongs to the domain of $G^{-1}$ (where $G$ is defined in \eqref{eq:def_G} with the corresponding $\GG$ given in \eqref{eq:estim_norm_DU}).
Then, $\eps$ is independent of $\by$ and $\HH(\by)$ from Definition~\ref{def:H} can be chosen as $\HH(\by)= B_\eps(W(\by))$ for all $\by\in\XX_{\R}$. Moreover, $\UU$ is uniformly bounded on $B_{\eps}(W(\by))$ by a constant $C_{\eps}>0$ that depends only on $\eps$.
\end{theorem}

\begin{proof}
Fix $\by\in\XX_{\R}$.

{\em Step~1}: We first show that $\UU$ is uniformly bounded on $\HH(\by)\cap B_{\eps}(W(\by))$.
To that end, fix $W\in \HH(\by)\cap B_{\eps}(W(\by))$ and let $W_{\sigma} \coloneqq \sigma W  + (1-\sigma) W(\by)$ for any $0\leq \sigma \leq 1$. We define $\varphi:[0,1]\rightarrow \U$ by $\varphi(\sigma) = \UU(W_\sigma)$.
Since by definition $\UU$ is differentiable in $\HH(\by)$, we may apply the fundamental theorem of calculus to obtain
\begin{equation}
\label{eq:varphi}
\varphi(t)
-
\varphi(s)
%=
%\int_s^t \varphi'(\sigma)\rmd\sigma
=
\int_s^{t} \UU'(W_{\sigma})[W-W(\by)]\rmd \sigma \quad\text{for all $s,t\in[0,1]$.}
\end{equation}
%\begin{align*}
%\UU(W_{\sigma}) = \UU(W(\by))  + \int_0^{\sigma} \UU'(W_{s})[W-W(\by)]\rmd s.
%\end{align*}
In particular, with $s=0$, the triangle inequality yields, recalling that $W\in B_\eps(W(\by))$,
\begin{align*}
\norm{\varphi(t)}{\U} \leq \norm{\varphi(0)}{\U} + \eps\int_0^{t} \norm{\UU'(W_{\sigma})}{\LL(\W,\U)}\rmd \sigma
%\norm{\UU(W_{\sigma}))}{\U} \leq \norm{\UU(W(\by))}{\U} + \eps\int_0^{\sigma} \norm{\UU'(W_{s})}{\LL(\W,\U)}\rmd s
\quad \text{for all }0\leq t\leq 1.
\end{align*}
Assumption \ref{assume:uniform_bounded_solution} and estimate \eqref{eq:estim_norm_DU} (consequence of Assumption \ref{assume:bounds_d1R_d2Rinv}) imply the estimate
\begin{align*}
\norm{\varphi(t)}{\U} \leq C_r + \eps \int_0^t \GG\left( \norm{\varphi(\sigma)}{\U} \right)\rmd \sigma\quad \text{for all }0\leq t \leq 1.
\end{align*}
Apply Lemma \ref{lem:Gronwall_nonlin} to conclude (note that, in the notation of Lemma \ref{lem:Gronwall_nonlin}, we have $d_1=d=1$ because of the definition of $\eps$ as well as $k(s)=1$)
\begin{equation}
\label{eq:varphi bounded}
\norm{\varphi(t)}{\U} \leq G^{-1}(G(C_r)+\eps t) \le G^{-1}(G(C_r)+\eps) 
\quad\text{for all $0\le t\le1.$}
\end{equation}
Since $\norm{\UU(W)}{\U} = \norm{\varphi(1)}{\U} \le C_{\eps}$, where $C_{\eps}:=G^{-1}(G(C_r)+\eps)$, we derive the uniform boundedness of~$\UU$ on $\HH(\by)$.  
Note that this bound is independent of $\by$ and $\HH(\by)$.

{\em Step~2}: We next show that $\varphi$ defined in Step~1 is Lipschitz on $[0,1]$.
Equation~\eqref{eq:varphi} implies, for $0 \le s < t \le 1$,
\begin{align*}
\norm{\varphi(t) - \varphi(s)}{\U} &\leq \int_s^t \norm{\UU'(W_{\sigma})}{\LL(\W, \U)} \norm{W-W(\by)}{\W} \rmd \sigma\\
&\leq \int_s^t \GG\left( \norm{\varphi(\sigma)}{\U} \right)\norm{W-W(\by)}{\W} \rmd \sigma.
\end{align*}
The desired results then follow from~\eqref{eq:varphi bounded}. 

{\em Step~3:} 
We can without loss of generality assume that 
 $0<\eps\leq \eps_W$ is such that $W\in B_{2\eps}(W(\by))$ implies $\UU(W)\in B_{\eps_u}(\UU(W(\by))$. This is possible due to the Lipschitz continuity of $\varphi$ proved in the previous step and by possibly making the $\eps$ chosen in Step~1 smaller.
We now show that $\HH(\by)$ can be chosen to be $B_\epsilon(W(\by))$.
 Assume by contradiction that the \emph{maximal} $\HH(\by)$ (in the sense as there is no superset of $\HH(\by)$ with the properties specified in Definition~\ref{def:H}) is a \emph{proper} subset of $B_{\eps}(W(\by))$, i.e. $\HH(\by)\subsetneq B_{\eps}(W(\by))$. Let $W\in\partial\HH(\by) \cap B_{\eps}(W(\by)) \neq \emptyset$.
Lipschitz continuity of $\varphi$ in Step~2 shows that $\UU$ can be extended continuously to $\overline{\HH(\by)}$.
Consequently, $\UU(W)$ is well-defined and equals $\lim_{\sigma\rightarrow 1^-}\UU(\sigma W+(1-\sigma)W(\by))\in \U$.
Since $\RR$ is continuous, $\RR(W, \UU(W)) = 0$.
By Assumption~\ref{assume:bounds_d1R_d2Rinv}, $\partial_u \RR(\ol{W}, \UU(\ol{W}))$ is a homeomorphism for any $\ol{W}$ in a neighborhood of $W$ in $\W$.
We may therefore apply the implicit function theorem in $W$ to show that the domain of existence of a holomorphic extension of $\UU$ can be further extended to an open neighborhood $B\supsetneq \HH(\by)$ of $W$ in $\W$. Clearly, the neighborhood can be chosen such that $\UU(\ol{W})\in B_{\eps_u}(\UU(W(\by)))$ for all $\ol{W}\in B$.
Since $B$ can be chosen star shaped with respect to $W(\by)$, this contradicts the maximality of $\HH(\by)$.
Thus, we proved that $B_{\eps}(W(\by))= \HH(\by)$. 
The argument used in Step~1 immediately implies the uniform boundedness.
\end{proof}

Theorem~\ref{th:global_holo_ext} provides all the tools to estimate parametric regularity through Cauchy's integral theorem.
The Lévy-Ciesielski expansion \eqref{eq:def_param_LC} can be (formally) extended to the complex parameters $\bz\in \C^{\N}$.
Thus, in view of Theorem \ref{th:global_holo_ext}, $\bz\mapsto \UU(W(\bz))$ is a holomorphic extension of the parameter-to-solution map in $\by$ for all $\bz$ such that $W(\bz)$ belongs to the domain of holomorphy of $\UU$, which in Theorem \ref{th:global_holo_ext} was proved to contain $ B_{\eps}(W(\by))$ (recall that $\eps$ is independent of $\by$).
Such a set of parameters can be defined as follows: Let $\brho=(\rho_n)_{n\in\N}$ be a sequence of non-negative real numbers, and consider the polydisk
\begin{align}\label{eq:def_domain_param_holo}
\bB_{\brho}(\by)\coloneqq \set{\bz\in\XX}{ \seminorm{z_n-y_n}{} < \rho_n\ \forall n\in\N}.
\end{align} 
\begin{assumption}\label{assume:sparsity}
For $\eps>0$, $\by\in\XX_{\R}$, there exists a real positive sequence $\brho = \brho(\eps) = (\rho_n)_{n\in\N}$ such that,
\begin{align*}
\bz \in \bB_{\brho}(\by) \Rightarrow W(\bz)\in B_{\eps}(W(\by)),
\end{align*}
\end{assumption}
In conclusion, for any $\by\in\XX_\R$, $\UU\circ W: \bB_{\brho}(\by)\rightarrow \U$ is holomorphic because it is a composition of holomorphic functions. Moreover, $\UU\circ W$ is uniformly bounded by $C_{\eps}$. % (see Theorem \ref{th:global_holo_ext}) independently of $\by$.

Consider a multi-index $\bnu=(\nu_1,\dots,\nu_n)\in \N_0^n$ and denote by $\partial^{\bnu}$ the mixed derivative $\partial_1^{\nu_1}\dots \partial_n^{\nu_n}$ where $\partial_j^{\nu_j}$ denotes the partial derivative of order $\nu_j$ with respect to $y_j$ (if $\nu_j=0$, the $j$-th partial derivative is omitted).
The regularity result proved above implies the following estimate on the derivatives of the parameter-to-solution map:
\begin{theorem} \label{th:bound_mix_derivative}
	Consider $u:\XX_{\R}\rightarrow \U$, the parameter-to-solution map that solves the parametric PDE \eqref{eq:parametric_problem}.
	Let Assumptions \ref{assume:uniform_bounded_solution}, \ref{assume:assumptions_IFTh}, \ref{assume:bounds_d1R_d2Rinv} hold and fix $\eps>0$ as in Theorem \ref{th:global_holo_ext}. 
	Finally, consider a real positive sequence $\brho = (\rho_n)_{n\in\N}$ as in Assumption \ref{assume:sparsity}.
	Then, for any $n\in\N$, $\bnu = \left(\nu_i\right)_{i=1}^n \in \N^n_0$, it holds that
	\begin{align}\label{eq:bound_mix_derivatives}
		\norm{\partial^{\bnu} u(\by)}{\U} 
		\leq \prod_{j=1}^n \nu_j!\rho_j^{-\nu_j} C_{\eps}\qquad \forall \by\in\XX_{\R},
	\end{align}
	where $C_{\eps}>$ 0 from Theorem \ref{th:global_holo_ext} is independent of $\bnu$ or $\by$.
	The same bound holds for $\norm{\partial^{\bnu} u}{L^2_{\bmu}(\XX_{\R}; \U)} $ (up to a constant), where $\bmu$ denotes a probability measure on $\XX_{\R}$.
\end{theorem}
\revision{
\begin{proof}
	Apply Cauchy's formula~\cite[Theorem 2.1.2]{Herve1989Analyticity} to each of the $n$ variables $y_1,\dots,y_n$ recursively and then differentiate.
\end{proof}
}

Note that Theorem \ref{th:bound_mix_derivative} gives the crucial bound on the derivatives that justifies many high-dimensional approximation methods e.g. sparse grids, polynomial chaos, quasi-Monte Carlo.

\section{The Stochastic Landau--Lifshitz--Gilbert equation}\label{sec:sllg}
In this section, we introduce the stochastic Landau-Lifshitz-Gilbert equation and we show that it fits the general theory described in the previous section.
Consider a bounded Lipschitz domain $D\subset \R^3$ representing a ferromagnetic body in the time interval $[0,T]$. $D_T\coloneqq [0,T] \times D$ denotes the space-time cylinder and $\partial_n$ the outward pointing normal derivative on $\partial D$. Given $\bM^0:D\rightarrow \Sb^2\coloneqq\set{\bx\in\R^3}{x_1^2+x_2^2+x_3^2=1}$ (the magnetization of the magnetic body at initial time), $\lambda>0$ (called the \emph{Gilbert damping parameter}), the deterministic version of the problem (the \emph{LLG equation}) consists of determining the dynamics of the magnetization:\\
Find $\bM:D_T\rightarrow \Sb^2$ such that
\begin{align}\label{eq:LLG}
\begin{cases}
\partial_t \bM &= \lambda_1\bM\times \Delta \bM - \lambda_2\bM\times \left(\bM\times \Delta \bM\right) \quad {\rm in}\ D_T,\\
\partial_n \bM &= \b{0} \quad {\rm on}\ \partial D\times[0,T],\\
\bM(0) &= \bM^0 \quad {\rm on} \ D,
\end{cases}
\end{align}
where $\lambda_1 = \frac{1}{1+\lambda^2}$, $\lambda_2=\frac{\lambda}{1+\lambda^2}$. The solution has constant magnitude in space and time (this follows immediately from scalar multiplication of~\eqref{eq:LLG} with $\bM$). This implies that, assuming a normalized initial condition $|\bM^0|\equiv 1$ on $D$, that
\begin{align*}
|\bM(t,\bx)|=1\qquad \text{for all }(t,\bx)\in D_T.
\end{align*}
In \eqref{eq:LLG}, the \emph{exchange term} $\Delta \bM$ can be substituted by a more general \emph{effective field} $\bH_{\rm eff}(\bM)$ containing $\Delta \bM$ and additional lower order contributions modelling additional physical effects like material anisotropy, magnetostatic energy, external magnetic fields or the more involved Dzyaloshinskii-Moriya interaction (DMI) (see e.g. \cite[Section 1.2]{PfeilerCarl-Martin2022Naae}).

The effect of heat fluctuations on the systems is described with a random model. Denote by $\left(\Omega, \EE, \P\right) $ a probability space and let $\rmd \bm{W}:\Omega \times D_T\rightarrow \R^3$ be a suitable space-time noise (note that the exact form of this noise is subject of research and below we consider a simple one-dimensional model). Consider the following formal equation for $\bM : \Omega \times D_T \rightarrow \Sb^2$:
\begin{align*}
\partial_t \bM = \lambda_1 \bM\times \left(\Delta \bM + \rmd \bm{W} \right) - \lambda_2\bM\times \left(\bM\times \Delta \bM \right) &\quad {\rm in}\ D_T,\P\textrm{-a.s.}
\end{align*}
with the same initial and boundary conditions as in \eqref{eq:LLG}.
It is customary not to include a noise in the second term of the right-hand side because of the smallness of $\lambda_2$ compared to $\lambda_1$ (see, e.g., \cite[page 3]{BrzezniakGoldys2013}).
For simplicity, we additionally assume one-dimensional noise $\bm{W}(\omega, t, \bx) = \bg(\bx) W(\omega, t)$ for all $\omega\in\Omega, (t,\bx)\in D_T$, where $\bg:D\rightarrow \R^3$ is given and $W:\Omega\times [0,T] \rightarrow \R$ denotes a (scalar) Wiener process.

The previous formal equation corresponds to the following \emph{stochastic} partial differential equation called the \emph{stochastic LLG equation}: Find $\bM: \Omega \times D_T \rightarrow \Sb^2$ such that
\begin{align}\label{eq:SLLG}
\rmd \bM =  \left( \lambda_1\bM\times \Delta \bM - \lambda_2\bM\times \left(\bM\times \Delta \bM\right) \right) \rmd t + \left(\lambda_1\bM\times \bm{g}\right) \circ\rmd W&\quad {\rm in}\ D_T,\ \P\textrm{-a.s.}
\end{align}
again with initial and boundary conditions as in \eqref{eq:LLG}. By $\circ\rmd W$ we denote the Stratonovich differential.
We define a weak solution of this problem following \cite{goldysLeTran}.

\begin{definition}
	A weak martingale solution of \eqref{eq:SLLG} is $\left(\Omega, \EE, \left(\EE_t\right)_{t\in[0,T]}, \P, W, \bM \right)$
	where 
	\begin{itemize}
		\item $\left(\Omega, \EE, \left(\EE_t\right)_{t\in[0,T]}, \P\right)$ is a filtered probability space;
		\item $W:\Omega\times [0,T]\rightarrow \R$ is a scalar Wiener process adapted to $\left(\EE_t\right)_{t\in[0,T]}$;
		\item $\bM: \Omega\times [0,T] \rightarrow L^2(D)^3$ is a progressively measurable stochastic process;
	\end{itemize}
	such that the following properties hold:
	\begin{itemize}
		\item $\bM(\omega, \cdot) \in C^0(0,T, H^{-1}(D) )$ $\P$-a.e. $\omega\in\Omega$;
		\item $\E\left(\textrm{esssup}_{t\in[0,T]} \norm{\nabla \bM(t)}{L^2(D)}^2\right) < \infty$;
		\item $\seminorm{\bM(\omega, t, \bx)}{} = 1$ $\P$-a.e. $\omega\in\Omega$, for all $t\in[0,T]$, for a.e. $\bx\in D$;
		\item For all $t\in [0,T]$ and all $\bphi \in C^{\infty}_0(D)^3$, $\P$-a.s. there holds
		\begin{align*}
			\dual{\bM(t), \bphi} &- \dual{\bM^0, \bphi}
			= -\lambda_1 \int_0^t \dual{\bM\times\ \nabla \bM, \nabla \bphi} \rmd s \\
			&-\lambda_2 \int_0^t \dual{\bM\times \nabla \bM, \nabla\left(\bM\times\bphi\right)}\rmd s
			+ \lambda_1 \int_0^t \dual{\bM\times \bg, \bphi} \circ \rmd W(s),
		\end{align*}
		where $\dual{\cdot, \cdot}$ denotes the $L^2(D)^3$ scalar product.
	\end{itemize}
\end{definition}
Existence of solutions to~\eqref{eq:SLLG} in this sense was first established in~\cite{BrzezniakGoldys2013}, while uniqueness of weak solutions is still an open question. An alternative existence proof was given in~\cite{goldysLeTran}. Here the authors use the Doss-Sussman transform to obtain a PDE with random coefficients instead of the stochastic differential as explained in the previous section.

\section{Random LLG equation by  Doss-Sussmann transform and parametric LLG equation by Lévy-Ciesielski expansion} \label{sec:rnd_LLG_by_Doss_Sussmann}
In this section, we apply the strategy outlined in Section \ref{sec:from_SPDE_to_paramPDE_general} to the SLLG equation~\eqref{eq:SLLG} in order to obtain a random coefficient PDE.
While this was done in~\cite{goldysLeTran} for technical reasons, we are mainly interested in obtaining an equivalent problem that is more amenable to collocation-type approximation. Another advantage is (formally) gaining a full order of differentiability of the solution.
Given $\bg:D\rightarrow \R^3$, $s\in\R$ and $\bv:D\rightarrow \R^3$ with suitable regularity, consider the following operators: 
\begin{align}
	G\bv &= \bv \times \bg,\\
	\mathcal{C} \bv &= \bv\times \Delta \bg + 2\nabla \bv\times \nabla \bg,\\
	e^{sG}\bv  &= \bv + \sin(s) G \bv + (1-\cos s) G^{2}\bv,  \label{eq:def_eWG}\\
	\EE(s, \bv) &= \sin(s) \mathcal{C}\bv + (1-\cos(s)) (\mathcal{C}G + G\mathcal{C})\bv,\\
	\hC(s,\bv) &= e^{-s G}\mathcal{E}(s,\bv) = \mathcal{E}(s,\bv) - \sin(s) G \mathcal{E}(s,\bv) + (1-\cos(s))G^{2}\mathcal{E}(s,\bv), \label{eq:def_CHat}
\end{align}
where we define $\nabla \bv\times \nabla \bg \coloneqq \sum_{j=1}^3 \frac{ \partial \bv}{\partial x_j} \times \frac{\partial \bg}{\partial x_j}$.
Note that $e^{sG}$ is the exponential of the operator $sG$. The fact $G\circ G\circ G\bv = -\bv$ simplifies the expression.
Expanding some definitions, the last operator can be written as
\begin{align*}
\hC(s,\bv) 
&= \sin(s) \CC \bv
+ (1-\cos(s)) \left(\CC G + G \CC\right)\bu
- \sin(s)^2 G\CC \bu\\
& -\sin(s)(1-\cos(s))G\left(\CC G+G\CC \right)\bu 
+ (1-\cos(s))\sin(s) G^2 \CC \bu\\
& + (1-\cos(s))^2 G^2\left(\CC G+G\CC \right) \bu
\end{align*}
or, in compact form, as
\begin{align}\label{eq:CHat_sum_structure}
\hC(s,\bv) = \sum_{i=1}^6 b_i(s) F_i(\bv),
\end{align}
where $b_i$ are uniformly bounded with bounded derivatives (let $0<\beta<\infty$ be a uniform bound for both, which depends only on $\bg$) and the $F_i$ are linear and globally Lipschitz with the Lipschitz constant $0<L<\infty$ depending only on $\bg$, i.e., for any $i=1,\dots, 6$, 
\begin{align*}
&\norm{b_i(W)}{L^{\infty}(\R)}\leq \beta, \quad \norm{b_i'(W)}{L^{\infty}(\R)}\leq \beta\qquad \forall W\in C^0([0,T]),\\
&\norm{F_i(\bu) - F_i(\bv)}{L^2(D)} \leq L \norm{\bu-\bv}{H^1(D)}\qquad \forall \bu,\bv\in H^1(D)^3.
\end{align*}
In the present setting, the \emph{Doss-Sussmann transform} \eqref{eq:def_Doss_Sussmann} reads
$ \m = e^{-W G}\bM.$
We obtain the \emph{random coefficients LLG equation}: Given $\bM^0:D\rightarrow \Sb^2$, find $\m:\Omega\times D_T \rightarrow \Sb^2$ such that for $\P$-a.e. $\omega\in \Omega$ 
\begin{align}\label{eq:rnd_LLG_problem}
\begin{cases}
\partial_t \m(\omega) &= \lambda_1 \m(\omega)\times\left(\Delta \m(\omega) + \hC(W(\omega), \m(\omega))\right) \\ 
& -\lambda_2 \m(\omega)\times\left(\m(\omega)\times\left(\Delta \m(\omega) + \hC(W(\omega), \m(\omega))\right)\right)\qquad \textrm{in } D_T,\\
\partial_n \m(\omega) &= 0\qquad \textrm{on } [0,T]\times\partial D, \\
\m(\omega, 0, \cdot) &= \bM^0\qquad \textrm{on } D.
\end{cases}
\end{align}
It is shown in~\cite[Lemma 4.6]{goldysLeTran} that any weak solution $\m$ of~\eqref{eq:rnd_LLG_problem} corresponds to a weak martingale solution ${\bM=e^{WG}\m}$ of~\eqref{eq:SLLG} through the \emph{inverse Doss-Sussmann transform}. Existence of solutions to~\eqref{eq:rnd_LLG_problem} is shown in~\cite{goldysLeTran}, but again uniqueness is open.

Following Section \ref{sec:from_SPDE_to_paramPDE_general}, we derive a parametric PDE problem using the Lévy-Ciesielski expansion of the Wiener process.
The \emph{parametric LLG equation} reads: Given $\bM^0:D\rightarrow \Sb^2$, find $\m:\XX_{\R}\times D_T \rightarrow \Sb^2$ such that for a.e. $\by\in\XX_{\R}$
\begin{align}\label{eq:parametric_LLG_stong}
\begin{cases}
\partial_t \m(\by) &= \m(\by)\times\left(\Delta \m(\by) + \hC(W(\by), \m(\by))\right) \\ 
 & - \m(\by)\times\left(\m(\by)\times\left(\Delta \m(\by) + \hC(W(\by), \m(\by))\right)\right)\qquad \textrm{in } D_T,\\
\partial_n \m(\by) &= 0\qquad \textrm{on } [0,T]\times \partial D, \\
\m(\by, 0, \cdot) &= \bM^0\qquad \textrm{on } D,
\end{cases}
\end{align}
where we  set $\lambda_1=\lambda_2=1$ for simplicity.
The precise definition of $\XX_{\R}$ is given below in \eqref{eq:def_param_space_R_Holder}.

Applying the triple cross-product formula ${\bm{a}\times(\bm{b}\times\bm{c}) = \bm{b}(\bm{a}\cdot \bm{c}) - \bm{c}(\bm{a}\cdot \bm{b})}$ on $\m(\by)\times\left(\m(\by)\times\left(\Delta \m(\by)\right)\right)$, together with the fact that $\seminorm{\m}{}\equiv 1$, gives an equivalent equation valid again for a.e. $\by\in \XX_{\R}$: 
\begin{align}\label{eq:llg_strong_alt}
    \partial_t \m(\by) &= \Delta \m(\by) + \m(\by)\times \Delta \m(\by) - \left( \nabla \m(\by):\nabla \m(\by)\right) \m(\by) \\
    &+ \m(\by) \times \hC(W,\m(\by)) - \m(\by)\times\left(\m(\by)\times \hC(W,\m(\by))\right)\quad \textrm{in } D_T.
\end{align}

\revision{ We conclude the section with a result on space and time Hölder regularity of solutions of the random LLG equation~\eqref{eq:rnd_LLG_problem}.
In Appendix~\ref{appendix:Holder_regularity}, we recall the definitions of basic Hölder spaces and give a detailed proof of the result.} 

We define the \emph{parabolic distance} 
$d(P,Q) \coloneqq \left( \vert t-s\vert + \vert \bx-\by\vert^2\right)^{1/2}$
 between $P=(t,\bx)$, $Q=(s,\by)\in D_T$.	
For $v:D_T\rightarrow \C$ and $0<\alpha<1$, define the seminorm
$\vert v\vert_{C^{\alpha/2, \alpha}(D_T)} \coloneqq \sup_{\substack{P,Q\in D_T \\ P\neq Q}} \frac{\vert v(P)-v(Q)\vert}{d(P,Q)^{\alpha}}$
and the Banach spaces
$C^{\alpha/2, \alpha}(D_T)$
with the norm $\norm{v}{C^{\alpha/2, \alpha}(D_T)} \coloneqq \norm{v}{C^0(D_T)} + \vert v\vert_{C^{\alpha/2, \alpha}(D_T)}$ (see~\cite[Section 1.2.3]{Wu2006Elliptic} for details).
Finally, consider the Banach space
\begin{align}\label{eq:space_Holder_1_time_2_space}
C^{1+\alpha/2, 2+\alpha}(D_T) \coloneqq
\set{v:D_T\rightarrow \C}{\partial_t^i \partial_x^j v \in C^{\alpha/2, \alpha}(D_T) \ \forall i,j\in\N_0: 2i+j \leq 2}
\end{align}
with the norm
$\norm{v}{C^{1+\alpha/2, 2+\alpha}(D_T)}\coloneqq 
	\sum_{j=0}^2 \norm{D^{j}v}{C^{\alpha/2,\alpha}(D_T)}
	+\norm{\partial_t v}{C^{\alpha/2,\alpha}(D_T)}$.
In what follows, we work with the corresponding Hölder seminorm
\begin{align}\label{eq:Holder_seminorm_k}
	\seminorm{v}{C^{1+\alpha/2, 2+\alpha}(D_T)} 
	\coloneqq 
	\seminorm{v}{C^{\alpha/2,\alpha}(D_T)}+
	\sum_{j=1}^2 \norm{D^{j}v}{C^{\alpha/2,\alpha}(D_T)}
	+\norm{\partial_t v}{C^{\alpha/2,\alpha}(D_T)}.
\end{align}
These H\"older spaces are closed under multiplication and their definitions generalize to vector fields as usual. 
In the remainder of this section, we adopt the short notation $\norm{\cdot}{\alpha} = \norm{\cdot}{C^{\alpha}(D)}$, $\norm{\cdot }{1+\alpha/2,2+\alpha} = \norm{\cdot}{C^{1+\alpha/2, 2+\alpha}(D_T)}$, and analogously for all other Hölder norms and seminorms.

The H\"older regularity properties of the sample paths are summarized in the following theorem, whose proof can be found in Appendix~\ref{appendix:Holder_regularity}.
\begin{theorem}\label{th:Holder_sample_paths}
	Let $0<\alpha<1$. Assume that $W\in C^{\alpha/2}([0,T])$, $\bM^0\in C^{2+\alpha}(D)$ and $\bg\in C^{2+\alpha}(D)$.
	There exists $\eps>0$ such that if 
	$\norm{\bM^0}{2+\alpha} \leq \eps$,
	$\norm{\Delta\bg}{\alpha} \leq \eps$, and
	$\norm{\nabla\bg}{\alpha} \leq \eps$,
	then the solution $\m$ of the random LLG equation~\eqref{eq:rnd_LLG_problem} with initial condition $\m(0)=\bM^0$ and homogeneous Neumann boundary conditions belongs to $C^{1+\alpha/2, 2+\alpha}(D_T)$. Moreover,
	\begin{align}\label{eq:estim_norm_m}
		\norm{\m}{1+\alpha/2, 2+\alpha} \leq C_r,
		% C \left(\norm{\m^0}{2+2\gamma} + 1 \right),
	\end{align}
	where $C_r>0$ depends on $\norm{\bg}{2+\alpha}$, $\norm{\bM^0}{2+\alpha}$, $\lambda$, $D$ and $T$ but is independent of $W$.
	% where $C$ is independent of $\m^0$ and depends on $\bg$.
\end{theorem}

\section{Holomorphic regularity of parameter-to-solution map with Hölder sample paths}\label{sec:holomoprhic_regularity_Holder}
In this section, we frequently work with complex-valued functions. If not mentioned otherwise, Banach spaces of functions such as $L^2(D)$ are understood to contain complex valued functions. To denote the codomain explicitly, we write e.g. $L^2(D; \C)$ or $L^2(D;\R)$.

We specify a possible choice of Banach spaces used in Section~\ref{sec:proofstrategy} for the case of the SLLG equation. Fix $0 < \alpha <1$ and consider the parameter set
\begin{align}\label{eq:def_param_space_Holder}
\XX = \XX(\alpha) \coloneqq \set{\bz\in \C^{\N}}{\norm{\bz}{\XX,\alpha} <\infty},\quad
\textrm{where }\norm{\bz}{\XX,\alpha} \coloneqq \sum_{\ell \in \N_0} \max_{j=1,\dots,\lceil 2^{\ell-1}\rceil} \seminorm{z_{\ell, j} }{} 2^{-(1-\alpha)\ell/2},
\end{align}
where we used the hierarchical indexing \eqref{eq:re-indexing_LC}.
For real parameters consider
\begin{align}\label{eq:def_param_space_R_Holder}
\XX_{\R} \coloneqq \XX \cap \R^{\N}.
\end{align}
The definition of the Banach spaces \revision{for real and complex coefficients sample paths follows from the Lévy-Ciesielski expansion~\eqref{eq:def_param_LC}:
\begin{align*}
	\W &\coloneqq \set{W:[0,T]\rightarrow\C}{\exists \bz\in\XX \textrm{ such that } W(t) = \sum_{n\in\N} z_n \eta_n(t) \forall t\in[0,T]},\\
	\W_\R &\coloneqq \set{W:[0,T]\rightarrow\R}{\exists \by\in\XX_\R \textrm{ such that } W(t) = \sum_{n\in\N} y_n \eta_n(t) \forall t\in[0,T]}.
\end{align*}}
It is however interesting to identify classical spaces to which they belong.
\begin{remark}\label{rk:choice_alpha}
	In the regularity results used below, we have to work in H\"older spaces with $\alpha\in(0,1)$. 
	For the Faber-Schauder basis functions on $[0,1]$ (see Section \ref{sec:from_SPDE_to_paramPDE_general}) we have 
	\begin{align*}
	\norm{\eta_{\ell, j}}{L^{\infty}(0,1)} \leq  2^{-\ell/2},\quad\seminorm{\eta_{\ell, j}}{C^1([0,1])} \leq  2^{\ell/2}, \quad\text{and}\quad \norm{\eta_{\ell, j}}{C^{\alpha}([0,1])} \leq  2\cdot 2^{-(1/2-\alpha)\ell}.
	\end{align*}
	Only for $\alpha\ll 1$, we obtain a decay of $\norm{\eta_{\ell, j}}{C^{\alpha}([0,1])}$ close to $2^{-\ell/2}$, which is what we expect for a truncated Brownian motion. Hence, from now we assume that $\alpha>0$ is arbitrarily small.
\end{remark}
It can be proved that
\begin{align}\label{eq:def_WR_HolderLLG}
\W_{\R} &\subset C^{\alpha/2}([0,T]; \R)\\
\label{eq:def_W_HolderLLG}
\W &\subset C^{\alpha/2}([0,T])
\end{align}
with the same techniques used in the proof of Lemma \ref{lemma:sparsity_param_to_sol_map} below. This choice of parameter space is motivated by the fact that the sample paths of the Wiener process belong to $C^{1/2-\eps}([0,T])$ almost surely for any $\eps>0$.
To define the space of solutions $\U$, write the magnetizations as
\begin{align*}
\m(\omega,t,\bx) = \bM^0(\bx) + u(\omega, t,\bx)\qquad \textrm{for a.e. }\omega\in \Omega, (t,\bx)\in D_T,
\end{align*}
where we recall $\bM^0$ is the given initial condition assumed to belong to $C^{2+\alpha}(D)$. Consider 
\begin{align}\label{eq:def_U_HolderLLG}
u \in \U &= C^{1+\alpha/2, 2+\alpha}_0(D_T) \coloneqq \left\{\bv \in C^{1+\alpha/2, 2+\alpha}(D_T) : \bv(0)=\b{0}\ \textrm{in } D,\ \partial_{n} \bv = \b{0}\ \textrm{on }\partial D\right\},\\
\U_{\R} &= \set{\bv:D_T\rightarrow \Sb^2}{\bv\in \U},
\end{align}
where $\Sb^2$ is the unit sphere in $\R^3$.
See Section \ref{sec:rnd_LLG_by_Doss_Sussmann} for the definition of the relevant Hölder spaces.
Given a noise coefficient $\bg\in C^{2+\alpha}(D)$, we define the residual as:
\begin{align}\label{eq:forward_map_holder_law}
\begin{split}
\RR(W,u) &\coloneqq \tRR(W, \bM^0+u), \qquad \textrm{where }\\
\tRR(W, \m) &\coloneqq \partial_t \m -\Delta \m - \m\times \Delta \m + \left( \nabla \m:\nabla \m\right) \m - \m \times \hC(W,\m) + \\
&+\m\times\left(\m\times \hC(W,\m)\right).
\end{split}
\end{align}
Here, the cross product is defined as in the real setting: for any $\bm{a},\bm{b}\in\C^3$, let
$
    \bm{a}\times \bm{b} = \left(a_2b_3-a_3b_2, a_3b_1-a_1b_3, a_1b_2-a_2b_1\right).
$
Note that due to the sesquilinear complex scalar product this implies that $\dual{\bm{a}\times \bm{b}, \bm{a}}$ might not vanish for complex valued vector fields $\bm{a},\bm{b}$.
Finally, the space of residuals is
\begin{align*}
	R = C^{\alpha/2, \alpha}(D_T),
\end{align*}
so that $\RR$ is understood as a function between Banach spaces:
\begin{align}\label{eq:forward_map_holder}
	\RR : \W\times \U\rightarrow R, \qquad (W,\m) \mapsto \RR(W,\m).
\end{align}
Observe that we already proved Assumption~\ref{assume:uniform_bounded_solution} in Theorem \ref{th:Holder_sample_paths}.

\subsection{Proof of Assumptions \ref{assume:assumptions_IFTh} and \ref{assume:bounds_d1R_d2Rinv}}
In order to apply the general strategy outlined in Section \ref{sec:proofstrategy}, we need to prove Assumption \ref{assume:assumptions_IFTh} and \ref{assume:bounds_d1R_d2Rinv} for the problem defined by~\eqref{eq:forward_map_holder_law}.

\begin{remark}
\revision{In the next lemma, we apply the well posedness result for parabolic PDEs with Hölder coefficient~\cite[Chapter~VII, \S~10, Theorem 10.3]{ladyzhenskaya1968linear}}.
\revision{
The validity of the theorem hinges on the fact that the problem is \emph{strongly parabolic}, i.e. the principal part $\AA_0$ of the elliptic operator satisfies: There exists $\delta>0$ such that for a.e. $(t,\bx)\in D_T$, 
\begin{align*}
	\Re{ \dual{ \AA_0(t, \bx) \bz, \bz}_{}} \geq \seminorm{\bz}{}^2 \qquad \forall\bz\in\C^3,
\end{align*}
where $\dual{\cdot,\cdot}$ and $\seminorm{\cdot}{}$ denote the standard scalar product and norm on $\C^3$.
}
Note that the compatibility conditions in~\cite[Chapter~VII, \S~10, Theorem 10.3]{ladyzhenskaya1968linear} of order zero ($\alpha<1$) are automatically satisfied in our case. This also takes care of the fact that~\cite[Chapter~VII, \S~10, Theorem 10.3]{ladyzhenskaya1968linear} only works for small end times $0<\widetilde T\leq T$ as we can restart the estimate at any time $\widetilde T$ and get the estimate for the full time interval. Moreover, while not stated explicitly, analyzing the proof of~\cite[Chapter~VII, \S~10, Theorem 10.3]{ladyzhenskaya1968linear} gives the dependence of $C_{\rm stab}$ on the coefficients of the problem.
\end{remark}
\begin{lemma}\label{lemma:hyp_implicit function_Holder}
Let $\alpha\in(0,1)$, $\bg\in C^{2+\alpha}(D)$ and $\bM^0\in C^{2+\alpha}(D)$.
Consider the spaces $\W, \W_{\R}, \U, \U_{\R}, R$ defined at the beginning of this section. 
Then, the residual $\RR$~\eqref{eq:forward_map_holder_law}-\eqref{eq:forward_map_holder} is a well-defined function and Assumptions \ref{assume:assumptions_IFTh} holds true. More generally, it can be proved that $\partial_u \RR(W, u)$ is a homeomorphism between $\U$ and $R$ if
\begin{align} \label{eq:SC_DuR_homeo_LLG}
W\in\W \quad\text{and}\quad u\in\U \text{ satisfies } \norm{\Im{u}}{L^{\infty}(D_T)} \leq \frac{1}{4}.
\end{align}
Finally, Assumption \ref{assume:bounds_d1R_d2Rinv} also holds true with $\eps_W>0$, $\eps_u = \frac{1}{4}$, and
\begin{align*}
    \GG_1(s) &= \left(1+e^{\eps_W}(1+\eps_W) \right)^2 \left(1+\norm{\bg}{C^{2+\alpha}(D)}\right)^4 \left(1+\norm{\bM^0}{C^{2+\alpha}(D)}+s\right)^3,\\
    \GG_2(s) &= C_{\rm stab}(s)\qquad \forall s\geq 0,
\end{align*}
and $C_{\rm stab}=C_{\rm stab}( \norm{u}{\U})>0$ is as in\revision{~\cite[Chapter~VII, \S~10, Theorem 10.3]{ladyzhenskaya1968linear}}, i.e. it guarantees that \\
$\norm{\left(\partial_u \RR(W, u)\right)^{-1} f}{\U} \leq C_{\rm stab}(\norm{u}{\U}) \norm{f}{R}$ for any $f\in R$, $W\in\W$, $u\in\U$.
\end{lemma}
\begin{proof}[Proof that $\RR$ is well-defined]
Let us first show that the residual $\RR$ is a well-defined function.
Clearly, $\bM^0 +u \in C^{1+\alpha/2, 2+\alpha}(D_T)$ if $u \in C^{1+\alpha/2, 2+\alpha}_0(D_T)$. 
Observe that
\begin{align*}
G &\colon C^{1+\alpha/2,2+\alpha}(D_T)\to C^{1+\alpha/2, 2+\alpha}(D_T)\quad\text{and}\quad
\mathcal{C} \colon C^{\alpha/2,1+\alpha}(D_T)\to C^{\alpha/2, \alpha}(D_T),
\end{align*}
so $\hC(W,\m)\in C^{\alpha/2, \alpha}(D_T)$.
Thus, $\RR(W, u)$ is a sum of functions belonging to $C^{\alpha/2, \alpha}(D_T)$.
The fact that $\RR$ is continuous can be easily verified by checking that each term in~\eqref{eq:forward_map_holder_law} is continuous.
\end{proof}
\begin{proof}[Proof of \emph{(i)} in Assumption \ref{assume:assumptions_IFTh}] The residual $\RR$ is differentiable because it is a linear combination of differentiable functions.
We now prove that each partial derivative is continuous.
For $\omega \in C^{\alpha/2}([0,T])$,
\begin{align}\label{eq:cont_D1R}
&\partial_W\mathcal{E}(W,\m)[\omega] =  \left( \cos(W) \mathcal{C} \m + \sin(W) (G\mathcal{C} + \mathcal{C} G)\m \right) \omega,\\
&\partial_W\hC(W,\m)[\omega] = e^{W G} \partial_1 \mathcal{E}(W,\m)[\omega] + \left(\cos(W) G \mathcal{E}(W,\m) + \sin(W) G^2 \mathcal{E}(W,\m)\right) \omega,\\
&\partial_W\tRR(W,\m)[\omega] = -\m\times \partial\hC(W,\m)[\omega] + \m\times\left(\m\times \partial\hC(W,\m)[\omega] \right).
\end{align}
Formally estimating the linear operator $\partial_W \RR(W,u)$ gives that for all $\omega\in C^{\alpha/2}([0,T])$:
\begin{align}\label{eq:continuity_D1F}
\begin{split}
\norm{\partial_W\RR(W,u)[\omega]}{C^{\alpha/2, \alpha}(D_T)} \leq
&\left(1+\norm{e^{\Im{W}}}{C^{\alpha/2}([0,T])}\right)^2 
\left(1+\norm{\bg}{C^{2+\alpha}(D)}\right)^4\\
&\quad\left(1+\norm{\bM^0+u}{C^{\alpha/2,1+\alpha}(D_T)}\right)^3
\norm{\omega}{C^{\alpha/2}([0,T])}.
\end{split}
\end{align}
The exponential dependence on $\Im{W}$ comes from the exponential behavior of sine and cosine in imaginary direction.
\revision{It can easily be proved, using $\norm{\Im{W}}{\W} \leq \eps$, that 
\begin{align}\label{eq:bound_Holder_norm_exp_Im_W}
	\norm{e^{\Im{W}}}{C^{\alpha/2}([0,T])} \lesssim e^{\eps}(1+\eps).
\end{align}
}

For $\bv\in C^{1+\alpha/2, 2+\alpha}(D_T)$, we get
\begin{align}\label{eq:d2R_law}
\begin{split}
\partial_{\m} \tRR(W, \m)[\bv]
&= \partial_t \bv - \Delta \bv - \bv\times \Delta \m - \m\times \Delta \bv + 2(\nabla \bv:\nabla\m) \m + \left(\nabla \m:\nabla \m\right) \bv \\
&- \left( \bv\times \hC(W, \m) + \m\times \hC(W, \bv)\right)\\
&- \left(\bv\times\left(\m\times \hC(W, \m)\right) + \m\times\left(\bv\times \hC(W, \m)+\m\times \hC(W, \bv)\right)\right),
\end{split}
\end{align}
and continuity of $\partial_u \RR(W, u) = \partial_{\m} \tRR(W, \bM^0+u)$ follows as for $\partial_W \RR(W, \m)$.
\end{proof}

\begin{proof}[Proof of \emph{(ii)} in Assumption \ref{assume:assumptions_IFTh}] 
While we are only interested in the case $W\in \W_{\R}$, $u\in\U_{\R}$ such that $\RR(W,u)=0$, let us consider the more general case \eqref{eq:SC_DuR_homeo_LLG} for future use.
Consider $\bm{f}\in R$ (the residuals space) and the problem:
\begin{align*}
\begin{cases}
\partial_u \RR(W_*, u_*)[\bv] &= \bm{f}\qquad \textrm{in } D_T,\\
\partial_n \bv &= \b{0}\qquad  \textrm{on } [0,T]\times \partial D,\\
\bv(0,\cdot)&=\bm{0}\qquad  \textrm{on } D.
\end{cases}
\end{align*}
With the aim of applying the well-posedness result\revision{~\cite[Chapter~VII, \S~10, Theorem 10.3]{ladyzhenskaya1968linear}}, we note that the principal part of $\partial_u \RR(W, u)[\bv]$ is $-\Delta \bv - u\times \Delta \bv$. 
We now show that for any $(t,\bx)\in D_T$ and $\bm{w}\in \C^3$,
\begin{align}\label{eq:condtion_well_posed_PDE}
    \Re{\dual{\bm{w}+u(t,\bx)\times \bm{w}, \bm{w}}} \geq \frac{1}{2} \norm{\bm{w}}{}^2,
\end{align}
where here $\norm{\cdot}{}$ and $\dual{\cdot, \cdot}$ denote respectively the standard norm and scalar product on $\C^3$.
Indeed, 
$
    \Re{\dual{\bm{w} + u(t,\bx)\times \bm{w}, \bm{w}}}
    = \norm{\bm{w}}{}^2 + \Re{\dual{u(t,\bx)\times\bm{w}, \bm{w}}}
$
and algebraic manipulations lead to the identity
$
	\Re{\dual{u(t,\bx)\times\bm{w} , \bm{w}}}
	=2 \dual{\Im{\bm{w}}\times \Re{\bm{w}}, \Im{u(t,\bx)}},
$
which implies the estimate
\begin{align*}
\lvert\Re{\dual{u(t,\bx)\times\bm{w} , \bm{w}}}\rvert \leq 2\norm{\Im{u(t,\bx)}}{L^\infty(D_T)} \norm{\bm{w}}{}^2.
\end{align*}
Thus, by virtue of Assumption \eqref{eq:SC_DuR_homeo_LLG}, we obtain \eqref{eq:condtion_well_posed_PDE}.
This shows that $\partial_u \RR(W, u)$ is parabolic in the sense of\revision{~\cite[Chapter~VII, \S~10, Theorem 10.3]{ladyzhenskaya1968linear}}
%Lemma~\ref{lemma:parabolic_regularity} 
and hence, we obtain that $\partial_u \RR(W, u)$ admits a continuous inverse. 
Together with its continuity, this implies that it is a homeomorphism.
The norm of the inverse can be estimated as 
\begin{align}\label{eq:stability_inv_D2F}
\norm{\partial_u \RR(W, u)^{-1}\left[\b{f}\right]}{C^{1+\alpha/2, 2+\alpha}(D_T)} 
\leq C_{\rm stab}(W, u) \norm{\bm{f}}{C^{\alpha/2, \alpha}(D_T)},
\end{align}
where $C_{\rm stab}(W, u)>0$ is independent of $\b{f}$ (but depends on $W$ and $u$).
\end{proof}
\begin{proof}[Proof of Assumption \ref{assume:bounds_d1R_d2Rinv}]
The continuity of $\partial_W \RR(W, u)$ follows from \eqref{eq:continuity_D1F}, \eqref{eq:bound_Holder_norm_exp_Im_W} with  
\begin{align*}
\GG_1(s) = 
\left(1+e^{\eps_W}(1+\eps_W) \right)^2 
\left(1+\norm{\bg}{C^{2+\alpha}(D)}\right)^4
\left(1+\norm{\bM^0}{C^{2+\alpha}(D)}+s\right)^3,
\end{align*}
where $\eps_W >0$.
The bound on $\left(\partial_u \RR(W, u)\right)^{-1}$ is already proved in \eqref{eq:stability_inv_D2F} with $\eps_u=\frac{1}{4}$ and $\GG_2 = C_{\rm stab}$. The fact that $C_{\rm stab}$ depends on $\UU(W)$ only through $\norm{\UU(W)}{\U}$ is implied by the sufficient condition for well posedness in \eqref{eq:SC_DuR_homeo_LLG}.
\end{proof}

We recall that, as shown in Section \ref{sec:param_regularity_general}, the implicit function theorem and Theorem \ref{th:global_holo_ext} prove the existence of $\eps>0$ such that for any $\by\in\XX_\R$ there exists a holomorphic map $\UU:B_{\eps}(W(\by)\rightarrow \U$ such that $\RR(W,\UU(W))=0$ for all $W\in B_{\eps}(W(\by))$. The function $\UU$ is bounded by a constant $C_{\eps}>0$ again independent of $\by$.

Moreover,  Assumption \ref{assume:bounds_d1R_d2Rinv} implies the bound \eqref{eq:estim_norm_DU} on the differential $\UU'(W)$ as a function of $\UU(W)$ through $\norm{\UU(W)}{\U}$ under the assumption that $W\in B_{\eps}(W(\by))$ in $\W$.

\subsection{Proof of Assumption \ref{assume:sparsity} and estimates of derivatives of parameter-to-solution map} \label{sec:uniform_holo_Holder}

Let us now estimate the derivatives of the parameter-to-solution map.
While this is a standard technique established already in~\cite{Cohen2011Analytic}, it turns out this will not be quite sharp enough to obtain dimension independent convergence of the sparse grid approximation. In Section~\ref{sec:holomorphic_regularity_L1_small}, we present a possible way to resolve this in the future.

Let us show that Assumption \ref{assume:sparsity} holds for the present problem. Recall the definitions of parameter spaces in~\eqref{eq:def_param_space_Holder} and~\eqref{eq:def_param_space_R_Holder}.
\begin{lemma}\label{lemma:sparsity_param_to_sol_map}
Assumption \ref{assume:sparsity} holds in the present setting. In particular, we can choose $\brho = (\rho_n)_{n\in\N}$ such that 
\begin{align}\label{eq:choice_rho_Holder}
    \norm{\brho}{\XX} \leq \frac{\eps}{2}.
\end{align}
\end{lemma}
\begin{proof}
Fix $\by\in \XX_{\R}$ and $\bz\in\bB_{\bm{\rho}}(\by)$ (i.e.$\seminorm{z_n-y_n}{} < \rho_n$ for all $n\in\N$). Let us prove that $W(\bz)\in B_\eps(W(\by))$.
By linearity,
$ W(\bz,\cdot) - W(\by, \cdot) = \sum_{n\in\N} (z_n-y_n) \eta_n(\cdot)$.
Recalling the hierarchical indexing \eqref{eq:re-indexing_LC} and by a triangle inequality, we obtain
\begin{align*}
\norm{W(\bz,\cdot) - W(\by,\cdot)}{C^{\alpha/2}([0,T])} \leq 
\sum_{\ell\in\N_0} \norm{\sum_{j=1}^{\lceil 2^{\ell-1} \rceil} {(z_{\ell, j}-y_{\ell,j}) } \eta_{\ell,j} }{C^{\alpha/2}([0,T])}.
\end{align*}
The terms on the right-hand side can be estimated by Banach space interpolation and the fact that all basis functions $\eta_{\ell,j}$ on the same level have disjoint supports, i.e.,
\begin{align*}
	&\norm{\sum_{j} (z_{\ell, j}-y_{\ell,j}) \eta_{\ell,j} }{C^{\alpha/2}([0,T])} 
	\leq \norm{\sum_{j} (z_{\ell, j}-y_{\ell,j}) \eta_{\ell,j} }{C^0([0,T])}^{1-\alpha/2}
	\norm{\sum_{j} (z_{\ell, j}-y_{\ell,j}) \eta_{\ell,j} }{C^1([0,T])}^{\alpha/2}\\
	\leq & \big(\max_j \seminorm{z_{\ell, j}-y_{\ell,j}}{} \norm{\eta_{\ell,j} }{C^{0}([0,T])}\big)^{1-\alpha/2}\\
	 & \big(\max_j \seminorm{z_{\ell, j}-y_{\ell,j}}{}\norm{\eta_{\ell,j} }{C^0([0,T])}+\max_j \seminorm{z_{\ell, j}-y_{\ell,j}}{}\seminorm{\eta_{\ell,j} }{C^{1}([0,T])}\big)^{\alpha/2}.
\end{align*}
Recalling that $\norm{\eta_{i(\ell)}}{C^{0}([0,T])} \leq 2^{-\ell/2}$ and $\seminorm{\eta_{i(\ell)}}{C^{1}([0,T])} \leq  2^{\ell/2}$ (see Remark~\ref{rk:choice_alpha}), we find 
\begin{align*}
\norm{\sum_{j} (z_{\ell, j}-y_{\ell,j}) \eta_{\ell,j} }{C^{\alpha/2}([0,T])}
\leq \max_j \seminorm{z_{\ell, j}-y_{\ell,j}}{} (2^{-\ell/2}+2^{-(1-\alpha)\ell/2}).
\end{align*}
With $\bz\in \bB_{\brho}(\by)$, we obtain $\norm{W(\bz,\cdot) - W(\by,\cdot)}{C^{\alpha/2}([0,T])} < \eps $, which gives the statement.
\end{proof}
An example of valid sequence of holomorphy radii is
\begin{align}\label{eq:def_rho_n}
\rho_n = \eps 2^{\frac{(1-\alpha)\lceil\log_2(n)\rceil}{2}}\qquad \forall n\in\N.
\end{align}
Having so concluded that for any $\by\in\XX_\R$ the parameter-to-solution map $\MM\circ W:\bB_{\bm{\rho}}(\by) \rightarrow \U$ is holomorphic and uniformly bounded, we can estimate its derivatives as in Theorem \ref{th:bound_mix_derivative}.
\begin{proposition}
Consider $\m = \bM^0 + u :\XX_{\R} \rightarrow C^{1+\alpha/2, 2+\alpha}(D_T)$, the parameter-to-solution map of the parametric LLG equation with Hölder spaces ($\XX_\R$ and $C^{1+\alpha/2, 2+\alpha}(D_T)$ defined in \eqref{eq:def_param_space_R_Holder} and \eqref{eq:space_Holder_1_time_2_space} respectively).
Fix $\eps>0$ as in Theorem \ref{th:global_holo_ext} and let $\bm{\rho} = (\rho_n)_{n\in\N}$ a positive sequence that satisfies \eqref{eq:choice_rho_Holder}.
Then, for any $n\in\N$, $\bnu = \left(\nu_i\right)_{i=1}^n \in \N^n$, it holds that
\begin{align}\label{eq:bound_mix_derivatives_Holder}
\norm{\partial^{\bnu} \m(\by)}{C^{1+\alpha/2, 2+\alpha}(D_T)} 
\leq  \prod_{j=1}^n \nu_j!\rho_j^{-\nu_j} C_{\eps}\qquad \forall \by\in\XX_{\R},
\end{align}
where $C_{\eps}>$ 0 from Theorem \ref{th:global_holo_ext} is independent of $\bnu$ or $\by$.
\end{proposition}
\begin{remark} Note that we essentially proved ``$(\bm{b}, \xi, \delta, X)$-holomorphy''~\cite[Definition 4.1]{dung2023analyticity} for the Stochastic LLG equation in the case of a Hölder-valued parameter-to-solution map. 
However, this regularity is not sufficient to apply the theory in~\cite{dung2023analyticity}, as the summability coefficient is $p=2$, which lies out of the range $(0, \frac{2}{3})$ considered in~\cite{dung2023analyticity}. This fact is analogous to what happens in our analysis. \end{remark}

\section{Holomorphy of a simplified parameter-to-solution map with Lebesgue sample paths}\label{sec:holomorphic_regularity_L1_small}
In this section, we aim at proving stronger regularity and sparsity properties of the random LLG parameter-to-solution map again based on the general strategy outlined in Section \ref{sec:proofstrategy}.
A key observation is that these properties depend on the Banach spaces chosen for the sample paths of the random coefficients (in our case, the Wiener process) and the sample paths of the solutions (in our case, the magnetizations). In this case, we show that using \emph{Lebesgue} spaces for the time variable is superior to using Hölder spaces.

Because of the nonlinear nature of the random LLG equation, the results hold only for a simplified version of the stochastic input. We make the following modelling assumptions:
\begin{itemize}
\item The sample paths of the Wiener process $W$ are ``small''. This is justified e.g. for small final times $T \ll 1$ with high probability;
\item The gradient $\nabla \bg$ is ``small'', meaning that the stochastic noise is spatially uniform. This is justified for small domain sizes (samples in real-world applications are often in the nano- and micrometer range).
\end{itemize}
Altogether, we end up with the following simplifications in the random LLG residual (defined in \eqref{eq:forward_map_holder_law}):
\begin{align*}
%\seminorm{\nabla \m}{}^2\m \approx 0,\qquad
\nabla \m\times \nabla \bg \approx 0,\qquad
\sin(W)&\approx W,\qquad
1-\cos(W)\approx \frac{W^2}{2}\approx 0.
\end{align*}
Consequently, we approximate $\hC(W,\m)$ defined in \eqref{eq:def_CHat} with the first order expansion
\begin{align*}
\tilde{\CC}(W,\m) \coloneqq W \m\times \Delta \bg,
\end{align*}
where $\bg\in C^{2+\alpha}(D)$.
This term appears in the \emph{simplified random LLG residual}
\begin{align}\label{eq:forward_map_lebesgue_law}
\begin{split}
\RR_s(W,u) &\coloneqq \tRR_s(W, \bM^0+u), \qquad \textrm{where }\\
\tRR_s(W, \m) &\coloneqq \partial_t \m -\Delta \m - \m\times \Delta \m + \left( \nabla \m:\nabla \m\right) \m - \m \times \tC(W,\m) +\\
& +\m\times\left(\m\times \tC(W,\m)\right).
\end{split}
\end{align}
Observe that the magnetization corresponding to $W(\omega, \cdot)$ is $\m(\omega) = \bM^0 + u(\omega)$ for any $\omega\in\Omega$.

In order to define the space for the coefficients, we again start from the parameters: Define, for  $1<q<\infty$,
$
\XX = \XX^q \coloneqq \set{\bz\in \C^{\N}}{\norm{\bz}{\XX^q} < \infty}$, where $\norm{\bz}{\XX^q} \coloneqq \sum_{\ell\in\N_0} \seminorm{\bz_{\ell}}{\ell^q}  2^{-\ell(1/2+1/q)}$, and we denoted $\by_{\ell} = (y_{\ell, 1},\dots, y_{\ell, \lceil2^{\ell-1}\rceil})$.
We then define the space of (complex) coefficients through the Lévy-Ciesielski expansion \eqref{eq:def_param_LC}:
$\W = \set{W(\bz, \cdot):[0,T]\rightarrow \C}{\bz\in \XX}$.
For real parameters, we fix $\theta>0$ and let
\begin{align}\label{eq:def_real_param_lebesgue}
\XX_{\R} = \XX(\alpha,\theta) \coloneqq \set{\by\in\R^{\N}}{\norm{\by}{\XX(\alpha)} < \theta },
\end{align}
where $\XX(\alpha)$ was defined in \eqref{eq:def_param_space_Holder}.
\begin{lemma}\label{lemma:Lp_paths_W} 
For fixed $1<q < \infty$ and $\theta>0$, there holds,
\begin{align*}
\W \subset L^q(0,T)\quad\text{and}\quad 
\W_{\R} \subset \set{W\in C^{\alpha}([0,T]; \R)}{\norm{W}{C^\alpha([0,T])} < \theta}.
\end{align*}
\end{lemma}
\begin{proof}
To prove the first inclusion, fix $\bz\in \XX$ and estimate
\begin{align*}
\norm{W(\bz)}{L^q(0,T)}= 
\norm{\sum_{\ell\in\N_0}\sum_{j=1}^{\lceil2^{\ell-1}\rceil} z_{\ell,j} \eta_{\ell,j}}{L^q(0,T)}
\leq \sum_{\ell\in\N_0} \norm{\sum_{j=1}^{\lceil2^{\ell-1}\rceil} z_{\ell,j} \eta_{\ell,j}}{L^q(0,T)}.
\end{align*}
Examine one summand at a time to get, using the fact that Faber-Schauder basis functions of same level have disjoint supports,
\begin{align}\label{eq:applic_disj_basis_LCE}
	\norm{\sum_{j=1}^{\lceil2^{\ell-1}\rceil} y_{\ell,j} \eta_{\ell,j}}{L^q(0,T)}^q
	=\int_0^T \sum_{j=1}^{\lceil2^{\ell-1}\rceil} y_{\ell,j}^q \eta_{\ell,j}^q
	=\sum_{j=1}^{\lceil2^{\ell-1}\rceil} y_{\ell,j}^q \int_0^T  \eta_{\ell,j}^q
	= \seminorm{\by_{\ell}}{\ell^q}^q \norm{ \eta_{\ell,1}}{L^q(0,T)}^q.
\end{align}
Finally, we integrate~\eqref{eq:faber-shauder_basis} to obtain that
$\norm{\eta_{\ell,j}}{L^q(0,T)} = 2^{-\ell(1/2+1/q)} 2^{-1/2} \left(\frac{2}{q+1}\right)^{1/q}$, 
so we get $\norm{W(\bz)}{L^q(0,T)} \leq \norm{\bz}{\XX^q}$, which implies the first inclusion.
The second inclusion follows with the methods of the proof of Lemma~\ref{lemma:sparsity_param_to_sol_map}.
\end{proof}
Intuitively, $\W_{\R}$ can be understood as the set of  ``small'' real valued Wiener processes.
The space of solutions is chosen as
\begin{align} \label{eq:def_sol_Lebesgue}
\U &= \left\{u:D_T\rightarrow \C^3 : u\in L^q(0,T, C^{2+\alpha}(D)),\ \partial_t u \in  L^{q}(0,T, C^{\alpha}(D)),\right.\\ 
&\qquad\qquad\qquad\qquad \left. u(0,\cdot)=\b{0} \textrm{ on } D,\ \partial_{n}u=\b{0} \textrm{ on }[0,T]\times \partial D\right\},\\
\label{eq:def_real_sol_Lebesgue}
\U_{\R} &= \set{u:D_T\rightarrow \Sb^2}{u\in\U}.
\end{align}
Finally the space of residuals is chosen as 
$ R \coloneqq L^q(0,T, C^{\alpha}(D))$.
The map $\sRR$ is understood as a function between Banach spaces:
\begin{align}\label{eq:forward_map_lebesgue}
\sRR:\W\times \U \rightarrow R,
\end{align}
Observe that if $u\in\U$ for $q>1$, then $\norm{u(t)}{C^{\alpha}(D)} \leq \norm{\partial_t u}{L^1(0,T, C^{\alpha}(D))}$ for all $t\in[0,T]$. 
This implies that $u \in C^0([0,T], C^{\alpha}(D))$ and $\norm{u}{C^0([0,T], C^{\alpha}(D))} \leq \norm{u}{\U}$.
In particular, interpolation shows that for any $U\in\U$, $\norm{u}{L^{\infty}(D_T)}+\norm{u}{L^2(0,T, C^1(D))} \leq \norm{u}{\U}$.
Note that $\tC$ is bounded and linear in both arguments: For all $W\in\W, \m\in C^0([0,T], C^{\alpha}(D))$ it holds 
\begin{align}\label{eq:bound_tC}
\normV{\tC(W,\m)} \leq &\normW{W} \normM0{\m} \norm{\bg}{C^{2+\alpha}(D)}.
\end{align}

 The proof of Theorem~\ref{th:Holder_sample_paths} can be transferred to this simplified version of LLG and hence we have that there exists $\ol{C}_r=\ol{C}_r(\theta)>0$ such that 
\begin{align}\label{eq:bound_M_real_W_Lebesgue}
\norm{\UU(W)}{\U} \leq \ol{C}_{r} \qquad \forall W \in \W_{\R}.
\end{align}
This gives the validity of Assumption \ref{assume:uniform_bounded_solution} with $C_r=\ol{C}_r$ for the present problem.

\subsection{Proof of Assumptions \ref{assume:assumptions_IFTh} and \ref{assume:bounds_d1R_d2Rinv}} \label{sec:L1holo}
In order to apply the general strategy outlined in Section \ref{sec:param_regularity_general}, we need to prove Assumptions \ref{assume:assumptions_IFTh} and \ref{assume:bounds_d1R_d2Rinv} for the spaces and residual chosen at the beginning of this section.
\begin{remark}
The proof of {\em ii.} in Assumption \ref{assume:assumptions_IFTh} requires the use of an $L^q$-regularity result for the linear parabolic problem given by the operator $\partial_u \sRR (W, u):\U\rightarrow R$ which coincides with~\eqref{eq:d2R_law} but $\hC$ replaced by $\tC$.
For scalar problems, this can be found in~\cite[Section~4]{Schnaubelt2011Solvability}. Strictly speaking, however, Lemma~\ref{lemma:hyp_implicit function_Lebesgue} only holds under the assumption that~\cite{Schnaubelt2011Solvability} can be generalized to the vector valued case.
\end{remark}
We can prove, analogously to Lemma~\ref{lemma:hyp_implicit function_Holder}, the following result:
\begin{lemma}\label{lemma:hyp_implicit function_Lebesgue}
Let $\alpha\in(0,1)$, $\bg\in C^{2+\alpha}(D)$, $\bM^0\in C^{2+\alpha}(D)$ and $0<\theta<\infty$. Consider the spaces $\W, \W_{\R}, \U, \U_{\R}$ defined at the beginning of this section. 
Then, the residual $\RR_s$ defined by \eqref{eq:forward_map_lebesgue_law},\eqref{eq:forward_map_lebesgue} is a well-defined function and Assumption \ref{assume:assumptions_IFTh} holds true.  More generally, it can be proved that $\partial_u \RR_s(W,u)$ is a homeomorphism between $\U$ and $R$ if
\begin{align}\label{eq:SC_DuR_homeo_LLG_Lebesgue}
W\in \W, u\in\U : \norm{\Im{u}}{L^{\infty}(D_T)} \leq \frac{1}{4}.
\end{align}
Finally, Assumption \ref{assume:bounds_d1R_d2Rinv} holds true with $\eps_W>0$ and $\eps_u=\frac{1}{4}$ and
\begin{align*}
    \GG_1(s) &= \norm{\bg}{C^{2+\alpha}(D)} \left(1+\norm{\bM^0}{\U}+s\right)^3\\
    \GG_2(s) &=  C_{\rm stab}(\eps+\theta, s)\qquad \forall s\geq0,
\end{align*}
where $C_{\rm stab}(\norm{W}{\W}, \norm{u}{\U})>0$ is as $c_p$ in \cite[Theorem 2.5]{Schnaubelt2011Solvability}, i.e. it guarantees that \\
$\norm{\left(\partial_u \RR(W, u)\right)^{-1} f}{\U} \leq C_{\rm stab}(\norm{W}{\W}, \norm{u}{\U}) \norm{f}{R}$ for any $f\in R$, $W\in\W$, $u\in\U$.
\end{lemma}

We recall that, as shown in Section \ref{sec:param_regularity_general}, the implicit function theorem and Theorem \ref{th:global_holo_ext} prove the existence of $\eps>0$ such that for any $\by\in\XX_\R$ there exists a holomorphic map $\UU:B_{\eps}(W(\by))\rightarrow \U$ such that $\RR(W,\UU(W))=0$ for all $W\in B_{\eps}(W(\by))$. The function $\UU$ is bounded by a constant $C_{\eps}>0$ again independent of $\by$.
Moreover, Assumption \ref{assume:bounds_d1R_d2Rinv} implies the bound \eqref{eq:estim_norm_DU} on the differential $\UU'(W)$ as a function of $\UU(W)$ through $\norm{\UU(W)}{\U}$ for all those $W\in B_{\eps}(W(\by))$ in $\W$.

\subsection{Proof of Assumption \ref{assume:sparsity} and estimates of derivatives of parameter-to-solution map}

Let us now estimate the derivatives of the parameter-to-solution map.
To this end, let us find a real positive sequence $\brho=(\rho_n)_n$ that verifies Assumption \ref{assume:sparsity}. Contrary to Section~\ref{sec:uniform_holo_Holder}, here $\brho$ depends on which mixed derivative $\partial^{\bnu}$ is considered: Given a multi-index $\bnu=(\nu_1,\dots,\nu_n)\in \N_0^n$, ${0<\delta<\frac{1}{2}}$ and $0<\gamma<1$ consider a sequence of positive numbers $\brho=\brho(\bnu, \delta, \gamma)$ defined as follows: $ \forall\ell\in\N_0, j = 1,\dots, \lceil2^{\ell-1}\rceil,$
\begin{align}\label{def:rho_L1}
\rho_{\ell, j} \coloneqq \gamma
\begin{cases}
1\qquad & \textrm{if } \nu_{\ell, j}=0\\
2^{\left(\frac{3}{2}-\delta\right)\ell} \frac{1}{r_{\ell}(\bnu)}\qquad & \textrm{if } \nu_{\ell,j}=1\\
2^{\left(\frac{1}{2}-\delta\right)\ell}\qquad & \textrm{otherwise},
\end{cases}
\end{align}
where we used the hierarchical indexing \eqref{eq:re-indexing_LC} and $r_{\ell}(\bnu) \coloneqq \#\left\{j\in 1,\dots,\lceil 2^{\ell-1}\rceil : \nu_{\ell,j}=1\right\}$. 

\begin{lemma}\label{lemma:sparsity_Lebesgue}
Consider a  multi-index $\bnu=(\nu_1,\dots,\nu_n)\in \N_0^n$, $\delta>0$ and $1<q<\frac{1}{1-\delta/2}$. There exists $0<\gamma<1$ such that defining $\brho= \brho(\bnu, \delta, \gamma)$ as in \eqref{def:rho_L1} verifies Assumption \ref{assume:sparsity}.
\end{lemma}
\begin{proof}
Let $\by\in \XX_{\R}$ and $\bz\in \bB_{\brho}(\by)$. 
 (i.e. $\seminorm{z_n-y_n}{} \leq \rho_n$ for all $n\in\N$).
A triangle inequality yields: 
$ \norm{W(\bz)-W(\by)}{L^q(0,T)} 
\leq \sum_{\ell\in\N_0} \sum_{j=1}^{\lceil 2^{\ell-1}\rceil} \seminorm{z_{\ell,j}-y_{\ell,j}}{}\norm{\eta_{\ell,j}}{L^q(0,T)}.$
For the Faber-Schauder basis functions \eqref{eq:faber-shauder_basis}, $\norm{\eta_{\ell,j}}{L^q(0,T)} \leq 2^{-(1/q+1/2)\ell}$ for any $\ell\in\N_0$ and $j=1,\dots, \lceil2^{\ell-1}\rceil$.
Together with the fact that $\bz\in\bB_{\brho}(\by)$, this gives
\begin{align}\label{eq:linear_Lq_estim_LCE}
\norm{W(\bz)-W(\by)}{L^q(0,T)}
\leq \sum_{\ell\in\N_0} 2^{-(1/q+1/2)\ell} \sum_{j=1}^{\lceil 2^{\ell-1}\rceil} \rho_{\ell, j}.
\end{align}
By the definition of $\brho$, we may write
\begin{align}\label{eq:estim_sum_Lq_rho}
\sum_{j=1}^{\lceil 2^{\ell-1}\rceil} \rho_{\ell, j} 
&= \gamma  \left( 
\#\left\{i: \nu_{\ell, i}=0\right\}
+2^{\left(\frac{3}{2}-\delta\right)\ell} \frac{1}{r_{\ell}(\bnu)} r_{\ell}(\bnu)
+ 2^{\left(\frac{1}{2}-\delta\right)\ell}\#\left\{i: \nu_{\ell, i}>1\right\}
\right).
\end{align}
Trivially, $\#\left\{i: \nu_{\ell, i}=0\right\} \leq 2^{\ell}$ and $\#\left\{i: \nu_{\ell, i}>1\right\} \leq 2^{\ell}$. 
This, together with \eqref{eq:linear_Lq_estim_LCE} and \eqref{eq:estim_sum_Lq_rho} yields
\begin{align*}
\norm{W(\bz)-W(\by)}{L^q(0,T)}
\leq \gamma \sum_{\ell\in\N_0} \big(2^{-(1/q-1/2)\ell} + 2^{-\delta \ell /2} + 2^{-\delta \ell /2}\big).
\end{align*}
Which is finite if $1<q<\frac{1}{1-\delta/2}$.
Thus, there exists $\gamma>0$ such that $W(\bz)\in B_{\eps}(W(\by))$.
\end{proof}
Having so concluded that for any $\by\in\XX_\R$ the parameter-to-solution map ${\MM\circ W:\bB_{\bm{\rho}}(\by)\rightarrow \U}$ is holomorphic and  uniformly bounded, we can estimate its derivatives as in Theorem \ref{th:bound_mix_derivative}.
\begin{proposition}\label{th:bound_mix_derivative_Lebesgue}
Consider $\m = \bM^0+u:\XX_{\R}\rightarrow \bM^0+\U_{\R}$, the parameter-to-solution map of the parametric LLG equation defined in the beginning of this section, where $\XX_{\R}$ and $\U_{\R}$ are defined in \eqref{eq:def_real_param_lebesgue}, \eqref{eq:def_real_sol_Lebesgue} respectively.
Fix $\eps>0$ as in Theorem \ref{th:global_holo_ext}, let $\delta>0$, $1<q<\frac{1}{1-\delta/2}$.
Fix a multi-index $\bnu = \left(\nu_i\right)_{i=1}^n \in \N^n_0$ for $n\in\N$.
Define the positive sequence $\bm{\rho} = (\rho_n)_{n\in\N}$ as in \eqref{def:rho_L1} and choose $0<\gamma<1$ such that Assumption \ref{assume:sparsity} holds.
Then, it holds that
\begin{align}\label{eq:bound_mix_derivatives_Lebesgie}
\norm{\partial^{\bnu} \m(\by)}{\U} 
\leq  \prod_{j=1}^n \nu_j!\rho_j^{-\nu_j} C_{\eps}\qquad \forall \by\in\XX_{\R},
\end{align}
where $C_{\eps}>$ 0 from Theorem \ref{th:global_holo_ext} is independent of $\bnu$ or $\by$.
\end{proposition}

\section{Sparse grid approximation of the parameter-to-solution map}\label{sec:SG}
We briefly recall the sparse grid interpolation construction. 
A complete discussion can be found e.g. in \cite{bungartz2004sparse} or \cite{nobile2008anisotropic}.

Consider \revision{the \emph{level-to-knot function} $m: \N_0 \rightarrow \N$, a strictly increasing function with  $m(0)=1$.
For any $\nu\in\N_0$, define the family of distinct nodes $\YY_{\nu} = \left(y_i^\nu\right)_{i=1}^{m(\nu)}\subset \R$ such that $y^0_1=0$. 
We shall write $y_i$ rather than $y_i^\nu$ when the context is clear.
Let $V_\nu$ denote a suitable $m(\nu)$-dimensional linear space and  
$I_\nu : C^0(\R) \rightarrow V_\nu$ an \emph{interpolation operator over $\YY_\nu$}, i.e. 
$I_\nu[u](y) = u(y)$ for all $y\in \YY_\nu$. 
For any $\nu\in \N_0$, the \emph{detail operator} is
$\Delta_\nu : C^0(\R) \rightarrow V_\nu$ with $\Delta_\nu u = I_{\nu}u - I_{\nu-1}u$,
where we assume $I_{-1} \equiv 0$ so that $\Delta_0 u = I_0 u \equiv u(0)$.}
Denote by $\FF$ the set of \emph{multi-indices} (i.e. integer-valued sequences) with finite support. 
For $\bnu\in\FF$, the corresponding \emph{hierarchical surplus} operator is 
$\Delta_{\bnu}: C^0(\R^{\N})\rightarrow V_{\bnu}$ where $V_{\bnu} = \bigotimes_{n\in\N} V_{\nu_n}$ and $\Delta_{\bnu} = \bigotimes_{n\in\N} \Delta_{\nu_n} = \bigotimes_{n\in\supp{\bnu}} \Delta_{\nu_n}$.
For any downward-closed $\Lambda \subset \FF$, define $V_{\Lambda}\coloneqq\bigoplus_{\bnu\in\Lambda} V_{\bnu}$ and the \emph{sparse grid interpolant}:
\begin{align}\label{eq:def_SG_intepolant}
	\II_{\Lambda}: C^0(\R^{\N}) \rightarrow V_{\Lambda} \qquad 
	\II_{\Lambda}\coloneqq \sum_{\bnu\in \Lambda} \Delta_{\bnu}.
\end{align}
The result~\cite[Theorem 2.1]{Chkifa2014High} shows that there exists a finite set $\HH_{\Lambda}\subset \R^{\N}$, the \emph{sparse grid}, such that 
$\II_{\Lambda} u (\by) = u (\by)$ for all $\by\in \HH_{\Lambda}$  and $\II_{\Lambda} u$ is the unique element of $V_{\Lambda}$ with this property.

With bounds 
$\norm{\Delta_{\bnu} u}{L^{2}_{\bmu}(\R^{\N})} \lesssim v_{\bnu}$ (\emph{value}) for all $\bnu\in\FF$ and 
$\#\HH_{\Lambda} \lesssim \sum_{\bnu\in\Lambda} w_{\bnu}$  (\emph{work}) for all downward closed $\Lambda\subset \FF$,
we may define the optimal $n$-elements multi-index set $\Lambda_n$ by choosing the $n$ multi-indices with maximum \emph{profit} $\mathcal{P}_{\bnu} \coloneqq \frac{v_{\bnu}}{w_{\bnu}}$.
\revision{This leads to the following theorem on sparse grid convergence, which, in the original reference, is proved only for finite dimensional domains. However, the proof applies verbatim to the infinite dimensional case since $\FF$ is countable.}
\begin{theorem}{\cite[Theorem 1]{nobile2016convergence}}\label{th:general_conv_th_SG}
If there exists $\tau \in (0,1]$ such that
\begin{align*}
	C_{\tau} \coloneqq \left(\sum_{\bnu\in\FF} \mathcal{P}_{\bnu}^{\tau} w_{\bnu}\right)^{1/\tau}  < \infty,
\end{align*}
then
\begin{align*}
	\norm{u- \II_{\Lambda_n} u}{L^2_{\bmu}(\R^{\N})} \leq C_{\tau} \#\HH_{\Lambda_n}^{1-1/\tau}.
\end{align*}
\end{theorem}
%In the rest of the section, we discuss the sparse grid methods defined using piecewise polynomial interpolation.

\subsection{1D piecewise polynomial interpolation on \texorpdfstring{$\mathbb{R}$}{R}}\label{sec:1D_interpol}
Let $\mu(x; \sigma^2) = \frac{1}{\sqrt{2\pi\sigma^2}} e^{-x^2/{2\sigma^2}}$ denote the normal density with mean zero and variance $\sigma^2>0$. Let $\mu(x)=\mu(x; 1)$ and $\tmu(x) =\mu(x; \sigma^2)$ for some fixed $\sigma^2>1$.
Consider the error function $\erf(x) = \frac{2}{\sqrt{\pi}}\int_0^x e^{-t^2} \rmd t$.

Define the \emph{level-to-knots function}
\begin{align}\label{eq:lev2knot_pw}
	m(\nu) \coloneqq 2^{\nu+1}-1\qquad \forall \nu\in \N_0.
\end{align}
Let $\YY_\nu = \left\{ y_1,\dots, y_{m(\nu)}\right\}\subset \R$, where
\begin{align}\label{eq:nodes_pwPoly}
	y_i &= \phi\left(-1+\frac{i}{m(\nu)+1}\right)\qquad \forall i=1,\dots,m(\nu),\\
	\label{eq:def_nodes_transform}
	\phi(x) &\coloneqq \alpha \erf^{-1}(x) \qquad \forall x\in(-1,1),\\ 
	\label{eq:const_nodes_transform}
	\alpha &= \alpha(p,\sigma^2) \coloneqq \sqrt{\frac{4p}{1-\frac{1}{\sigma^2}}}.
\end{align}
The $m(\nu)$ nodes define $m(\nu)+1$ intervals (the first and last are unbounded).
The nodes families are nested, i.e. $\YY_{m(\nu)}\subset \YY_{m(\nu+1)}$ for all $\nu\in\N_0$. 

As 1D interpolant, we consider the following continuous piecewise polynomial interpolant:
When $\nu=0$, consider the constant interpolation in $y=0$, i.e.
\begin{align*}
	I_0 [u] \equiv u(0).
\end{align*}
Now fix $p\in\N$, $p\geq 2$.
When $\nu \geq 1$, $I_\nu[\cdot]$ is the continuous piecewise polynomial interpolant of degree $p-1$ over the intervals defined by $\YY_\nu$. 
More precisely,
\begin{align*}
	&I_\nu[u](y_i) = u(y_i)\qquad &&\forall i=1,\dots,m(\nu),\\
	&I_\nu[u]|_{[y_i, y_{i+1}]} \textrm{ is a degree  $p-1$ polynomial} \qquad &&\forall i=1,\dots, m(\nu)-1,\\
	&I_\nu[u](y) \textrm{ polynomial extension of } I_\nu[u]|_{[y_1, y_2]}\qquad && \textrm{if } y \leq y_1,\\
	&I_\nu[u](y) \textrm{ polynomial extension of } I_\nu[u]|_{[y_{m(\nu)-1}, y_{m(\nu)}]}\qquad &&\textrm{if } y \geq y_{m(\nu)}.
\end{align*}
We assume that for each $i=1,\dots, m(\nu)-1$, the interval $(y_i, y_{i+1})$ contains additional $p-2$ distinct interpolation nodes so that $I_\nu[u]$ is uniquely defined.

The function $\phi$~\eqref{eq:def_nodes_transform} is such that $\left(\phi'(x)\right)^{2p} \tmu^{-1}(\phi(x)) \mu(\phi(x))$ is constant in $x$ and equals 
\begin{align}\label{eq:def_Cphi}
C_{\phi} = \sqrt{\sigma^2} \left(\frac{\alpha\sqrt{\pi}}{2}\right)^{2p},
\end{align}
%where $\alpha$ was defined in \eqref{eq:const_nodes_transform}.

\revision{Appendix~\ref{appendix:SG_convergence} proves some standard convergence estimates for the 1D interpolation.
The following proposition shows estimates for the hierarchical surpluses in the $L^2_{\bmu}$-norm (as in Section~\ref{sec:from_SPDE_to_paramPDE_general}, $\bmu$ denotes the standard Gaussian measure on $\R^{\N}$).}
\begin{proposition}\label{prop:estim_Deltas} 
Let $u:\R^{\N}\rightarrow \R$, $p\geq 2$ and $\bnu\in\FF$. Then 
\begin{align*}
\norm{\Delta_{\bnu}[u]}{L^2_{\bmu}(\R^\N)} \leq 
\left(\prod_{i:\nu_i=1} C_1 \right)
\prod_{i:\nu_i>1} \left(\frac{C_2 2^{-p\nu_i} }{p!}\right)
\norm{\partial_{\left\{i: \nu_i=1\right\}} \partial_{\left\{i:\nu_i>1\right\}}^p u}{L^2_{\tilde{\bmu}}(\R^{\N})},
\end{align*}
where 
\revision{$\tilde{\bmu}$ denotes the infinite product measure $\tilde{\bmu}\coloneqq \bigotimes_{n\in\N} \tmu_n$ and $\tmu_n=\tmu$ for all $N\in\N$, }
$u$ is understood to be sufficiently regular for the right-hand side to be well-defined, and 
$C_1, C_2>0$ are constants defined in the previous lemma.
\end{proposition}
\begin{proof}
Assume without loss of generality that all components of $\bnu$ are non-zero except the first $N\in\N$. 
Then, denoting by $\mu^N$ the $N$-dimensional standard Gaussian measure, by $\hat{\bnu}_1 = (\nu_2,\dots, \nu_N)$ and $\hat{\mu}_1$ the $(N-1)$-dimensional standard Gaussian measure, 
\begin{align*}
	\norm{\Delta_{\bnu}[u]}{L^2_{\bmu}(\R^\N)}^2
	=\int_{\R^{N}} \seminorm{\Delta_{\bnu}[u]}{}^2 \rmd \mu^N
	=\int_{\R^{N-1}} \int_{\R} \seminorm{\Delta_{\nu_1}[y_1\mapsto \Delta_{\hat{\bnu}_1} u ]}{}^2 \rmd \mu_1 \rmd \hat{\mu}_1.
\end{align*}
We apply Lemma~\ref{lemma:estimate_details} (assume that $\nu_1=1$, the other case is analogous) to get
\begin{align*}
\norm{\Delta_{\bnu}[u]}{L^2_{\bmu}(\R^\N)}^2
\leq \int_{\R^{N-1}} C_1^2\int_{\R} \seminorm{ \partial_1 \Delta_{\hat{\bnu}_1} [u] }{}^2 \rmd \tmu_1 \rmd \hat{\mu}_1.
\end{align*}
Exchanging the integrals as well as the operators acting on $u$ shows
\begin{align*}
\norm{\Delta_{\bnu}[u]}{L^2_{\bmu}(\R^\N)}^2
\leq C_1^2\int_{\R}\int_{\R^{N-1}}  \seminorm{  \Delta_{\hat{\bnu}_1} [\partial_1 u] }{}^2  \rmd \hat{\mu}_1\rmd \tmu_1.
\end{align*}
We can iterate this procedure $N-1$ additional times to obtain the statement.
\end{proof}

\subsection{Basic profits and dimension dependent convergence} \label{sec:midset_selection_theoretical}
In this section, we discuss the convergence of sparse grid approximation when the sample paths of Wiener processes and magnetizations are assumed to be Hölder-continuous. To this end, we apply the results found in Section \ref{sec:holomoprhic_regularity_Holder}.
\revision{
Let us begin by estimating the norm of hierarchical surpluses.
\begin{proposition}
Denote by $u:\XX_\R\rightarrow \U_\R$ the parameter-to-solution map of the (parametric) SLLG equation as in Section~\ref{sec:holomoprhic_regularity_Holder}.
Denote, for a finite support multi-index $\bnu\in\FF$, the corresponding hierarchical surplus operator $\Delta_{\bnu}$ as defined in the beginning of Section~\ref{sec:SG} using 1D piecewise polynomial interpolation of degree $p-1$ with $p\geq 2$. 
Denoting by $\bmu$ the standard Gaussian measure on $\R^{\N}$, there holds:
\begin{align*}
	\norm{\Delta_{\bnu}[u]}{L^2_{\bmu}(\R^\N)} 
	\lesssim \prod_{i\in\supp(\bnu)} \tilde{v}_{\nu_i}\quad\text{with}\quad \tilde{v}_{\nu_i}
	= \begin{cases}
		C_1 \rho_i^{-1} \qquad &\textrm{if } \nu_i=1\\
		C_2 \left(2^{\nu_i} \rho_i \right)^{-p} \qquad &\textrm{if } \nu_i>1,
	\end{cases}
\end{align*}
and $\rho_i = \eps 2^{\frac{(1-\alpha)\lceil\log_2(i)\rceil}{2}}$ for all $i\in\N$, as in~\eqref{eq:def_rho_n}. The hidden constant is independent of $\bnu$.
\end{proposition}
\begin{proof}
	We apply Proposition~\ref{prop:estim_Deltas} to estimate
	\begin{align*}
		\norm{\Delta_{\bnu}[u]}{L^2_{\bmu}(\R^\N)} 
		\leq 
		\left(\prod_{i:\nu_i=1} C_1 \right)
		\left(\prod_{i:\nu_i>1} \frac{C_2 2^{-p\nu_i} }{p!}\right)
		\norm{\partial_{\left\{i: \nu_i=1\right\}} \partial_{\left\{i:\nu_i>1\right\}}^p u}{L^2_{\tilde{\bmu}}(\R^{\N})}.
	\end{align*}
	Theorem~\ref{th:bound_mix_derivative} allows us to estimate the derivatives as:
	\begin{align*}
		\norm{\partial_{\left\{i: \nu_i=1\right\}} \partial_{\left\{i:\nu_i>1\right\}}^p u}{L^2_{\tilde{\bmu}}(\R^{\N})}
		\lesssim \left(\prod_{i:\nu_i=1} \rho_i^{-1} \right)
		\left(\prod_{i:\nu_i>1} p! \rho_i^{-p} \right).
	\end{align*}
	Combining the two estimates gives the statement.
\end{proof}
}
Recall the framework presented at the beginning of Section \ref{sec:SG}. 
Given a multi-index $\bnu\in\FF$, we define as its value and work respectively
\begin{align}\label{eq:def_value}
\tilde{v}_{\bnu} &= \prod_{i\in\supp(\bnu)} \tilde{v}_{\nu_i},\\
\label{eq:def_work}
w_{\bnu} &= \prod_{i\in\supp(\bnu)} p 2^{\nu_i}.
\end{align}
The definition of work is justified as follows: From the definition of 1D nodes \eqref{eq:nodes_pwPoly} and level-to-knots function \eqref{eq:lev2knot_pw}, each time a multi-index is added to the multi-index set, the sparse grid gains $(2^{\nu_i+1}-2)(p-1)+1$ new nodes in the $i$-th coordinate.

Recall that the profit is the ratio of value and work. In this case, it reads
\begin{align}\label{eq:basic_profit}
\tilde{\PP}_{\bnu} = \frac{\tilde{v}_{\bnu}}{w_{\bnu}}.
\end{align}

We apply the convergence Theorem \ref{th:general_conv_th_SG} to obtain a convergence rate that depends root-exponentially on the number of approximated parameters.
\begin{theorem}\label{th:conv_SG_C}
	Let $N\in\N$ and denote $\m_N:\R^N\rightarrow C^{1+\alpha/2, 2+\alpha}(D_T)$ the parameter-to-solution map from  the parametric LLG equation under the assumption that, for all $t\in[0,T]$ and all $\by\in\R^{\N}$,
	$W(\by, t)= \sum_{i=1}^N y_i \eta_i(t)$.
	Let $\Lambda_n\subset \N^{N}_0$ denote the optimal multi-index set with $\#\Lambda_n = n$ with respect to $\tilde{\mathcal{P}}_{\bm{\nu}}$ from~\eqref{eq:basic_profit}. Let $\II_{\Lambda_n}$ denote the corresponding piecewise polynomial sparse grid interpolant of degree $p-1$ with nodes \eqref{eq:nodes_pwPoly} and $p\geq 2$. Denote $\HH_{\Lambda_n}\subset \R^N$ the corresponding sparse grid.
	Under the assumptions of Theorem~\ref{th:Holder_sample_paths}, for any $\frac{2}{(1+\alpha)p} < \tau < 1$,
	\begin{align}\label{eq:conv_pwPoly_SG}
		& \norm{\m_N- \II_{\Lambda_n} \m_N}{L^2_{\bmu}(\R^N, C^{1+\alpha/2, 2+\alpha}(D_T))} \leq C_{\tau,p}(N) \left(\#\HH_{\Lambda_n}\right)^{1-1/\tau},
	\end{align}
	where $C_{\tau,p}(N)$ is a function of $\tau$, $p$, $N$ defined as  
	\begin{align*}
		C_{\tau,p}(N) = 
		(1+P_0)^{1/\tau}
		\exp\frac{1}{\tau}\left( 
		\frac{C_1^{\tau} \left(2p\right)^{1-\tau}}{2} \frac{1-N^{(1-(1-\alpha)\tau/2)}}{1-2^{1-(1-\alpha)\tau/2}}
		+ \frac{C_2^{\tau} \sigma(p,\tau)}{2} \frac{1}{1-2^{1-(1-\alpha)p\tau/2}}
		 \right),
	\end{align*}
	where
	$P_0 = C_1^{\tau} (2p)^{1-\tau} + C_2^{\tau} p^{1-\tau}\sigma(p,\tau)$, 
	$\sigma(p,\tau) = \frac{2^{2(1-\tau(p+1))}}{1-2^{1-\tau(p+1)}}$ and $C_1, C_2$ were defined in Lemma \ref{lemma:estimate_details}. In particular, the bound grows root-exponentially in the number of dimensions.
\end{theorem}
\revision{ 
\begin{proof}
With the aim of applying the convergence Theorem \ref{th:general_conv_th_SG}, we estimate:
\begin{align*}
	\sum_{\bnu\in\N_0^N}\PP_{\bnu}^{\tau} w_{\bnu} & =
	\sum_{\bnu\in\N_0^N}v_{\bnu}^{\tau} w_{\bnu}^{1-\tau} \leq
	\prod_{i=1}^N \sum_{\nu_i\geq 0}  v_{\nu_i}^{\tau} w_{\nu_i}^{1-\tau} \\
	& = \prod_{i=1}^N \left( 1 + \left(C_1 \rho_i^{-1}\right)^{\tau} \left(2p\right)^{1-\tau}+
	\sum_{\nu_i\geq 2} \left(C_2 (2^{\nu_i}\rho_i)^{-p}\right)^{\tau} \left(p2^{\nu_i}\right)^{1-\tau}
	\right).
\end{align*}
The remainder of the proof consists of estimating the product under the condition on $\tau$. See Appendix~\ref{appendix:SG_convergence} for details. 
\end{proof}
}

\subsection{Improved profits and dimension independent convergence}\label{sec:improved_profit}
In the previous section, we could prove only a dimension-\emph{dependent} convergence. This may be attributed to the slow growth of the holomorphy radii $\rho_i \lesssim 2^{\frac{(1-\alpha)\ell(i)}{2}}$. 
Let us consider the setting from Section \ref{sec:holomorphic_regularity_L1_small}, in which we assumed small Wiener processes and a coefficient $\bg$ with small gradient. With these modelling assumptions, we proved that the holomorphy radii can be chosen as \eqref{def:rho_L1}. This will be sufficient to obtain dimension-\emph{independent} convergence.
Again we work within the framework described at the beginning of Section \ref{sec:SG}.

We need to define \emph{values} that, for any $\bnu\in\FF$, bound $\norm{\Delta_{\bnu}u}{L^2_{\bmu}(\R^{\N}, \U)}$ from above. 
The estimates from Proposition \ref{prop:estim_Deltas} and the estimate on the derivatives from Proposition \ref{th:bound_mix_derivative_Lebesgue} motivate the following choice of values:
\begin{align*}
v_{\bnu} = \prod_{i\in\supp(\bnu)} v_{\nu_i},
\qquad \textrm{where } v_{\nu_i} =
\begin{cases}
C_1 \rho_i^{-1} \qquad &\textrm{if } \nu_i=1\\
C_2 \left(2^{\nu_i} \rho_i \right)^{-p} \qquad &\textrm{if } \nu_i>1
\end{cases}
\end{align*}
and 
\begin{align*}
\rho_i = \rho_{\ell, j} \coloneqq \gamma
\begin{cases}
2^{\left(\frac{3}{2}-\delta\right)\ell} \frac{1}{r_{\ell}(\bnu)}\qquad & \textrm{if } \nu_{\ell,j}=1\\
2^{\left(\frac{1}{2}-\delta\right)\ell}\qquad & \textrm{otherwise}
\end{cases}.
\end{align*}
Here, $i$ and $(\ell,j)$ are related through the hierarchical indexing \eqref{eq:re-indexing_LC}, $\delta>0$ is small and for any $\ell\in\N_0$, $\bnu\in\FF$,
$r_{\ell}(\bnu) = \#\left\{j\in \left\{ 1,\dots, \lceil 2^{\ell-1}\rceil \right\}: \nu_{\ell,j}=1\right\}$.
% We define the work and profit functions as in the previous section.
With the \emph{work} defined as in \eqref{eq:def_work}, the profits now read
\begin{align}\label{eq:improved_profit}
\PP_{\bnu} = \frac{v_{\bnu}}{w_{\bnu}}.
\end{align}
Let us determine for which $\tau\in (0,1)$ the sum $\sum_{\bnu\in\FF} v_{\bnu}^{\tau} w_{\bnu}^{1-\tau}$ is finite. This setting is more complex than the one in the previous section because the factors $v_{\nu_i}$ that define the values $v_{\bnu}$ depend in general on $\bnu$ rather than $\nu_i$ alone. 
Define
\begin{align*}
\FF^*\coloneqq  \left\{\bnu\in\FF: \nu_i \neq 1\  \forall i\in\N\right\}
\end{align*}
and for any $\bnu\in\FF^*$ 
\begin{align*}
K_{\bnu}\coloneqq \set{\hat{\bnu}\in\FF}{
\hat{\nu}_i = \nu_i \textrm{ if } \nu_i>1 \text{ and }\hat{\nu}_i \in \left\{0,1\right\} \textrm{ if } \nu_i=0 
}.
\end{align*}
The family $\left\{K_{\bnu}\right\}_{\bnu\in\FF^*}$ is a partition of $\FF$.
As a consequence,
\begin{align}\label{eq:estim_conv_SG_L1_split}
\begin{split}
\sum_{\bnu\in\FF} v_{\bnu}^{\tau} w_{\bnu}^{1-\tau}
= \sum_{\bnu\in\FF^*}\sum_{\hbnu\in K_{\bnu}} v_{\hbnu}^{\tau}w_{\hbnu}^{1-\tau}
%= \sum_{\bnu\in\FF^*}\sum_{\hbnu\in K_{\bnu}} \prod_{i:\hat{\nu}_i\leq1 } \left(v_{\hat{\nu}_i}^{\tau}w_{\hat{\nu}_i}^{1-\tau}\right)
%\prod_{i:\hat{\nu}_i>1} \left(v_{\hat{\nu}_i}^{\tau}w_{\hat{\nu}_i}^{1-\tau}\right)
= \sum_{\bnu\in\FF^*} \prod_{i:\nu_i>1} \left(v_{\nu_i}^{\tau}w_{\nu_i}^{1-\tau}\right)
\sum_{\hbnu\in K_{\bnu}} \prod_{i:\hat{\nu}_i\leq1 } \left(v_{\hat{\nu}_i}^{\tau}w_{\hat{\nu}_i}^{1-\tau}\right).
\end{split}
\end{align}
Consider the following subset of $\FF$:
\begin{align*}
\FFd \coloneqq K_{\b{0}} = \set{\bnu\in\FF}{\nu_i\in\left\{0,1\right\}\forall i\in\N}.
\end{align*}
The following two technical lemmata are proved in Appendix~\ref{appendix:SG_convergence}.
\begin{lemma}\label{lemma:summability_in_F_01}
Let $0<p< 1 $, $ p<q<\infty$, and the sequence $\bm{a} = (a_j)_{j\in\N} \in \ell^p(\N)$. Then,
\begin{align*}
\left( \seminorm{\bnu}{1}!\ \bm{a}^{\bnu} \right)_{\bnu\in\FFd} \in \ell^q(\FFd).
\end{align*}
\end{lemma}
\begin{lemma}\label{lemma:finite_factor_product_1}
If $\tau > \frac{1}{\frac{3}{2}-\delta}$, there exists $C>0$ such that for any $\bnu\in\FF^*$,
\begin{align*}
\sum_{\hbnu\in K_{\bnu}} \prod_{i:\hat{\nu}_i\leq1 } \left(v_{\hat{\nu}_i}^{\tau}w_{\hat{\nu}_i}^{1-\tau}\right) \leq C.
\end{align*}
\end{lemma}

Going back to \eqref{eq:estim_conv_SG_L1_split}, we are left with determining for which parameters $p\geq 3, \tau > \frac{1}{\frac{3}{2}-\delta}$ the series $\sum_{\bnu\in\FF^*} \prod_{i:\nu_i>1} \left(v_{\nu_i}^{\tau}w_{\nu_i}^{1-\tau}\right)$ is summable.
By means of the product structure of the summands, we can write
\begin{align*}
\sum_{\bnu\in\FF^*} \prod_{i:\nu_i>1} \left(v_{\nu_i}^{\tau}w_{\nu_i}^{1-\tau}\right)
&= \prod_{i\in\N} \sum_{\nu_i\in\N\setminus \left\{1\right\}} v_{\nu_i}^{\tau} w_{\nu_i}^{1-\tau} 
= \prod_{i\in\N} \big(1+ \sum_{\nu_i\geq 2}\left( C_2 2^{-p\left(\left(\frac{1}{2}-\delta\right)\ell(i)+\nu_i\right)}\right)^\tau\left(p2^{\nu_i}\right)^{1-\tau}\big).
\end{align*}
Observe that the sum is finite if $\tau\geq \frac{1}{p+1}$ and in this case
\begin{align*}
\sum_{\nu_i\geq 2} \left(C_2 2^{-p\left(\left(\frac{1}{2}-\delta\right)\ell(i)+\nu_i\right)}\right)^{\tau}\left(p2^{\nu_i}\right)^{1-\tau}
= C_2^{\tau} 2^{-p\left(\frac{1}{2}-\delta\right)\ell(i)\tau} p^{1-\tau} \sigma,
\end{align*}
where $\sigma = \sigma(p,\tau) = 
\frac{2^{2(-(p+1)\tau+1)}}{1 - 2^{-(p+1)\tau+1}}$. 
To summarize, denoting 
$F_{\ell} \coloneqq C_2^{\tau} 2^{-p\left(\frac{1}{2}-\delta\right)\ell\tau} p^{1-\tau} \sigma$, 
so far we have estimated
$\sum_{\bnu\in\FF^*} \prod_{i:\nu_i>1} \left(v_{\nu_i}^{\tau}w_{\nu_i}^{1-\tau}\right)
\leq \prod_{i\in\N} \left(1+ F_{\ell(i)}\right)$.
We can further estimate, recalling the hierarchical indexing \eqref{eq:re-indexing_LC},
\begin{align*}
\prod_{i\in\N} \left(1+ F_{\ell(i)}\right)
\leq \exp\left(\sum_{i\in\N} \log\left(1+F_{\ell(i)}\right)\right)
\leq \exp\left(\sum_{\ell\in\N_0} 2^{\ell} \log\left(1+F_{\ell}\right)\right)
\leq \exp\left(\sum_{\ell\in\N_0} 2^{\ell} F_{\ell}\right).
\end{align*}
The last sum can be written as
$
\sum_{\ell\in\N_0} 2^{\ell} F_{\ell}
= C_2^{\tau} p^{1-\tau} \sigma \sum_{\ell\in\N_0} 2^{\left(1-\left(\frac{1}{2}-\delta\right)p\tau\right)\ell}$, which is finite for $\tau > \frac{1}{p\left(\frac{1}{2}-\delta\right)}$ and in this case equals $C_2^{\tau} p^{1-\tau} \sigma \left(1-2^{1-\left(\frac{1}{2}-\delta\right)p \tau}\right)^{-1}$.

\begin{remark}\label{rem:dim}
When $p=2$ the condition $\tau > \frac{1}{p\left(\frac{1}{2}-\delta\right)}$ just above gives $\tau > 1$ for any $\delta>0$. This means that we are not able to show that piecewise \emph{linear} sparse grid interpolant converges independently of the number of dimensions (although we see it in the numerical experiments below).
Conversely, if $p\geq3$ there exists $\frac{2}{3} < \tau < 1$ that satisfies all the conditions (remember that while $\delta$ cannot be 0, it can be chosen arbitrarily small).
\end{remark}

Finally, Theorem \ref{th:general_conv_th_SG} implies the following convergence.
\begin{theorem}\label{th:conv_SG_L1}
Consider the parameter-to-solution map of the random LLG equation $\m = \bM^0+u$ as in Section \ref{sec:holomorphic_regularity_L1_small}. Recall that $\bM^0\in C^{2+\alpha}(D)$ is the initial condition and $u:\XX_{\R}\rightarrow \U_{\R}$ with $\XX_{\R}$ and $\U_{\R}$ defined in \eqref{eq:def_real_param_lebesgue} and \eqref{eq:def_real_sol_Lebesgue} respectively.
Let $\Lambda_n\subset \FF$ denote the optimal multi-index set with $\#\Lambda_n = n$ defined using the profits $\mathcal{P}_{\nu}$ \eqref{eq:improved_profit}. Let $\II_{\Lambda_n}$ denote the corresponding piecewise polynomial sparse grid interpolant of degree $p-1$ for $p\geq 3$ with nodes \eqref{eq:nodes_pwPoly}. Assume that the corresponding sparse grid satisfies $\HH_{\Lambda_n}\subset \XX_{\R}$.
Under the assumptions of Theorem~\ref{th:Holder_sample_paths}, for any $\frac{2}{3} < \tau < 1$,
\begin{align*}
& \norm{\m- \II_{\Lambda_n} \m}{L^2_{\bmu}(\XX_{\R}; \U)} \leq C_{\tau,p} \left(\#\HH_{\Lambda_n}\right)^{1-1/\tau},
\end{align*}
where $C_{\tau, p}$ is dimension independent and defined as 
\begin{align*}
C_{\tau, p} =
C^{\frac{1}{\tau}}
\exp\left( \frac{1}{\tau}
C_2^{\tau}
p^{1-\tau}
\frac{2^{2(-(p+1)\tau+1)}}{1 - 2^{-(p+1)\tau+1}} 
\frac{1}{1-2^{1-\left(\frac{1}{2}-\delta\right) p\tau}}
\right),
\end{align*}
where in turn $C$ is defined in Lemma \ref{lemma:finite_factor_product_1} and $C_2$ is defined in Lemma \ref{lemma:estimate_details}.
\end{theorem}

\begin{remark}[Optimality of the convergence rate $\frac{1}{2}$]\label{rk:optimality_improved_profit_SG}
\revision{We expect the optimal approximation rate of any collocation-type method to be $\frac{1}{2}$ with respect to the number of collocation nodes.
This is because the LCE is a piecewise affine approximation of the Wiener process sample paths, which are only Hölder-regular with any Hölder exponent less than $\frac{1}{2}$ and hence does not admit a better rate.
Since the parametric LLG equation is not expected to have any smoothing effect in general and since it is not possible to have less than one collocation node per dimension, the sparse grid algorithm also achieves an approximation rate of at most $\frac{1}{2}$.
As a consequence,} piecewise quadratic approximation ($p=3$) has optimal convergence rate and using $p>3$ does not improve the convergence rate (but may improve the constant $C_{\tau, p}$). For the same reason, sparse grid interpolation based on other 1D interpolations schemes (e.g. global polynomials) cannot give a better convergence rate (but may improve the constant).
\end{remark}

\begin{remark}[Approximation of the random field solution of the SLLG equation]
	Given an approximation $\m_{\Lambda}(\by)$ of the solution to the parametric LLG equation \eqref{eq:parametric_LLG_stong} (for example, the sparse grid interpolant studied above), we sample an approximation of the random field solution of the \emph{random} LLG equation \eqref{eq:rnd_LLG_problem} as follows:
	Sample i.i.d. standard normal random variables $\bm{Y} = (Y_i)_{i=1}^{N_{\Lambda}}$ and evaluate $\m_{\Lambda}(Y_1,\dots, Y_{N_{\Lambda}})$.
	Here, $N_{\Lambda}$ is the number of active parameters in the sparse grid interpolant $\II_{\Lambda}$, i.e. $N_{\Lambda} \coloneqq \min{\left\{ n\in\N: \forall \bm{\nu}\in\Lambda\ \supp(\bnu)\subset \left\{1,\dots,n\right\} \right\}}$.
	The root-mean-square error is naturally the same as the one we estimated in the previous theorem:
	$
		\sqrt{\E_{\bm{Y}} \norm{\m(\bm{Y}) - \m_{\Lambda}(\bm{Y})}{\U}}
		= \norm{\m- \II_{\Lambda} \m}{L^{2}_{\bmu}(\XX_{\R}, \U)}.
	$
	
	To draw approximate samples from the random solution of the \emph{stochastic} PDE \eqref{eq:SLLG}, we: 
	\begin{enumerate}
		\item Sample a Wiener process $W(\omega, \cdot)$;
		\item Compute the first $N_{\Lambda}$ coordinates $\bm{Y} = (Y_1, \dots, Y_{N_{\Lambda}}) \in \R^N$ of its LCE~\eqref{eq:def_LCE};
		\item Compute $\m_{\Lambda}(\bm{Y})$, the approximate solution to the random LLG equation \eqref{eq:rnd_LLG_problem};
		\item Finally, compute the inverse Doss-Sussmann transform
		$\bM_{\Lambda}\coloneqq e^{W(\bm{Y}) G} \m_{\Lambda}(\bm{Y})$.
	\end{enumerate}
	For the last step, recall the convenient expression for $e^{WG}$ in \eqref{eq:def_eWG}.The approximation error is again comparable to the one form the previous theorem. 
	Indeed, the Doss-Sussmann transform implies that
	$ \sqrt{\E_{W} \norm{\bM - \bM_{\Lambda}}{\U}}
	= \norm{e^{WG}\left( \m - \m_{\Lambda}\right)}{}$, where we denote by $\norm{\cdot}{}$ the root-mean-square error.
	Identity \eqref{eq:def_eWG} followed by a triangle inequality then gives 
	\begin{align*}
		\norm{\bM - \bM_{\Lambda}}{} &\leq \left(1+\norm{\bg}{L^{\infty}(D)}+\norm{\bg}{L^{\infty}(D)}^2\right)\norm{\m-\m_{\Lambda}}{}.
	\end{align*}
\end{remark}  

\revision{
\begin{remark}[Comparison with Monte Carlo quadrature] \label{rk:Comparison_MC}
	While the sparse grid interpolation does not offer any rate advantages over Monte Carlo to compute statistical quantities, it provides much more information on the solution. 
	For example, it can be used to approximate minima and maxima of scalar quantities of interest of the solution.
\end{remark} 
}
\subsection{Numerical tests}
We numerically test the convergence of the sparse grid interpolation defined above. \
Since no exact sample path of the solution is available, we apply the space and time approximation from \cite{Akrivis2021Highorder}. 
This method is based on the \emph{tangent plane scheme}. In particular, it inherits the advantage of solving one \emph{linear} elliptic problem per time step with finite elements. 
The time-stepping is based on a BDF formula. The method is high-order for both finite elements and BDF discretizations. 

We consider the problem on the 2D domain $D=[0,1]^2$ with $z=0$. The final time is $T=1$.
The noise coefficient is defined as 
\begin{align}\label{eq:g_example}
\bg(\bx) = \left(-\frac{1}{2}\cos(\pi x_1), -\frac{1}{2}\cos(\pi x_2), 
\sqrt{1-\left(\frac{1}{2}\cos(\pi x_1) \right)^2-\left(\frac{1}{2}\cos(\pi x_2) \right)^2} \right).    
\end{align}
Observe that $\partial_n \bg = 0$ on $\partial D$ and that $\vert \bg\vert = 1$ on $D$.
The initial condition is $\bM^0 = (0,0,1)$. 

The space discretization is order 1 on a structured triangular mesh with $2048$ elements and mesh-size $h>0$.
The time discretization is order 1 on $256$ equispaced time steps of size $\tau>0$.
We use piecewise affine sparse grid, corresponding to $p=2$.
As for the multi-index selection, we compare two strategies:
\begin{itemize}
\item The basic profit from Section \ref{sec:midset_selection_theoretical}. We recall that
\begin{align*}
\tilde{\PP}_{\bnu} = \prod_{i:\nu_i=1} 2^{-\frac{1}{2}\ell(i)} \prod_{i:\nu_i>1} \left(2^{\nu_i + \frac{1}{2}\ell(i)} \right)^{-p}\left(\prod_{i:\nu_i\geq1} p2^{\nu_i}\right)^{-1}\qquad\forall\bnu\in\FF,
\end{align*}
where $\ell(i) = \lceil\log_2(i)\rceil$. Compared to \eqref{eq:basic_profit}, here we have set $C_1=C_2=\eps=1$ and $\alpha=0$ for simplicity.
\item A modified version of the improved profit from Section \ref{sec:improved_profit}, namely
\begin{align}\label{eq:profit_L1_numerics}
\PP_{\bnu} = \prod_{i:\nu_i=1} 2^{-\frac{3}{2}\ell(i)} \prod_{i:\nu_i>1} \left(2^{\nu_i+\frac{1}{2}\ell(i)}\right)^{-p}\left(\prod_{i:\nu_i\geq1} p2^{\nu_i}\right)^{-1}\qquad\forall\bnu\in \FF,
\end{align}
where again $\ell(i) = \lceil\log_2(i)\rceil$. Compared to Section~\ref{sec:improved_profit}, we have set $C_1=C_2=\gamma=1$ and neglected the factor $r_{\ell}(\bnu)$.
\end{itemize}

We estimate the approximation error of the sparse grid approximations with the following computable quantity:
$\frac{1}{N} \sum_{i=1}^N \norm{\m_{\tau h}(\by_i) - \II_{\Lambda}[\m_{\tau h}] (\by_i)}{L^{2}(0,T, H^1(D))}$, where $N=1024$, $\left(\by_i\right)_{i=1}^N$ are i.i.d. standard normal samples of dimension $2^{10}$ each and $\m_{hk}(\by_i)$ denotes the corresponding space and time approximation of the sample paths.

Observe that if the time step is $\tau = 2^{-n}$, then the parameter-to-finite element solution map depends only on the first $n+1$ levels of the Lévy-Ciesielski expansion. In our case, $n=8$, so the maximum relevant level is $L=9$, i.e. 512 dimensions. In the following numerical examples we always approximate fewer dimensions, which means that the time-discretization error is negligible compared to the parametric approximation error.
\pgfplotsset{
  every axis label/.append style={font=\small},
  every tick label/.append style={font=\small},
  legend style={nodes={scale=0.9, transform shape}}
}
\begin{figure}
	\hspace{-0.5cm}
	\begin{subfigure}{0.5\textwidth}
		\includegraphics{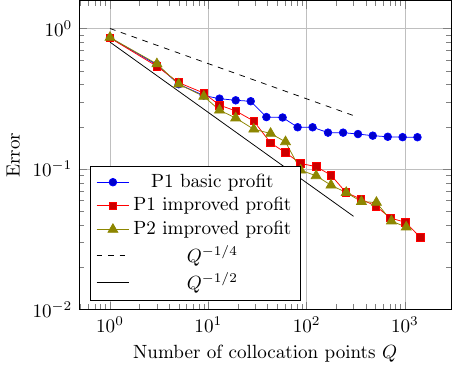}
	\end{subfigure}
	\begin{subfigure}{0.5\textwidth}
		\includegraphics{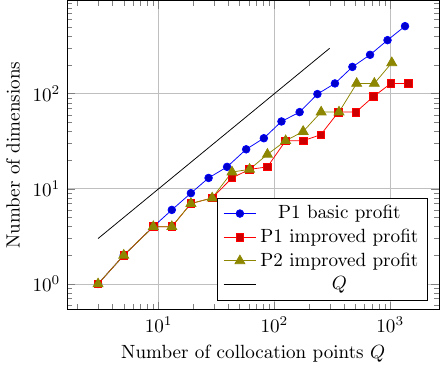}
	\end{subfigure} 
	
\hspace{-0.5cm}
	\begin{subfigure}{0.5\textwidth}
		\includegraphics{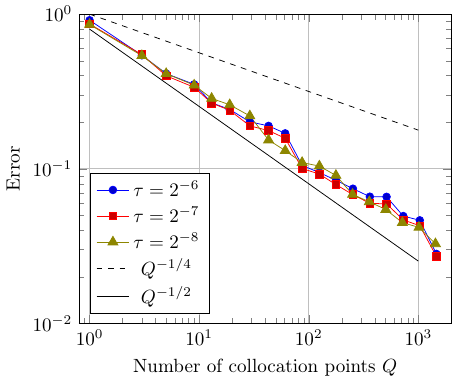}
	\end{subfigure} 
	\begin{subfigure}{\textwidth}
		\includegraphics{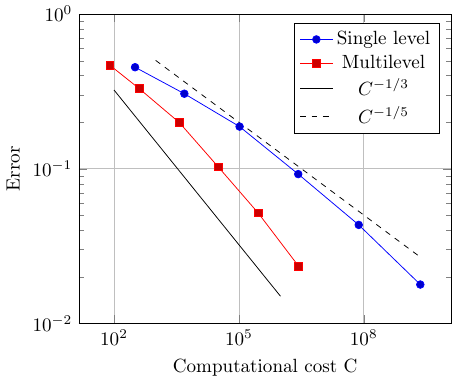}
	\end{subfigure}
	\caption{Approximation of $\by\mapsto\m(\by)$. Top left: Error vs. number of collocation nodes. Top right: Number of effective dimensions vs. number of collocation nodes. Bottom left: Comparison of convergence of the sparse grid approximation ($p=3$, i.e. piecewise quadratic) for different space and time discretization parameters. In all cases time step $\tau$ and mesh size $h$ are related by $h = 8\tau$. Bottom right: Comparison of single- and multilevel approximations based on piecewise polynomial sparse grid interpolation for the parametric approximation and the linearly implicit BDF-finite elements methods from \cite{Akrivis2021Highorder} for time and space approximation.}
	\label{fig:conv_test_SG}
\end{figure}
The results are displayed in Figure \ref{fig:conv_test_SG}.
In the top plot, we observe that using basic profits leads to a sub-algebraic convergence rate which decreases as the number of approximated dimensions increases. 
Conversely, improved profits leads to a robust algebraic convergence of order about $\frac{1}{2}$. 
Piecewise \emph{quadratic} interpolation is optimal as predicted in Section \ref{sec:improved_profit} and it delivers the same convergence rate as piecewise linear interpolation. Hence, the restriction in Theorem \ref{th:conv_SG_L1} is possibly an artifact of the proof.
In view of Remark \ref{rk:optimality_improved_profit_SG}, it seems unnecessary to test higher polynomial degrees.
In the bottom left plot, we observe that the number of active dimensions (i.e., those dimension which are seen by the sparse grid algorithm) grows similarly for all methods, with the basic profit having a slightly higher value.
Finally, we verify numerically that the approximation power of the method does not degrade when space and time approximations are refined, see the bottom left plot in Figure~\ref{fig:conv_test_SG}.

\section{Multilevel sparse grid collocation}\label{sec:ML}
In this section, we show how the sparse grid scheme defined and studied in this work can be combined with a method for space and time approximation to define a fully discrete approximation scheme. Here we employ again the linearly implicit BDF-finite element scheme from \cite{Akrivis2021Highorder}.

Given $\tau>0$, consider $N_{\tau}=\frac{T}{\tau}$ equispaced time steps on $[0,T]$.
Given $h>0$, define a quasi-uniform triangulation $\TT_h$ of the domain $D\in \R^d$ for $d\in\N$ with mesh-spacing $h$.
Denote, for any $\by\in\XX_{\R}$, $\m_{\tau h}(\by)$ the space and time approximation of $\m(\by)$.
Assume that there exists a constant $C_{\rm FE}>0$ independent of $h$ or $\tau$ such that
\begin{align*}
\normT{\m - \m_{\tau,h}} \leq C_{\rm FE} (\tau+h).
\end{align*}
Moreover, we assume that the computational cost (number of floating-point operations) of computing a single $\m_{\tau h}(\by)$ is proportional to
\begin{align*}
C_{\rm sample}(\tau, h) = \tau^{-1} h^{-d}.
\end{align*}
Indeed, the numerical scheme requires, at each time step, solving a linear system of size proportional to the number of elements of $\TT_h$, which in turn is proportional to $h^{-d}$. The latter operation can be executed with empirical linear complexity using GMRES with multigrid preconditioning. See~\cite{precond}, for a mathematically rigorous preconditioning strategies for LLG.

Theorem~\ref{th:conv_SG_L1} shows that there exists $C_{\rm SG}>0$ and  $0 < r < \frac{1}{2}$ such that, denoting $\II_{\Lambda}$ the sparse grid interpolant and $\HH_{\Lambda}$ the corresponding sparse grid,
\begin{align*}
\normT{\m - \II_{\Lambda}[\m]} \leq C_{\rm SG} \left(\# \HH_{\Lambda}\right)^{-r}.
\end{align*}

A \emph{Multilevel} approximation of $\m$ can be defined following \cite{teckentrup2015multilevel}. Let $K\geq 0$ and consider a sequence of approximation parameters $\left(\Lambda_k\right)_{k=0}^K$, $\left(\tau_k\right)_{k=0}^K$ and $\left(h_k\right)_{k=0}^K$. 
Denote $\m_k = \m_{\tau_k, h_k}$ for $0\leq k\leq K$ and $\m_{-1}\equiv 0$. Define the multilevel approximation as
\begin{align*}
\m_K^{\rm ML} \coloneqq \sum_{k=0}^K \II_{\Lambda_k}\left[ \m _{K-k} - \m _{K-k-1}\right].
\end{align*}
The computational cost is proportional to 
$ C^{\rm ML}_{K} = \sum_{k=0}^K \#\HH_{\Lambda_k} C_{\rm sample}(\tau_{K-k}, h_{K-k})$.
As suggested in~\cite{teckentrup2015multilevel} we choose the multi-index sets $\Lambda_k$ such that
\begin{align}\label{eq:ML_choice_SG}
\left(\#\HH_{\Lambda_{K-k}}\right)^{-r} \leq C_{\rm FE} \left(C_{\rm SG} (K+1)\right)^{-1} \frac{\tau_K+h_K}{\tau_k+h_k},
\end{align}
which implies the error-vs-cost estimate
\begin{align}\label{eq:ML_convergence_rate}
\normT{\m - \m_{K}^{\rm ML}} \lesssim \left(C^{\rm ML}_{K}\right)^{-\frac{1}{d+1}}.
\end{align}
Note the significant improvement over the single-level case with $-\frac{1}{\frac{1}{r}+(d+1)}$ in the exponent.

We compare numerically single- and multilevel schemes on the following example of relaxation dynamics with thermal noise. The domain is $D=[0,1]^2$ with $z=0$. The final time is $T=1$. The noise coefficient $\bg$ is set to one fifth of the coefficient defined in~\eqref{eq:g_example}.
The initial condition $\bM^0$ coincides with \eqref{eq:g_example}.
%%%%%%%%%%%%%%%%%%%%%%%%%%%%
The time and space approximations are both of order 1. The sparse grid scheme is piecewise linear and the multi-index sets are built using the improved profit \eqref{eq:profit_L1_numerics} from the previous numerical experiments. 
Observe that, in the following convergence tests, refinement leads automatically to an increase of the number of approximated parameters and a reduction of the parametric \emph{truncation} error.
We consider $0\leq K\leq 5$ and define  $\tau_k = 2^{-k-2}$, $h_k = 2^{-k}$, and $\Lambda_k$ using the same profit-maximization as in the previous section.

For the single-level approximation, we choose $\Lambda_k$ minimal such that $\#\HH_{\Lambda_k} > 2^{2k}$. This choice corresponds to assuming that the sparse grid approximation converges with order $r=\frac{1}{2}$ with respect to the number of collocation nodes.
We compute a sequence of single-level approximations $\m^{\rm SL}_{\Lambda_k, \tau_k, h_k}$ for $k=0,\dots K$ and show the results in the bottom-right of Figure~\ref{fig:conv_test_SG}.

For the multilevel approximation, we follow formula \eqref{eq:ML_choice_SG}. The constants $C_{\rm FE}\approx 0.7510$, $C_{\rm SG}\approx 0.1721$ and $r\approx 0.4703$ are determined with short sparse grid and finite element convergence tests. We obtain:
\begin{center}
\begin{tabular}{c | c c c c c c}
K & $\#\HH_{\Lambda_0}$ & $\#\HH_{\Lambda_1}$ & $\#\HH_{\Lambda_2}$ & $\#\HH_{\Lambda_3}$ & $\#\HH_{\Lambda_4}$ & $\#\HH_{\Lambda_5}$\\
\hline
0 & 1\\
1 & 1 & 3\\
2 & 1 & 3 & 10\\
3 & 1 & 4 & 18 & 82\\
4 & 2 & 7 & 27 & 131 & 602\\
5 & 2 & 10 & 42 & 193 & 887 & 1500\\
\end{tabular}
\end{center}
\revision{The cardinality of $\HH_{\Lambda_5}$ for $K=5$ should actually be at least 4082 if formula~\eqref{eq:ML_choice_SG} is used. Here, we reduce it to 1500 in order to guarantee reasonable computational times.}

Since the solution in closed form is not available, we approximate it with a reference solution. We consider $128$ Monte Carlo samples of $W$ and approximate the corresponding sample paths in space and time with time step $\tau_{\rm ref} = 2^{-9}$ and mesh size $h_{\rm ref} = 2^{-7}$. 

%%%%%%%%%%%%%%%%%%%%%%%%%% APPENDIX %%%%%%%%%%%%%%%%%%%%%%%%%%
\appendix
\section{Definitions of Hölder spaces and proofs of lemmata for space and time Hölder regularity proof}\label{appendix:Holder_regularity}
%%%%%%%%%%%%%%%%%%%%%%%%%%%%%%%%%%%%%%%%%% DEFINITIONS %%%%%%%%%%%%%%%%%%%%%%%%%%%%%%%%%%%%%%%%%%%%%%% 
We recall basic definitions, notation, and important facts about Hölder spaces. 
Let $n\in\N$, $D\subset \R^n$, $\alpha\in(0,1)$,  $v:D\rightarrow \C$. The Hölder-seminorm reads 
\begin{align*}
	\vert v\vert_{C^{\alpha}(D)} \coloneqq \sup_{\bx,\by\in D, \bx\neq \by} \frac{\vert v(\bx)-v(\by)\vert}{\vert \bx-\by\vert^{\alpha}},
\end{align*}
and by $C^{\alpha}(D)$ we denote the Banach space of functions with finite Hölder-norm 
\begin{align*}
	\norm{v}{C^{\alpha}(D)} \coloneqq \norm{v}{C^0(D)} + \seminorm{v}{C^{\alpha}(D)}.
\end{align*}
Hölder spaces are closed under multiplication: If $u,v\in C^{\alpha}(D)$, then $uv\in C^{\alpha}(D)$.

Higher Hölder regularity of order $k\in\N$ is characterized by the seminorm 
\begin{align*}
	\vert v \vert_{C^{k+\alpha}(D)} \coloneqq \sum_{j=1}^k \vert D^j v\vert_{C^{\alpha}(D)}.
\end{align*}
The corresponding Banach space and norm read respectively
\begin{align*}
	C^{k+\alpha}(D) &\coloneqq \set{v:D\rightarrow \C}{ D^{j}v\in C^{\alpha}(D)\ \forall j =0,\dots,k},\\
	\norm{v}{C^{k+\alpha}(D)} &\coloneqq \sum_{j =0}^k \norm{D^jv}{C^{\alpha}(D)}.
\end{align*}
Again $u,v\in C^{k+\alpha}(D)$ implies $uv\in C^{k+\alpha}(D)$.

Recall that we employ the short notation to denote the Hölder norms, for example $\seminorm{\cdot}{\alpha} = \seminorm{\cdot}{C^{\alpha}(D)}$ and analogously for other norms.

%%%%%%%%%%%%%%%%%%%%%%%%%%%%%%%%%%%%%%%%%%%%%%% PROOF REGULARITY THEOREM %%%%%%%%%%%%%%%%%%%%%%%%%%%%%%%%%%%%%%%%%%%%%%% 
We now give a proof of Theorem~\ref{th:Holder_sample_paths}, which is inspired by~\cite{feischl2017existence}. 
The proofs in the mentioned work require higher temporal regularity than is available for stochastic LLG, which we circumvent by the use of H\"older spaces instead of Sobolev spaces.

To prove Hölder regularity of sample paths, we work with the following equivalent form of \eqref{eq:rnd_LLG_problem}, obtained using algebraic manipulations including the triple product expansion and the fact that $\seminorm{\m}{}=1$ for all $t\in[0,T]$ and a.e. $\bx\in D$:
\begin{align}\label{eq:rndLLG_alternative_form}
	\lambda\partial_t\m +\m\times\partial_t\m = \Delta \m + \vert\nabla \m\vert^2\m - \m\times\left(\m\times \hC(W, \m) \right),
\end{align}
where we recall that $\lambda>0$ is the Gilbert damping parameter and $\hC$ was defined in \eqref{eq:def_CHat}.

For the proof, we require some additional notation:
\begin{align*}
	H(\bu,\bv,\bw) &\coloneqq \bu\times(\bv\times\hC(W,\bw)) \quad \forall\bu,\bv\in C^{\alpha/2,\alpha}(D_T), \bw\in C^{\alpha/2,1+\alpha}(D_T), \\
	\RR_a(\bv) &\coloneqq  
		\lambda\partial_t \bv + \bv\times\partial_t\bv - \vert\bv\vert^2\Delta \bv -\vert\nabla \bv\vert^2\bv + H(\bv,\bv,\bv) 
		\quad \forall\bv\in C^{1+\alpha/2,2+\alpha}(D_T),\\
	L\bv &\coloneqq L_{\bx_0}\bv \coloneqq \lambda\bv + \bM^0(\b{x}_0) \times \bv \quad \forall\bx_0\in D, \bv\in C^{\alpha/2,\alpha}(D_T).
\end{align*}
We note that $\RR_a$ is the residual defined from the alternative form \eqref{eq:rndLLG_alternative_form} of the LLG equation; confer \eqref{eq:parametric_problem}.

We will require a couple of technical results.
\begin{lemma}[Continuity of the trilinear form $H$ and of the LLG residual $\RR_a$]\label{lemma:cont_residual_LLG}
	If $\bu$, $\bv\in C^{\alpha/2,\alpha}(D_T)$ and $\bw \in C^{\alpha/2,1+\alpha}(D_T)$, then $H(\bu,\bv,\bw)\in C^{\alpha/2,\alpha}(D_T)$ and
	\begin{align}\label{eq:estimate_trilinear_form}
		\norm{H(\bu, \bv, \bw)}{\alpha/2, \alpha} \leq C_{\bg} \norm{\bu}{\alpha/2,\alpha}\norm{\bv}{\alpha/2,\alpha} \norm{\bw}{\alpha/2,1+\alpha},
	\end{align}
	where $C_{\bg} \coloneqq \left(1+\norm{\bg}{1+\alpha}\right)^3 \left(\norm{\nabla\bg}{\alpha}+\norm{\Delta\bg}{\alpha}\right)$.
	Moreover, if $\bv \in C^{1+\alpha/2, 2+\alpha}(D_T)$, then \\ $\RR_a(\bv)\in C^{\alpha/2,\alpha}(D_T)$ and
	\begin{align}\label{eq:Rav}
		\normRes{\RR_a(\bv)} \leq
		\left(\seminormMag{\bv}+\seminormMag{\bv}^2\right)
		 \left(\lambda+\normMag{\bv}\right)^2
		 + C_{\bg} \normRes{\bv}^2 \norm{\bv}{\alpha/2, 1+\alpha}.
	\end{align}
	In particular, $\normRes{\RR_a(\bv)}$ vanishes when $\seminormMag{\bv}$ and $\normRes{\nabla \bg}+\normRes{\Delta \bg}$ both vanish.
	\end{lemma}
	\begin{proof}%[Proof of Lemma~\ref{lemma:cont_residual_LLG}]
	To prove \eqref{eq:estimate_trilinear_form}, note the following elementary estimates 
	\begin{align*}
	\normRes{\CC \bv} &\leq 2\normRes{\nabla \bv} \norm{\nabla \bg}{\alpha} + \normRes{\bv} \norm{\Delta\bg}{\alpha},
	\\
	\normRes{\CC G\bv} &\leq \normRes{\bv} \norm{\bg}{\alpha}\norm{\Delta\bg}{\alpha} + 2(\normRes{\nabla \bv} \norm{\bg}{\alpha} + \normRes{\bv}\norm{\nabla\bg}{\alpha}) \norm{\nabla\bg}{\alpha},
	\\
	\normRes{\mathcal{E}(s, \bv)} &\leq \normRes{\CC \bv} + \normRes{\CC G \bv} + \norm{\bg}{\alpha} \normRes{\CC \bv},
	\\
	\normRes{\hC(s, \bv)} &\leq \left(1 + \norm{\bg}{\alpha} + \norm{\bg}{\alpha}^2\right) \normRes{\mathcal{E}(s, \bv)},\\
	\normRes{H(\bu,\bv,\bw})
	&\leq 
	\normRes{\bu}
	\normRes{\bv}
	\normRes{\hC(W, \bw)}.
	\end{align*}
	Putting these facts together, one obtains \eqref{eq:estimate_trilinear_form}.
	To get the second inequality \eqref{eq:Rav}, estimate
	\begin{align*}
	\normRes{\RR_a(\bv)}
	&\leq \lambda\seminormMag{\bv}
	+ \normMag{\bv}\seminormMag{\bv}\\
	&\quad +\normMag{\bv}^2 \seminormMag{\bv}
	+ \seminormMag{\bv}^2 \normMag{\bv}
	+ \normRes{H(\bv, \bv, \bv)}{}\\
	&\leq \left(\seminormMag{\bv}+\seminormMag{\bv}^2\right)
	 \left(\lambda+\normMag{\bv}\right)^2
	+ \normRes{H(\bv, \bv, \bv)}{}.
	\end{align*}
	Using \eqref{eq:estimate_trilinear_form} to estimate the last term yields \eqref{eq:Rav}. 
\end{proof}
Additionally, we need some finer control over the boundedness of $\RR_a$. The point of the following result is that all terms apart from the first one on the right-hand side of the estimate in Lemma \ref{lemma:R_difference} below are either at least quadratic in $\bw$ or can be made small by choosing $\bv$ close to a constant function. This will allow us to treat the nonlinear parts as perturbations of the heat equation.
\begin{lemma}\label{lemma:R_difference}
	For $\bv, \bw\in C^{1+\alpha/2, 2+\alpha}(D_T)$ and $\bx_0\in D$, there holds
	\begin{align*}
		\normRes{\RR_a(\bv-\bw)} &\leq  
		\normRes{\RR_a(\bv) - (L\partial_t-\Delta)\bw}
		+\normRes{\bv-\bM^0(\b{x}_0)} \normMag{\bw}\\
		&+\normRes{\left(1-\vert\bv\vert^2\right) \Delta \bw}\\
		&+\normMag{\bw} \left(\seminormMag{\bv}+ C_{\bg}\right) \left(1+\normMag{\bv}\right)^2\\
		&+\normMag{\bw}^2\left(1 + (1+C_{\bg})\normMag{\bv}\right)
		+\normMag{\bw}^3 (1+C_{\bg}),
	\end{align*}
	where $C_{\bg}>0$ is defined in Lemma \ref{lemma:cont_residual_LLG}.
\end{lemma}
\begin{proof}
	All but the last term in the definition of $\RR_a$ are estimated as in \cite{feischl2017existence}. As for the last term, observe that 
	\begin{align*}
		H(\bv-\bw,\bv-\bw,\bv-\bw) &= H(\bv, \bv, \bv)-H(\bw, \bw, \bw)\\
		 &- H(\bw, \bv, \bv) - H(\bv, \bw, \bv) - H(\bv, \bv, \bw)\\
		 &+ H(\bv, \bw, \bw) + H(\bw, \bv, \bw) + H(\bw, \bw, \bv).
	\end{align*}
	The term $H(\bv, \bv, \bv)$ is absorbed in $\RR_a(\bv)$. Then, by the previous lemma:
	\begin{align*}
		\norm{- H(\bw, \bv, \bv) - H(\bv, \bw, \bv) - H(\bv, \bv, \bw)}{\alpha/2, \alpha} &\lesssim 
			C_{\bg} \norm{\bw}{\alpha/2, 1+\alpha} \norm{\bv}{\alpha/2, 1+\alpha}^2,\\
		\norm{H(\bv, \bw, \bw) + H(\bw, \bv, \bw) + H(\bw, \bw, \bv)}{\alpha/2, \alpha} &\lesssim 
			C_{\bg} \norm{\bw}{\alpha/2, 1+\alpha}^2 \norm{\bv}{\alpha/2, 1+\alpha},\\
		\norm{-H(\bw, \bw, \bw)}{\alpha/2, \alpha} &\lesssim 
			C_{\bg} \norm{\bw}{\alpha/2, 1+\alpha}^3.
	\end{align*}
	Altogether, we obtain the required result.
\end{proof}

To prove Theorem~\ref{th:Holder_sample_paths}, we use a fixed point iteration. 
\begin{proof}[Proof of Theorem~\ref{th:Holder_sample_paths}]
Consider the initial guess $\m_0(t,\b{x}) = \bM^0(\b{x})$ for all $t\in[0,T]$, $\bx\in D$, and fix one $\bx_0\in D$ (for the definition of $L=L_{\bx_0}$).
Define the sequence $\left(\m_{\ell}\right)_{\ell}$ as follows:
For $\ell = 0,1,\dots$
\begin{enumerate}
	\item Define $\br_{\ell} := \RR_a(\m_{\ell})$
	\item Solve
	\begin{align*}
	\begin{cases}
	L \partial_t \bR_{\ell} - \Delta \bR_{\ell} &= \br_{\ell}\qquad \textrm{in } D_T,\\
	\partial_n  \bR_{\ell} &= 0\qquad \textrm{on } [0,T]\times \partial D,\\
	\bR_{\ell}(0) &= 0\qquad \textrm{on } D.
	\end{cases}
	\end{align*}
	\item Update $\m_{{\ell}+1} := \m_{\ell} - \bR_{\ell}.$
\end{enumerate}
\emph{Step 1 (Well-posedness):}
By definition, we have $\m_0\in C^{1+\alpha/2, 2+\alpha}(D_T)$ as well as $\partial_n \m_0=0$. 
Assume that $\m_{\ell}\in C^{1+\alpha/2, 2+\alpha}(D_T)$ and $\partial_n \m_{\ell} = 0$. Then, Lemma~\ref{lemma:cont_residual_LLG} implies that $\br_\ell\in C^{\alpha/2, \alpha}(D_T)$. The parabolic regularity result~\cite[Theorem 10.4,\S 10, VII]{ladyzhenskaya1968linear} yields $\bR_{\ell}\in C^{1+\alpha/2, 2+\alpha}(D_T)$.

\emph{Step~2 (Convergence):}
We show the Cauchy property of the sequence $(\m_\ell)_\ell$: Fix $0\leq \ell' < \ell<\infty$ and observe that 
$ \normMag{\m_{\ell} - \m_{\ell'}} \leq \sum_{j = \ell'}^{\ell-1} \normMag{\bR_j}.$
By the previous lemmata, we have
\begin{align}\label{eq:initial_estim_R_jp1}
\normMag{\bR_{j+1}} \leq C_s \normRes{\br_{j+1}} = C_s \normRes{\RR_a(\m_{j+1})} = C_s \normRes{\RR_a(\m_{j} - \bR_{j})},
\end{align}
where $C_s>0$ is the stability constant from~\cite[Theorem 10.4,\S 10, VII]{ladyzhenskaya1968linear}, which only depends on $D_T$ and $L$ (particularly, it is independent of $\ell$).
We invoke Lemma \ref{lemma:R_difference} with $\bv = \m_j$ and $\bw=\bR_j$. By construction, $\RR_a(\m_j) - \left(L\partial_t - \Delta\right) \bR_j = 0$. What remains is estimated as
\begin{align}\label{eq:boundRj}
\begin{split}
\normRes{\RR_a(\m_j-\bR_j)} 
&\leq \normRes{\m_j-\bM^0(\bx_0)} \normMag{\bR_j}
+\normRes{\left(1-\vert\m_j\vert^2\right) \Delta\bR_j} \\
&+\normMag{\bR_j}\left(\seminormMag{\m_j}+C_{\bg}\right) \left(1+\normMag{\m_j}\right)^2\\
&+\normMag{\bR_j}^2\left(1 + (1+C_{\bg})\seminormMag{\m_j}\right)\\
&+\normMag{\bR_j}^3 (1+C_{\bg}).
\end{split}
\end{align}
Let us estimate the first term in \eqref{eq:boundRj}. For any $(t,\bx)\in D_T$, the fundamental theorem of calculus yields
$\vert\m_j(\bx,t)-\bM^0(\bx_0) \vert 
\lesssim \norm{(\partial_t, \nabla) \m_j}{C^0(D_T)} 
\leq \seminormMag{\m_j}$. Analogously, we get
\begin{align*}
\normRes{\m_j-\bM^0(\bx_0)}
%&= \norm{\m_j(x,t)-\m^0(x_0)}{C^0(D_T)} + \seminorm{\m_j(x,t)-\m^0(x_0)}{\gamma,k-2+2\gamma}\\
%& \lesssim \seminorm{\m_j}{1+\gamma,k+2\gamma} + \seminorm{\m_j}{\gamma,k-2+2\gamma}\\
&\leq 2 \seminormMag{\m_j}.
\end{align*}
Let us estimate the second term in \eqref{eq:boundRj}. Since $\m_j = \m_0 + \sum_{i=0}^{j-1} \bR_i$ and $\seminorm{\m_0}{}=1$ a.e., we have 
$\vert\m_j\vert^2 = 1 + 2\m_0\cdot \sum_{i=0}^{j-1} \bR_i + \left(\sum_{i=0}^{j-1} \bR_i\right)^2.$
Thus, the fact that Hölder spaces are closed under multiplication and the triangle inequality imply
\begin{align*}
\normRes{1-\vert\m_j\vert^2}
\leq 2\normRes{\m_0} \normRes{\sum_{i=0}^{j-1} \bR_i} + \left(\sum_{i=0}^{j-1} \normRes{\bR_i}\right)^2.
\end{align*}
All in all, we obtain 
\begin{align}\label{eq:contraction}
\normMag{\bR_{j+1}} &\leq \tilde{C} Q_j \normMag{\bR_{j}},
\end{align}
where $\tilde{C}>0$ is independent of $j$ and 
\begin{align*}
Q_j&\coloneqq 
\seminormMag{\m_j}
+ \normRes{\m_0}\normRes{\sum_{i=0}^{j-1} \bR_i} 
+ \left(\sum_{i=0}^{j-1} \normRes{\bR_i}\right)^2 \\
& + \left(\seminormMag{\m_j}+C_{\bg}\right) \left(1+\normMag{\m_j}\right)^2\\
& + \normMag{\bR_j} (1+(1+C_{\bg})\seminormMag{\m_j})
+ \normMag{\bR_j}^2(1+C_{\bg}).
\end{align*}
It can be proved that for any $q\in(0,1)$ there exists $\eps>0$ such that $\tilde{C} Q_j < q$ for all $j\in\N$. 
One proceeds by induction, as done in~\cite{feischl2017existence}, using additionally the assumption on the smallness of $\nabla \bg$ and $\Delta \bg$.
Therefore, $\normMag{\bR_{j+1}} \leq q \normMag{\bR_{j}}$, which implies that $(\m_\ell)_\ell$ is a Cauchy sequence in $C^{1+\alpha/2,2+\alpha}(D_T)$.
Hence, we find a limit $\m\in C^{1+\alpha/2,2+\alpha}(D_T)$ and the arguments above already imply the estimate in Theorem~\ref{th:Holder_sample_paths}.

\emph{Step~3 ($\m$ solves~\eqref{eq:rndLLG_alternative_form}):} $\m$ fulfills the initial condition $\m(0) = \bM^0$ (and thus $|\m(0)|=1$) and boundary condition $\partial_n \m = 0$ on $[0,T]\times \partial D$ by the continuity of the trace operator. The continuity of $\RR_a$ and the contraction \eqref{eq:contraction} imply
\begin{align*}
\normRes{\RR_a(\m)}
=\lim_{\ell} \normRes{\RR_a(\m_{\ell})} 
\lesssim \lim_{\ell} \normMag{\bR_{\ell}} 
\leq \lim_{\ell} q^{\ell} \normMag{\bR_0} = 0.
\end{align*}
The arguments of the proof of~\cite[Lemma~4.8]{feischl2017existence} show that $\RR_a(\m)=0$ implies that $\m$ solves~\eqref{eq:rndLLG_alternative_form} and hence concludes the proof.
\end{proof}

\section{Proofs of lemmata for sparse grid interpolation convergence}\label{appendix:SG_convergence}
In this appendix, we prove technical results needed in the analysis of the sparse grid interpolation in Section~\ref{sec:SG}. We refer to the same section for the definition of the objects and notation used here.

The following result is a standard interpolation error estimate on weighted spaces which, in this precise form, we could not find in the literature.
\revision{See Section~\ref{sec:1D_interpol} for the definition of objects such as the level-to-knot function $m(\nu)$, the knots family $\YY_\nu$ and the piecewise polynomial interpolant  $I_\nu[\cdot]$.}
\begin{lemma}\label{lemma:conv_1D_interpol_R}
	Consider $u:\R\rightarrow \R$ with $\partial u\in L^2_{\tmu}(\R)$. Then,
	\begin{align*}
	\norm{u- I_0 [u]}{L^2_{\mu}(\R)} \leq \tilde{C}_1 \norm{\partial u}{L^2_{\tmu}(\R)},
	\end{align*}
	where $\tilde{C}_1 = \sqrt{\int_{\R} \seminorm{y}{} \tmu^{-1}(y) \rmd \mu(y) }$. 
	If additionally, $\partial^p u\in L^2_{\tmu}(\R)$ for $p\geq 2$, then
	\begin{align*}
		\norm{u- I_\nu[u]}{L^2_{\mu}(\R)} 
		\leq \tilde{C}_2 (m(\nu)+1)^{-p} \frac{\norm{\partial^p u}{L^2_{\tmu}(\R)}}{p!}\qquad\forall\nu\geq 1,
	\end{align*}
	where $\tilde{C}_2 = \sqrt{C_{\phi}\frac{p}{2} \left(m(\nu)-1+2^{2p+1}\right) }$ and $C_{\phi}$ was defined in \eqref{eq:def_Cphi}.
\end{lemma}
\begin{proof}
	For the first estimate, the fundamental theorem of calculus and Cauchy-Schwarz inequality yield
	$ u(y)-u(0) = \int_0^y \partial u \leq \norm{\partial u}{L^2_{\tmu}(\R)} \sqrt{\int_0^y \tmu^{-1}}$.
	Substitute this in $\norm{u- I_0[u]}{L^2_{\mu}(\R)} $ to obtain the first estimate.
	
	For the second estimate, let $i\in\left\{\frac{m(\nu)+1}{2},\dots,m(\nu)-2\right\}$. 
	Apply the fundamental theorem of calculus $p$ times and recall that ${I_\nu[u]}|_{[y_i, y_{i+1}]}\in P_{p-1}([y_i, y_{i+1}])$ to obtain: 
	\revision{
	\begin{align}\label{eq:tmp_1D_conv}
		\left(u-I_\nu[u]\right)(y) 
		= \int_{y_i}^y \int_{\xi_1}^{z_1}\dots \int_{\xi_{p-1}}^{z_{p-1}} \partial^p u
		\qquad \forall y\in [y_i, y_{i+1}],
	\end{align}
	where $\xi_j\in[y_i, y_{i+1}]$ is such that $\partial^j(u-I_\nu[u])(\xi_j)=0$ for any $j=1,\dots,p-1$.
	Let us now focus on the last integral. The Cauchy-Schwarz inequality gives
	\begin{align*}
		\int_{\xi_{p-1}}^{z_{p-1}} \partial^p u 
		\leq \norm{\partial^p u}{L^2_{\tmu}(\xi_{p-1}, z_{p-1})} \sqrt{ \int_{\xi_{p-1}}^{z_{p-1}} \tmu^{-1} }.
	\end{align*} 
	The monotonicity of the integral with respect to the integration domain and the fact that 
	$\tmu^{-1}$ is monotonically increasing on the positive semi-axis give
	\begin{align*}
		\int_{\xi_{p-1}}^{z_{p-1}} \partial^p u 
		\leq \norm{\partial^p u}{L^2_{\tmu}(y_i, y_{i+1})} \tmu^{-1}(y) \sqrt{z_{p-1}-y_i}.
	\end{align*}
	Applying this to~\eqref{eq:tmp_1D_conv} and after integration, we obtain
	\begin{align*}
		\left(u-I_\nu[u]\right)(y) 
		\leq \norm{\partial^p u}{L^2_{\tmu}(y_i, y_{i+1})} \tmu^{-1/2}(y) 
			\frac{\seminorm{y-y_i}{}^{p-1+\frac{1}{2}}}{(p-1)!}.
	\end{align*}
	}
	Thus,
	\begin{align*}
		\int_{y_i}^{y_{i+1}} \seminorm{u(y)-I_\nu[u](y)}{}^2 \rmd \mu(y)
		\leq \left(\frac{\norm{\partial^p u}{L^2_{\tmu}(y_i, y_{i+1})}}{(p-1)!}\right)^2 \int_{y_i}^{y_{i+1}} \seminorm{y-y_i}{}^{2p-1} \tmu^{-1}(y) \rmd \mu(y).
	\end{align*}
	In order to estimate the last integral, change variables using $\phi$ defined in \eqref{eq:def_nodes_transform}. We get
	\begin{align*}
		\int_{y_i}^{y_{i+1}} \seminorm{y-y_i}{}^{2p-1} \tmu^{-1}(y) \rmd \mu(y)
		\leq \int_{x_i}^{x_{i+1}} \seminorm{\phi(x)-\phi(x_i)}{}^{2p-1} \tmu^{-1}(\phi(x)) \mu(\phi(x)) \phi'(x) \rmd x.
	\end{align*}
	A Taylor expansion and the fact that $\phi'$ is increasing give $\phi(x)-\phi(x_i) \leq \phi'(x) (x-x_i)$. Thus,
	\begin{align*}
		\int_{y_i}^{y_{i+1}} \seminorm{y-y_i}{}^{2p-1} \tmu^{-1}(y) \rmd \mu(y)
		\leq \int_{x_i}^{x_{i+1}} (x-x_i)^{2p-1} \left(\phi'(x)\right)^{2p} \tmu^{-1}(\phi(x)) \mu(\phi(x)) \rmd x.
	\end{align*}
	Recall now that $\left(\phi'(x)\right)^{2p} \tmu^{-1}(\phi(x)) \mu(\phi(x)) \equiv C_{\phi}$ defined in~\eqref{eq:def_Cphi}. Integration yields
	\begin{align*}
		\int_{y_i}^{y_{i+1}} \seminorm{y-y_i}{}^{2p-1} \tmu^{-1}(y) \rmd \mu(y)
		\leq \frac{(m+1)^{-2p}}{2p} C_{\phi}.
	\end{align*}
	For the original quantity, we get
	\begin{align*}
		\int_{y_i}^{y_{i+1}} \seminorm{u-I_\nu[u]}{}^2(y) \rmd \mu(y)
		\leq C_{\phi}\frac{p^2}{2p}  (m+1)^{-2p} \left(\frac{\norm{\partial^p u}{L^2_{\tmu}(y_i, y_{i+1})}}{p!}\right)^2.
	\end{align*}
	For $i=m(\nu)-1, m(\nu)$, recall that $I_\nu$ is defined in $[y_{m(\nu)},+\infty)$ as the polynomial extension from the previous interval. 
	Analogous estimates give
	\begin{align*}
		\int_{y_{m(\nu)-1}}^{+\infty} \seminorm{u-I_{\nu}[u]}{}^2(y) \rmd \mu(y)
		\leq C_{\phi}\frac{p^2}{2p} 2^{2p} (m(\nu)+1)^{-2p} 
			\left(\frac{\norm{\partial^p u}{L^2_{\tmu}(y_i, y_{i+1})}}{p!}\right)^2.
	\end{align*}
	Finally, together with analogous estimates for $i\leq \frac{m(\nu)+1}{2}$, this gives the second estimate in the statement.
\end{proof}

\begin{lemma}\label{lemma:estimate_details}
	Consider $u:\R \rightarrow \R$, a continuous function with $\partial u\in L^2_{\tmu}(\R)$ and $p\geq 2$. There holds
	\begin{align*}
		\norm{\Delta_1[u]}{L^2_{\mu}(\R)} \leq C_1 \norm{\partial u}{L^2_{\tmu}(\R)},
	\end{align*}
	where 
	$C_1 = 2^{3/2} \tilde{C}_1  \sqrt{\int_0^{\infty} \sum_{j=1}^p \seminorm{l_j'}{}^2\rmd \tmu}  \sqrt[4]{\int_0^{y_3} \tmu^{-1}}$, 
	$\tilde{C}_1>0$ was defined in the previous lemma, $y_1,y_2, y_3$ delimit the intervals of definition of the piecewise polynomial $I_1[u]$ and $(l_j)_{j=1}^p$ is the Lagrange basis of $P_{p-1}([y_2, y_3])$ with respect to $y_2, y_3$ and other $p-2$ district points in $(y_2, y_3)$.\\
	If additionally $\partial^p u \in L^2_{\tmu}(\R)$, then we have 
	\begin{align*}
		\norm{\Delta_{\nu}[u]}{L^2_{\mu}(\R)} 
		\leq C_2 2^{-p\nu} \frac{\norm{\partial^p u}{L^2_{\tmu}(\R)}}{p!}
		\qquad \forall\nu > 1,
	\end{align*}
	where $C_2 = \tilde{C}_2 (1+2^{-p})$ and $\tilde{C}_2>0$ was defined in the previous lemma.
\end{lemma}
\begin{proof}
	To prove the first estimate, recall that nodes are nested so 
	$I_0[u] = I_0 \left[ I_1^p[u]\right]$. Thus,
	\begin{align*}
		\Delta_1[u] 
		= I_1[u] - I_0[u] 
		= I_1[u] - I_0\left[I_1[u]\right] 
		= (1-I_0)\left[I_1[u]\right].
	\end{align*}
	The previous lemma gives
	\begin{align*}
		\norm{\Delta_1[u]}{L^2{\mu}(\R)} 
		\leq \tilde{C}_1 \norm{\partial I_1[u]}{L^2_{\tmu}(\R)}.
	\end{align*}
	To estimate the last integral, consider $x_1=y_2<x_2<\dots< x_p=y_3$ the interpolation nodes in the interval $[y_2, y_3]$.
	Observe that $\partial I_1[u] = \partial I_1[u-u(0)]$ and estimate
	\begin{align*}
		\int_0^{\infty} \seminorm{\partial I_1[u]}{}^2\rmd \tmu
		&= \int_0^{\infty} \seminorm{\partial I_1[u-u(0)]}{}^2\rmd \tmu
		= \int_0^{\infty} \seminorm{\sum_{j=1}^p (u(x_j)-u(0)) l_j'}{}^2\rmd \tmu\\
		&\leq 2\max_{j=1,\dots,n} \seminorm{u(x_j)-u(0)}{}^2 \int_0^{\infty} \sum_{j=1}^p \seminorm{l_j'}{}^2\rmd \tmu
	\end{align*}
	The second term is bounded for fixed $p$. As for the first term, simple computations give
	\begin{align*}
		\max_{j=1,\dots,n} \seminorm{u(x_j)-u(0)}{}
		\leq\int_0^{y_3}  \seminorm{\partial u}{}
		\leq \norm{\partial u }{L^2_{\tmu}(0, y_3)}\sqrt{\int_0^{y_3} \tmu^{-1}}.
	\end{align*}
	This, together with analogous computations on $(-\infty, 0]$, gives the first estimate.
	
	To prove the second estimate, observe that
	\begin{align*}
		\norm{\Delta_{\nu}[u]}{L^2_{\mu}(\R)}
		=\norm{I_\nu[u] - I_{\nu-1}[u]}{L^2_{\mu}(\R)}
		\leq \norm{u- I_\nu[u] }{L^2_{\mu}(\R)} + \norm{u- I_{\nu-1}[u] }{L^2_{\mu}(\R)}
	\end{align*}
	The previous lemma and simple computations imply the second estimate.
\end{proof}

We now give additional details about the sparse grid interpolation result from Section~\ref{sec:midset_selection_theoretical}.
\begin{proof}[Detailed proof of Theorem~\ref{th:conv_SG_C}]
	The sum in the proof of Theorem~\ref{th:conv_SG_C}  is finite if $\tau \geq \frac{1}{p+1}$ and in this case it equals
\begin{align}\label{eq:sigma_conv_test_pwPolySG}
	\sum_{\nu_i\geq 2} \left(C_2 (2^{\nu_i}\rho_i)^{-p}\right)^{\tau} \left(p2^{\nu_i}\right)^{1-\tau}
	= C_2^{\tau} \rho_i^{-p\tau} p^{1-\tau} \sigma(p,\tau),
\end{align}
where $\sigma(p,\tau)$ is defined in the statement.
Therefore,
\begin{align*}
	\sum_{\bnu\in\N_0^N}\PP_{\bnu}^{\tau} w_{\bnu} \leq
	\prod_{i=1}^N \left( 1 + \left(C_1 \rho_i^{-1}\right)^{\tau} \left(2p\right)^{1-\tau}+
	C_2^{\tau} \rho_i^{-p\tau} p^{1-\tau} \sigma(p,\tau)
	\right).
\end{align*}
Recall that $\rho_i = 2^{(1-\alpha)\lceil\log_2(i)\rceil/2}$ for all $i\in\N$ as in \eqref{eq:def_rho_n}. 
Denote $\ell(i)\coloneqq \lceil\log_2(i)\rceil$,
$\rhol \coloneqq 2^{(1-\alpha)\ell/2}$ for all $i,\ell\in\N$. 
Thus,
\begin{align*}
	\sum_{\bnu\in\N_0^N}\PP_{\bnu}^{\tau} w_{\bnu} \leq
	(1+F_0) \prod_{\ell\geq 1} (1+F_{\ell})^{2^{\ell-1}},
\end{align*}
where
\begin{align*}
	F_{\ell} \coloneqq 
	\left(C_1 \rhol^{-1}\right)^{\tau} \left(2p\right)^{1-\tau}
	+ C_2^{\tau} \rhol^{-p\tau} p^{1-\tau} \sigma(p,\tau)\qquad \forall\ell\geq 0.
\end{align*}
We estimate
\begin{align*}
	\prod_{\ell\geq 1} (1+F_{\ell})^{2^{\ell-1}}
	\leq \exp\left( \sum_{\ell\geq1} 2^{\ell-1} \log \left(1+F_{\ell}\right)\right)
	\leq \exp\left( \sum_{\ell\geq1} 2^{\ell-1} F_{\ell} \right).
\end{align*}
From the definition of $F_{\ell}$,
\begin{align*}
	\sum_{\ell\geq1} 2^{\ell-1} F_{\ell}
	= \frac{C_1^{\tau} \left(2p\right)^{1-\tau}}{2} \sum_{\ell\geq1} 2^{\ell} \rhol^{-\tau} +  
	\frac{C_2^{\tau} \sigma(p,\tau) p^{1-\tau}}{2} \sum_{\ell\geq1} 2^{\ell}  \rhol^{-p\tau}.
\end{align*}
The second sum in the right-hand side is finite if ${(1-\alpha)}\tau>\frac{2}{p}$ (implies the condition $\tau > \frac{1}{p+1}$ found above) and in this case it equals:
\begin{align*}
	\sum_{\ell\geq1} 2^{\ell} \rhol^{-p\tau} = \sum_{\ell\geq1} 2^{\ell} 2^{-(1-\alpha)p\tau\ell/2} = \frac{1}{1-2^{1-{(1-\alpha)}p\tau/2}}.
\end{align*}
Conversely, the first sum diverges as $0<\tau,\alpha<1$ and the sum of the first $L\in\N$ terms equals
\begin{align*}
	\sum_{\ell=0}^L 2^{\ell} \rhol^{-\tau} = \frac{1-2^{(1-(1-\alpha)\tau/2)L}}{1-2^{1-{(1-\alpha)}\tau/2}} = \OO\left(2^{(1-(1-\alpha)\tau/2)L}\right).
\end{align*}
Finally, we recall that the number of parametric dimensions is $N=2^L$ (see Section \ref{sec:from_SPDE_to_paramPDE_general}) to obtain the statement.
\end{proof}
We now prove the lemmata we stated in Section~\ref{sec:improved_profit}.
These results, while technical, are instrumental in the  proof of dimension-independent convergence of sparse grid interpolation with improved profits.
\begin{proof}[Proof of Lemma~\ref{lemma:summability_in_F_01}]
Choose $\eps>0$ such that $\frac{1}{1+\eps}\geq p$ and $q> p(1+\eps)$. We consider $\ba\in\R^{\N}$ such that $\alpha > \seminorm{\ba}{1/(1+\eps)}$ and write
\begin{align*}
\sum_{\bnu\in\FFd} \left( \seminorm{\bnu}{1}!\ \bm{a}^{\bnu} \right)^q
= \sum_{\bnu\in\FFd} \left( \seminorm{\bnu}{1}!\ \alpha^{\seminorm{\bnu}{1}} \left(\frac{\bm{a}}{\alpha}\right)^{\bnu} \right)^q.
\end{align*}
There exists $C_{\eps}>0$ such that $ \alpha^{\seminorm{\bnu}{1}} \leq C_{\eps} \left(\seminorm{\bnu}{1}!\right)^{\eps}$ for all $\bnu\in \FFd$. Thus,
\begin{align*}
\sum_{\bnu\in\FFd} \left( \seminorm{\bnu}{1}!\ \bm{a}^{\bnu} \right)^q
\lesssim
\sum_{\bnu\in\FFd} \left( \left(\seminorm{\bnu}{1}!\right)^{1+\eps} \left(\frac{\bm{a}}{\alpha}\right)^{\bnu} \right)^q.
\end{align*}
Factorizing out the $1+\eps$ yields
\begin{align*}
\sum_{\bnu\in\FFd} \left( \seminorm{\bnu}{1}!\ \bm{a}^{\bnu} \right)^q
\lesssim
\sum_{\bnu\in\FFd} \left( \seminorm{\bnu}{1}! \left(\frac{\bm{a}}{\alpha}\right)^{\frac{1}{1+\eps}\bnu} \right)^{(1+\eps)q}.
\end{align*}
Since $\bnu! =1$ for all $\bnu\in\FFd$, we can write
\begin{align}\label{eq:bound_proof_summability_F01}
\sum_{\bnu\in\FFd} \left( \seminorm{\bnu}{1}!\ \bm{a}^{\bnu} \right)^q
\lesssim
\sum_{\bnu\in\FFd} \left( \frac{\seminorm{\bnu}{1}!}{\bnu!} \left(\frac{\bm{a}}{\alpha}\right)^{\frac{1}{1+\eps}\bnu} \right)^{(1+\eps)q}.
\end{align}
Observe that $\sum_{j} (\frac{a_j}{\alpha})^{\frac{1}{1+\eps}} <1 $ because of the definition of $\alpha$.
Moreover, from the assumption on $\ba$ we have $\left(\frac{\ba}{\alpha}\right)^{\frac{1}{1+\eps}}\in \ell^{r}(\N)$ for any $r\geq p(1+\eps)$.
Then, \cite[Theorem 1]{Cohen2011Analytic} implies that the second sum in \eqref{eq:bound_proof_summability_F01} is finite, thus proving the statement.
\end{proof}

\begin{proof}[Proof of Lemma~\ref{lemma:finite_factor_product_1}]
For this proof, we denote the level of $i$ by $\ell(i)$.
First observe that, from the definitions of value and work, we may write
\begin{align*}
\prod_{i:\hat{\nu}_i\leq1 } \left(v_{\hat{\nu}_i}^{\tau}w_{\hat{\nu}_i}^{1-\tau}\right) =
\prod_{i:\hat{\nu}_i=1 } \left(C_1 2^{-\left(\frac{3}{2}-\delta\right)\ell(i)} r_{\ell(i)}(\bnu) \right)^{\tau} \left( 2p \right)^{1-\tau}.
\end{align*}
The factors in the right-hand side are independent of the components of $\bnu$ for which $\nu_i\neq 1$. Thus, we define
\begin{align*}
D_{\bnu} = \set{\bm{d}\in \FF}{\begin{cases}
d_i=0\qquad &\textrm{if } \nu_i>1\\
d_i\in\left\{0,1\right\} \qquad & \textrm{otherwise}
\end{cases}} \subset \FFd
\end{align*}
and substitute
\begin{align*}
\sum_{\hbnu\in K_{\bnu}} \prod_{i:\hat{\nu}_i=1 }
\left( C_1 2^{-\left(\frac{3}{2}-\delta\right)\ell(i)} r_{\ell(i)}(\hbnu) \right)^{\tau}
\left( 2p \right)^{1-\tau}
=
\sum_{\bm{d}\in D_{\bnu}} \prod_{i:d_i=1 }
\left(C_1 2^{-\left(\frac{3}{2}-\delta\right)\ell(i)} r_{\ell(i)}(\bm{d}) \right)^{\tau}
\left(2p \right)^{1-\tau}.
\end{align*}
From the definition of $r_{\ell(i)}(\bm{d})$, we  estimate 
$ \prod_{i:d_i=1 } r_{\ell(i)}(\bm{d}) \leq
\prod_{\ell:\exists j:d_{\ell,j}=1} r_{\ell}(\bm{d})^{r_{\ell} (\bm{d})}$.
Stirling's formula gives $r_{\ell}(\bm{d})^{r_{\ell} (\bm{d})} \leq r_{\ell}(\bm{d})! e^{r_{\ell}(\bm{d})} $. 
Denote $\b{d}_{\ell}= \left(d_{\ell,1},\dots, d_{\ell, \lceil 2^{\ell-1}\rceil}\right)$ for any $\ell\in\N_0$ and observe that $r_{\ell}(\bm{d}) \leq \seminorm{\bm{d}_{\ell}}{1}$. Together with an elementary property of the factorial, this gives
$
\prod_{\ell:\exists j:d_{\ell,j}=1} \left(r_{\ell}(\bm{d})\right)!
\leq \prod_{\ell:\exists j:d_{\ell,j}=1} \seminorm{\bm{d}_{\ell}}{1}!
\leq \left(\sum_{\ell\in\N}\seminorm{\bm{d}_{\ell}}{1}\right)!
= \seminorm{\bm{d}}{1}!.
$
To summarize, we have estimated
\begin{align*}
\sum_{\hbnu\in K_{\bnu}} \prod_{i:\hat{\nu}_i\leq1 } \left(v_{\hat{\nu}_i}^{\tau}w_{\hat{\nu}_i}^{1-\tau}\right)
\leq
\sum_{\bm{d}\in D_{\bnu}}  (\seminorm{\bm{d}}{1}!)^\tau \prod_{i:d_i=1 }
\left( C_1^{\tau} 2^{-\left(\frac{3}{2}-\delta\right)\ell(i) \tau} \left( 2p \right)^{1-\tau} e^\tau \right).
\end{align*}
Define $c_j \coloneqq C_1 2^{-\left(\frac{3}{2}-\delta\right)\ell(j)} \left( 2p \right)^{(1-\tau)/\tau} e$ for all $j\in\N$ to obtain
$\sum_{\hbnu\in K_{\bnu}} \prod_{i:\hat{\nu}_i\leq1 } \left(v_{\hat{\nu}_i}^{\tau}w_{\hat{\nu}_i}^{1-\tau}\right)
\leq 
\sum_{\bm{d}\in D_{\bnu}}  \Big(\seminorm{\bm{d}}{1}! \bm{c}^{\bm{d}}\Big)^{\tau}$.
Simple computations reveal that $\bm{c}=(c_j)_j\in \ell^\tau(\N)$ for all $\tau> (\frac{3}{2}-\delta)^{-1}$. We apply the previous lemma and conclude the proof.
\end{proof}
%%%%%%%%%%%%%%%%%%%%%%%%%% END APPENDIX %%%%%%%%%%%%%%%%%%%%%%%%%%
%%%%%%%%%%%%%%%%%%%%%%%%%% BIBLIOGRAPHY %%%%%%%%%%%%%%%%%%%%%%%%%%
% \bibliographystyle{siamplain} 
\bibliographystyle{alpha}
\bibliography{literature}
%%%%%%%%%%%%%%%%%%%%%%%%%% END BIBLIOGRAPHY %%%%%%%%%%%%%%%%%%%%%%%%%%
\end{document}